\documentclass[12pt]{article}

\usepackage[utf8]{inputenc}
\usepackage{placeins}
\usepackage{graphicx}
\usepackage{float}
\usepackage{ragged2e}
\usepackage{amsthm}
\usepackage{amsmath}
\usepackage{mathtools}
\usepackage{amsfonts}
\usepackage{amssymb}
\usepackage[normalem]{ulem}
\usepackage{bbm}
\usepackage{bm}
\usepackage{tikz}
\usepackage{algpseudocode}
\usepackage{breqn}
\usepackage{booktabs}
\usepackage{threeparttable}
\usepackage{natbib}
\usepackage{nicefrac}
\usepackage[multiple]{footmisc}
\usepackage{hyperref}
\usepackage{enumitem}
\setenumerate{wide}
\usepackage{thmtools}
\usepackage{thm-restate}
\usepackage[nameinlink, capitalize]{cleveref}

\usepackage{cprotect}
\usepackage{etoc}
\usepackage{minitoc}
\usepackage{etaremune}

\usepackage[T1]{fontenc}    
\usepackage{booktabs}       
\usepackage{microtype}      
\floatstyle{ruled}
\usepackage{xcolor}

\usepackage{caption}
\usepackage{comment}
\usepackage{subcaption}

\usetikzlibrary{arrows,automata, calc, positioning}
\usepackage{adjustbox}
\usepackage{wrapfig}
\DeclareCaptionLabelSeparator{custom}{---}

\usepackage{multirow}
\usepackage{accents}

\usepackage[subtle]{savetrees}
\usepackage{apptools}

\usepackage{setspace}
\def\spacingset#1{\renewcommand{\baselinestretch}%
{#1}\small\normalsize} \spacingset{1}

\DeclareRobustCommand{\bb}[1]{\mathbb{#1}}

\DeclareRobustCommand{\HS}{\mathrm{HS}}
\DeclareRobustCommand{\op}{\mathrm{op}}
\DeclareRobustCommand{\H}{\mathcal{H}}

\DeclareRobustCommand{\E}{\mathbb{E}}
\DeclareRobustCommand{\R}{\mathbb{R}}

\DeclareRobustCommand{\H}{\mathcal{H}}
\DeclareRobustCommand{\X}{\mathcal{X}}
\DeclareRobustCommand{\Z}{\mathcal{Z}}
\DeclareRobustCommand{\Xsub}{x}
\DeclareRobustCommand{\Zsub}{z}
\DeclareRobustCommand{\Sz}{S_\Zsub}
\DeclareRobustCommand{\Szi}{S_{\Zsub,i}}

\DeclareRobustCommand{\kx}{k_\Xsub}
\DeclareRobustCommand{\kz}{k_\Zsub}
\DeclareRobustCommand{\Hx}{\H_{\Xsub}}
\DeclareRobustCommand{\Hz}{\H_{\Zsub}}

\newcommand{\bml}{b_{\mu, \lambda}}
\newcommand{\vml}{v_{\mu, \lambda}}
\newcommand{\hml}{h_{\mu,\lambda}}
\newcommand{\Ocal}{\mathcal{O}}
\newcommand{\Tml}{T_{\mu,\lambda}}
\newcommand{\Szmu}{(\Sz+I\mu)^{-1}}
\newcommand{\hvec}{(h_0-h_{\mu,\lambda})}
\newcommand{\dSi}{(S_i-S)}
\newcommand{\dSzi}{(\Sz-\Szi)}
\newcommand{\hTml}{\hat{T}_{\mu,\lambda}^{-1}}
\newcommand{\nmu}{\mathfrak{n}_z(\mu)}

\DeclareRobustCommand{\bb}[1]{\mathbb{#1}} 
\DeclareRobustCommand{\bk}[2]{\left\langle #1,\,#2\right\rangle}

\DeclareRobustCommand{\norm}[1]{\left\lVert #1 \right\rVert}
\DeclareRobustCommand{\snorm}[1]{\lVert #1 \rVert}

\DeclareRobustCommand{\d}{\delta}

\DeclareRobustCommand{\ep}{\varepsilon}
\DeclareRobustCommand{\eps}{\epsilon}

\DeclareRobustCommand{\bk}[2]{\left\langle #1,\,#2\right\rangle}

\DeclareRobustCommand{\norm}[1]{\left\lVert #1 \right\rVert}
\DeclareRobustCommand{\snorm}[1]{\lVert #1 \rVert}

\DeclareMathOperator{\tr}{tr}

\DeclareMathOperator{\diag}{diag}

\DeclareMathOperator*{\argmin}{\arg\!\min}
\DeclareMathOperator*{\argmax}{\arg\!\max}

\newtheorem{lemma}{Lemma}
\newtheorem{theorem}{Theorem}

\newtheorem{corollary}{Corollary}
\newtheorem{assumption}{Assumption}
\newtheorem{proposition}{Proposition}

\AtAppendix{\counterwithin{lemma}{section}}
\AtAppendix{\counterwithin{theorem}{section}}
\AtAppendix{\counterwithin{algorithm}{section}}
\AtAppendix{\counterwithin{corollary}{section}}
\AtAppendix{\counterwithin{assumption}{section}}
\AtAppendix{\counterwithin{proposition}{section}}

\theoremstyle{definition}
\newtheorem{algorithm}{Algorithm}
\newtheorem{definition}{Definition}

\theoremstyle{remark}
\newtheorem{example}{Example}
\newtheorem{remark}{Remark}

\AtAppendix{\counterwithin{definition}{section}}
\AtAppendix{\counterwithin{example}{section}}
\AtAppendix{\counterwithin{remark}{section}}



\begin{document}

\doparttoc 
\faketableofcontents 

\def\spacingset#1{\renewcommand{\baselinestretch}%
{#1}\small\normalsize} \spacingset{1}


 \title{\bf Uniform inference for kernel instrumental variable regression}
  \author{
  Marvin Lob \\
  Seminar for Statistics, ETH Zurich \\
  \and
  Rahul Singh\\
    Society of Fellows and Department of Economics, Harvard University \\
    \and
    Suhas Vijaykumar\\
    Department of Economics, U.C.~San Diego
    }
    
     \date{Original draft: November 2025. This draft: July 2026.}
  \maketitle

\bigskip

\begin{abstract}
Instrumental variable regression is a foundational tool for causal analysis across the social and biomedical sciences. 
Recent advances use kernel methods to estimate nonparametric causal relationships, with general data types, while retaining a simple closed-form expression. 
Empirical researchers ultimately need reliable inference on causal estimates; however, uniform confidence sets for the method remain unavailable. 
To fill this gap, we develop valid and sharp confidence sets for kernel instrumental variable regression, allowing general nonlinearities and data types.
Computationally, our bootstrap procedure requires only a single run of the kernel instrumental variable regression estimator.
Theoretically, it relies on the same key assumptions.
Overall, we provide a practical procedure for inference that substantially increases the value of kernel methods for causal analysis.
%
\end{abstract}

\noindent%

\vfill

\newpage
\spacingset{1.9} 

\section{Introduction and related work}

Nonparametric instrumental variable regression is a leading framework for causal analysis from observational data \citep{newey2003instrumental,ai2003efficient,hall2005nonparametric,blundell2007semi,darolles2011nonparametric}. A recent literature advocates for kernel methods as natural extensions from linear models to nonlinear models \citep{singh2019kernel,dikkala2020minimax}. Similar to a linear method, a kernel method has a simple closed-form solution \citep{kimeldorf1971some}. Unlike a linear method, a kernel method allows for rich nonlinearity in the causal relationship as well as general data types, such as preferences, sequences, and graphs, which often arise in economics and epidemiology. 

This literature proposes nonparametric estimators and proves uniform consistency, yet uniform inference guarantees are unavailable; these causal estimators lack uniform confidence bands. Without confidence bands, social and biomedical scientists are reluctant to fully rely on these new estimators in causal analysis. More generally, uniform confidence bands appear to be absent from the recent, burgeoning literature on machine learning estimation of nonparametric instrumental variable regression. Our research question is how to construct them.

Our primary contribution is to develop a uniform confidence band for a kernel estimator of the nonparametric instrumental variable regression function. Our inferential procedure retains the practicality of kernel methods; computationally, it is a bootstrap that involves running kernel instrumental variable regression exactly once and sampling many anti-symmetric Gaussian multipliers. The anti-symmetry is effective at canceling out the complex bias of the estimator.

Our secondary contribution is to prove that the uniform confidence band is valid and sharp: it obtains coverage of at least, and not much more than, the nominal level. Formally, we derive nonasymptotic Gaussian and bootstrap couplings, which overcome the challenge of the ill-posed inverse problem inherent in nonparametric instrumental variable regression. Nonasymptotic analysis is necessary because a stable Gaussian limit does not exist.

We show that well-known assumptions \citep{smale2007learning,caponnetto2007optimal,mendelson2010regularization,fischer2020sobolev} imply not only estimation but also inference guarantees.
Our key assumptions are: (i) the data have a low effective dimension when expressed in the basis of the kernel; (ii) the true nonparametric instrumental variable regression function is smooth in terms of the basis of the kernel; and (iii) the expectation of the nonparametric instrumental variable regression function, conditional upon the instrumental variable, is a smooth function as well. These assumptions are called decay, source, and link conditions, respectively.

By studying inference, we complement several works on estimation and consistency of kernel methods for nonparametric instrumental variable regression. 
Previous work provides rates in mean square error after projection upon the instrument \citep{singh2019kernel,dikkala2020minimax}, in mean square error \citep{liao2020provably,bennett2023minimax,bennett2023source}, in $\sup$ norm \citep{singh2020negative}, and in interpolation norms \citep{meunier2024nonparametric}. 
Such rates can be used to verify conditions for inference on certain well-behaved functionals \citep{kallus2021causal,ghassami2021minimax,chernozhukov2021simple}. 
However, none of these works provide uniform confidence bands, which are the focus of the present work.

A recent paper provides uniform inference for kernel ridge regression \citep{singh2023kernel}. Kernel ridge regression is an easier estimation problem, which does not require ill-posed inversion of a conditional expectation operator. 
 Previous results for kernel ridge regression do not apply to our setting. 
Still, we build on the broad structure of their argument. Specifically, we analyze Gaussian couplings \citep{zaitsev1987estimates,buzun2022strong} and bootstrap couplings \citep{freedman1981bootstrapping,chernozhukov2014anticoncentration,chernozhukov2016empirical} in settings where the limit distribution may be degenerate \citep{andrews2013inference}. In doing so, we develop new techniques that may be used to provide uniform statistical inference in other ill-posed inverse problems.

Finally, our inferential procedure for a kernel instrumental variable estimator complements existing results for series estimators of nonparametric instrumental variable regression \citep{chen2007large,carrasco2007linear}. The series procedures are designed for and theoretically justified in a setting with low- to moderate-dimensional Euclidean data \citep{belloni2015some,chen2018optimal,chen2024adaptive}. By contrast, our kernel procedure permits complex and nonstandard data, as long as the data have a low effective dimension relative to the kernel.

Section~\ref{sec:model} recaps the kernel instrumental variable regression estimator and interprets our main assumptions: low effective dimension and high smoothness. 
Section~\ref{sec:algo_main} presents our main contribution: valid and sharp confidence sets for a kernel instrumental variable regression estimator.
Section~\ref{sec:conclusion} concludes by discussing consequences for the uptake of kernel methods in causal analysis.
\section{Model and assumptions}\label{sec:model}

We begin by introducing some notation. 
We denote the $L^2$ norm by $
\|f\|_2 = \left( \mathbb{E}[f(Z)^2] \right)^{1/2},
$
with empirical counterpart
$
\|f\|_{2,n} = \left( \frac{1}{n} \sum_{i=1}^n f(Z_i)^2 \right)^{1/2}.
$
 For any operator $A$ mapping between Hilbert spaces, i.e. $A: \mathcal{H} \to\mathcal{H}'$, let $\|A\|_{\HS}$ and $\|A\|_{\op}$ denote the Hilbert-Schmidt and operator norm, respectively. We denote the eigendecomposition of a compact and self-adjoint operator $A$ with eigenvalues $\{\nu_1(A), \nu_2(A), \dots\}$ and eigenfunctions $\{e_1(A), e_2(A), \dots\}$. 

We use $C$, potentially with subscripts, to denote a positive constant that may only depend on the subscript parameter. For example, $C_\sigma$ is a positive constant depending only on the parameter $\sigma$. We use $\lesssim$ (or $\lesssim_{C_\sigma})$ to denote an inequality that holds up to a positive multiplicative constant (or function of $\sigma$). Equations and inequalities containing the parameter $\eta$ are understood to hold with probability at least $1-\eta$, where $\eta \in (0,1)$. 
\subsection{Previous work: Closed form estimation of KIV}

Our goal is to learn and conduct inference on the nonparametric instrumental variable regression function $h_0$, which is defined as the solution to the following operator equation:
\begin{equation}\label{eq:moment}
   Y=h_0(X)+\ep,\quad \E(\ep|Z)=0 \quad \iff \quad  \E(Y|Z)=\E\{h_0(X)|Z\}.
\end{equation}
We refer to $Y$ as the outcome, $X$ as the covariate, and $Z$ as the instrument. The former formulation is clearly a generalized regression problem; when $X=Z$, $h_0(X)=\E(Y|X)$. However, in our setting, $X\neq Z$. The latter formulation has the ``reduced form'' function $\E(Y|Z)$ on the left-hand side, and the composition of a ``first stage'' conditional expectation operator $\E\{...|Z\}$ and a ``second stage'' function of interest $h_0(X)$ on the right-hand side. Isolating $h_0$ involves inverting the conditional expectation operator, which is an ill-posed task that makes this statistical problem challenging.

To estimate $h_0$, we use a pair of approximating function spaces $\Hx$ and $\Hz$ that possess a specific structure: both are reproducing kernel Hilbert spaces (RKHSs) defined by the kernels $\kx: \X \times \X \to \mathbb{R}$ and  $\kz: \Z \times \Z \to \mathbb{R}$, respectively. 
We denote the associated feature maps by $\psi : \X \rightarrow \Hx$ and $\phi : \Z \rightarrow \Hz$, giving rise to the inner products $\kx(x,x') = \bk{\psi(x)}{{\psi(x')}}_{\Hx}$ and $\kz(z,z') = \bk{\phi(z)}{{\phi(z')}}_{\Hz}$. 
These spaces posses the reproducing property: $h(x) = {\bk{h}{\psi(x)}}_{\Hx}$ for any $h\in \Hx$, and $f(z) = {\bk{f}{\phi(z)}}_{\Hz}$ for any $f\in \Hz$.
In other words, $\psi(x)$ is the dictionary of basis functions for $\Hx$, and likewise $\phi(z)$ for $\Hz$.
The implied norms $\|h\|_{\Hx} = {\bk{h}{h}}_{\Hx}^{1/2}$ and $\|f\|_{\Hz} = {\bk{f}{f}}_{\Hz}^{1/2}$ quantify regularity. Concretely, the RKHS norm quantifies not only magnitude but also smoothness, generalizing the Sobolev norm.

Throughout the paper, we maintain a few regularity conditions to simplify the exposition. We assume that $\psi$ and $\phi$ are measurable and bounded, i.e., $\sup_{x \in \mathcal{X}} \|\psi(x)\|_{\Hx} \leq \kappa_x$ and $\sup_{z \in \mathcal{Z}} \|\phi(z)\|_{\Hz} \leq \kappa_z$, which is satisfied by all kernels commonly used in practice. Additionally, we maintain that the residual $\ep := Y - h_0(X)$ satisfies $|\ep|\leq \bar \sigma$.

We now define covariance operators which are central to our analysis. Let the symbol $\otimes$ mean outer product.
The covariance operator $S_x=\E \{\psi(X)\otimes \psi(X)^*\}$ satisfies $\bk{u}{S_x v}_{\Hx }= \mathbb{E}[u(X)v(X)]$.
The covariance operator $S_z=\E \{\phi(Z)\otimes \phi(Z)^*\}$ satisfies  $\bk{u}{S_z v}_{\Hz }= \mathbb{E}[u(Z)v(Z)]$. Finally, we define the cross covariance operator $S=\E\{\phi(Z)\otimes \psi(X)^*\}$, with adjoint $S^*=\E\{\psi(X)\otimes \phi(Z)^*\}$, satisfying 
$\bk{f}{Sh}_{\Hz}=\bk{S^*f}{h}_{\Hx}  = \mathbb{E}[f(Z)h(X)].$ 
Together, these operators give rise to the modified covariance operator $T = S^*S_z^{-1}S$, which can be shown to satisfy $
\bk{f}{Tf}_{\Hx} = \mathbb{E}[\mathbb{E}\{f(X)|Z\}^2]$.

Kernel instrument variable regression (KIV) is a nonlinear extension of the standard two-stage least-squares (2SLS) method for linear estimation. In fact, 2SLS is a special case of KIV when $\kx$ and $\kz$ are linear kernels and with regularization set to zero. In other words, KIV generalizes from unregularized linear estimation to regularized nonlinear estimation. In the KIV model, the conditional expectation operator in~\eqref{eq:moment}  is given by $T^{\frac{1}{2}}$.\footnote{Formally, the conditional expectation is $S_z^{-1}S:h \mapsto \mathbb{E}\{h(X)\mid Z=\cdot\}$, and $T=S^*S_z^{-1}S$. Thus $T^{1/2}$ captures the operator norm over $L^2$, not the operator itself.}

We consider two equivalent formulations of the estimator's objective.  One formulation resembles regression, following \cite{singh2019kernel}:
\begin{align*}
    h_{\mu, \lambda} &= \argmin_{h\in \Hx} 
    \|(S_z+\mu)^{-1/2}S(h_0-h)\|^2_{\Hz}+\lambda\|h\|^2_{\Hx}.
\end{align*}
The first term is a projected mean square error, with ``first stage'' regularization $\mu>0$. The second term is a ridge penalty, with ``second stage'' regularization $\lambda>0$. With $X=Z$ and $\mu=0$, this objective reduces to the kernel ridge regression objective.

A second formulation is based on the conditional moment restriction
$\E(Y -h_0(X)|Z) = 0$. The adversarial formulation, following \cite{dikkala2020minimax} is
\begin{align}
h_{\mu, \lambda}=\argmin_{h\in\Hx} ~ \max_{f\in \Hz} ~ 2\E[\{Y-h(X)\}f(Z)]-\|f\|_{2}^2-\mu\|f\|^2_{\Hz}+\lambda\|h\|^2_{\Hx}. \label{eq:adversarial-definition}
\end{align}
Intuitively, the adversary $f$ maximizes the violation of the conditional moment. The estimator $h$
minimizes the violation of the conditional moment, anticipating this adversary.

Regardless of the formulation, the estimator has a convenient closed-form solution due to the kernel trick. Specifically, the empirical analogues of both objectives are minimized by the following algorithm.

\begin{algorithm}[Kernel instrumental variable regression]\label{alg:1}
Given a sample $D=\{(Z_i, X_i,Y_i)\}_{i=1}^n$, kernels $\kx$ and $\kz$, and regularization parameters $\lambda,\mu>0$: 
\begin{enumerate}
    \item Compute the kernel matrices
    $K_{XX}, K_{ZZ} \in \mathbb{R}^{n\times n}$ with $(i,j)$th entries  $\kx(X_i, X_j)$ and $\kz(Z_i, Z_j)$, respectively.
    \item Compute the kernel vector $K_{xX}\in \mathbb{R}^{1\times n} $ with $i$th entry $\kx(x, X_i)$
    \item Estimate KIV as $
    \hat{h}(x)
    =K_{xX} \{K_{ZZ}(K_{ZZ}+ n \mu I)^{-1}K_{XX}+n\lambda I\}^{-1}K_{ZZ}(K_{ZZ}+ n \mu I)^{-1}Y 
$.
\end{enumerate}

\end{algorithm}

\begin{example}[Linear kernel] Let $\mathcal{X} = \mathbb{R}^p$ and  $\mathcal{Z} = \mathbb{R}^q$, so that the covariates and instruments are finite-dimensional vectors, and consider the linear kernels $\kx(x,x')  = x^\top x'$ and $\kz(z,z')  = z^\top z'$. Then, $\Hx$ consists of linear functions of the form $h_\gamma(x) = \gamma^\top x$ for $\gamma \in \mathbb{R}^p$. Let $\mathbf{X} \in \mathbb{R}^{n\times p}$ and $\mathbf{Z} \in \mathbb{R}^{n\times q}$ be the design matrices. Then, the kernel objects become $K_{xX} = x\mathbf{X}^T$, $K_{XX} = \mathbf{X}\mathbf{X}^\top$ and $K_{ZZ} = \mathbf{Z}\mathbf{Z}^T$,  leading to 
\begin{align*}
    \hat h (x) &= x^\top\mathbf{X}^\top \Big[ \mathbf{Z} \mathbf{Z}^\top (\mathbf{Z} \mathbf{Z}^\top + n\mu I_n)^{-1} \mathbf{X} \mathbf{X}^\top + n\lambda I_n \Big]^{-1}
\mathbf{Z} \mathbf{Z}^\top (\mathbf{Z} \mathbf{Z}^\top + n\mu I_n)^{-1} Y
= x^\top \hat \gamma
\end{align*}
where $\hat{\gamma} =
\Big[ \mathbf{X}^\top \mathbf{Z} (\mathbf{Z}^\top \mathbf{Z} + n\mu I_q)^{-1} \mathbf{Z}^\top \mathbf{X} + n\lambda I_p \Big]^{-1}
\mathbf{X}^\top \mathbf{Z} (\mathbf{Z}^\top \mathbf{Z} + n\mu I_q)^{-1} \mathbf{Z}^\top Y$ is regularized 2SLS.
\end{example}

\begin{example}[Polynomial kernel]
Let $\mathcal{X}=\mathbb{R}^p$ and $\mathcal{Z}=\mathbb{R}^q$ as before, but now choose the $d$-degree polynomial kernels with offsets $c_x,c_z\ge 0$: $
\kx(x,x')=(x^\top x'+c_x)^d$ and $\kz(z,z')=(z^\top z'+c_z)^d.$
Then there exist finite-dimensional feature maps $\psi:\mathbb{R}^p\to\mathbb{R}^{M_x}$ and $\phi:\mathbb{R}^q\to\mathbb{R}^{M_z}$ with
$
M_x=\binom{p+d}{d},$ and $ M_z=\binom{q+d}{d},
$
such that $\kx(x,x')=\langle \psi(x),\psi(x')\rangle_{\Hx}$ and $\kz(z,z')=\langle \phi(z),\phi(z')\rangle_{\Hz}$.
The RKHS $\Hx$ consists of polynomials of degree at most $d$, of the form
$h(x)=w^\top \psi(x)$, with $w\in\mathbb{R}^{M_x}.$

Let $\Psi_X=\big[\psi(X_1),\ldots,\psi(X_n)\big]^\top\in\mathbb{R}^{n\times M_x}$ and
$\Phi_Z=\big[\phi(Z_1),\ldots,\phi(Z_n)\big]^\top\in\mathbb{R}^{n\times M_z}$ be the design matrices. Then, the kernel objects are $
K_{xX}=\psi(x)^\top \Psi_X^\top$, $
K_{XX}=\Psi_X\Psi_X^\top$, $
K_{ZZ}=\Phi_Z\Phi_Z^\top$,\footnote{Equivalently, the $(i,j)$th entries of $K_{XX}$ and $K_{ZZ}$ are $(X_i^\top X_j+c_x)^d$ and $(Z_i^\top Z_j+c_z)^d$, respectively, while the $i$th entry of $K_{xX}$ is $(x^\top X_i+c_x)^d$.} and
\begin{align*}
\hat h(x)
&= \psi(x)^\top \Psi_X^\top \Big[ \Phi_Z \Phi_Z^\top \big(\Phi_Z \Phi_Z^\top + n\mu I_n\big)^{-1} \Psi_X \Psi_X^\top + n\lambda I_n \Big]^{-1}  \Phi_Z \Phi_Z^\top \big(\Phi_Z \Phi_Z^\top + n\mu I_n\big)^{-1} Y= \psi(x)^\top \hat w,
\end{align*}
where
$
\hat w
= \Big[ \Psi_X^\top \Phi_Z \big(\Phi_Z^\top \Phi_Z + n\mu I_{M_z}\big)^{-1} \Phi_Z^\top \Psi_X + n\lambda I_{M_x} \Big]^{-1}
\Psi_X^\top \Phi_Z \big(\Phi_Z^\top \Phi_Z + n\mu I_{M_z}\big)^{-1} \Phi_Z^\top Y
$ is regularized 2SLS with $d$-order polynomial expansions.
\end{example}
\begin{example}[Preference kernel]
A key advantage of kernel methods is that kernels can be chosen by the researcher to handle non-standard data. As a leading example, \cite{singh2023kernel} consider student preferences data over 25 Boston schools. The space of preferences $\mathcal{X}$ has dimension $25!$. It is not feasible to model these preferences with $25!$ indicators. However, it is feasible to model these preferences with kernels. A natural preference kernel is
$
\kx(x,x') = \exp\{-N(x,x')\},
$
where $N(x,x')$ counts the number of pairwise disagreements between rankings $x$ and $x'$. 
This kernel induces an RKHS $\Hx$ over preferences.  

Unobserved confounding is an important concern when studying school choice. If families receive a ``nudge'' of a randomly assigned default preference in an online system, this randomly assigned preference may be viewed as an instrument. It would be natural to use the preference kernel for the instrument as well: $
\kz(z,z') = \exp\{-N(z,z')\},
$ inducing an RKHS $\Hz$.
\end{example}

\subsection{Goal: Valid and sharp confidence sets}

In this paper, our goal is to construct confidence sets $\hat{C}_n$ for the estimator in Algorithm~\ref{alg:1}, applicable to various data types. 
We would like these confidence sets to be computationally efficient: they should not require additional kernel evaluations or matrix inversions beyond Algorithm~\ref{alg:1}.
Theoretically, we would like these confidence sets to be valid and sharp: they should contain the true nonparametric instrumental variable regression function $h_0$ with at least, but not much more than, nominal coverage. In what follows, we carefully define validity and sharpness.

\begin{definition}[Validity]
$\hat{C}_n$ is $\tau$-valid at level $\chi$ if
$
\mathbb{P}\bigl(h_0 \in \hat{C}_n\bigr) \ge 1-\chi-\tau .
$
\end{definition}

\begin{definition}[Sharpness]
$\hat{C}_n$ is $(\delta,\tau)$-sharp at level $\chi$ if 
$
\mathbb{P}\!\left\{ h_0 \in (1-\delta)\hat{C}_n  +  \delta \hat{h}  \right\}
\le 1-\chi+\tau .
$
\end{definition}

A valid confidence set contains $h_0$ with at least nominal coverage, up to a tolerance level $\tau$. A sharp confidence set is not too conservative: if we slightly contract the set $\hat{C}_n$ towards its center $\hat{h}$, i.e. if we examine $(1-\delta)\hat{C}_n  +  \delta \hat{h}$, then this contracted set should contain $h_0$ at most at the nominal level, up to a tolerance level $\tau$. Intuitively, if $\tau=0$ and $\delta=0$, then validity and sharpness give exact coverage $
\mathbb{P}\bigl(h_0 \in \hat{C}_n\bigr)= 1-\chi.
$ We show that $\tau=\Ocal(n^{-1})$ and $\delta=\log(n)^{-1}$ in our nonasymptotic analysis.

The following bias-variance decomposition illuminates the structure of our argument:
$$
n^{1/2}(\hat{h}-h_0)=\underbrace{n^{1/2}\{(\hat{h}-h_{\mu,\lambda})-\E_n(U)\}}_{\text{residual}}+\underbrace{n^{1/2}\E_n(U)}_{\text{pre-Gaussian}}+\underbrace{n^{1/2}(h_{\mu,\lambda}-h_0)}_{\text{bias}},
$$
where $U_i\in \Hx$ is a mean-zero random function explicitly defined below. We prove that the residual term is asymptotically negligible in a strong sense. For the pre-Gaussian term, we provide Gaussian and bootstrap couplings under a decay condition. Lastly, we show that the bias of the estimator vanishes under source and link conditions. These conditions are standard assumptions in the NPIV literature, though we take care while interpreting them in our setting.

\subsection{Main assumption: Low effective dimension}

Our main assumption is that the data have low effective dimensions relative to the bases of the kernels. Specifically, we assume that the covariance operators $S_x$, $S_z$, and $T$ have eigenvalues that decay. Recall that $S_x=\E\{\psi(X)\otimes \psi(X)^*\}$, so decaying eigenvalues mean that relatively few dimensions of the features $\psi(X)$ can convey most of the information in the distribution of covariates $X$.
We quantify the rate of spectral decay via the local width. 

\begin{definition}[Local width]\label{def:local} Given $m>0$, the local width of operator $A$ is given by $\sigma^2(A,m)= \sum_{s>m}\nu_s(A)$, where $\{\nu_1(A), \nu_2(A), ...\}$ are decreasing eigenvalues.  
\end{definition}
The local width is the tail sum of the eigenvalues. It quantifies how much information we lose when only considering the initial $m$ dimensions.

In this paper, we prove that the pre-Gaussian term of KIV is of the form $n^{1/2}\E_n(U)$, where each summand is given by the expression 
$$
U_i=T_{\mu,\lambda}^{-1}\{M_i +M_i^* + N_i\}(h_0-h_{\mu,\lambda}) +T_{\mu,\lambda}^{-1}S^*(S_z+\mu)^{-1}\phi(Z_i)\ep_i.
$$
In this compact notation, $T_{\mu,\lambda}=S^*(S_z+\mu)^{-1}S+\lambda$. Moreover, 
$
M_i = S^*(S_z+\mu)^{-1}\{\phi(Z_i)\otimes\psi(X_i)^*-S\}
$ 
and 
$
N_i = S^*(S_z + \mu)^{-1}\{S_z - \phi(Z_i)\otimes \phi(Z_i)^*\}(S_z + \mu)^{-1}S.
$

The complexity of the pre-Gaussian term, and therefore the challenge of Gaussian approximation, is quantified by the local width of
$
\Sigma := \E(U \otimes U^*).
$
In our analysis, we show that the local width of $\Sigma$ is bounded by the local width of the covariance operator $S_x = \E\{\psi(X)\otimes \psi(X)^*\}$ in the sense that
$
\sigma^2(\Sigma, m)\leq \frac{4 \bar{\sigma}^2}{\lambda^2} \sigma^2(S_x, m)
.$
Therefore, the fundamental condition for Gaussian approximation is $\sigma^2(S_x, m) \downarrow 0$ as $m \uparrow \infty$. A simple sufficient condition is that the eigenvalues of $S_x$ decay polynomially.

In this paper, we focus on the case of polynomial decay for simplicity. Our results naturally extend to exponential decay, appealing to the corresponding bounds on local width in \cite{singh2023kernel}.

\subsection{Smoothness assumptions: Source and link conditions}

Next, we assume that the target function $h_0$ is smooth, and that its conditional expectation is smooth. These assumptions are called source and link conditions in the nonparametric instrumental variable regression literature \citep{chen2007rate,caponnetto2007optimal}. Such conditions are necessary for consistent estimation. Naturally, then, we also use them to derive valid inference.
\begin{assumption} [Source condition]
\label{ass:Source}
The target $h_0$ is smooth: there exists $\alpha\in[0,1]$ and $w_0 \in \Hx$ such that 
$h_0 = T^\alpha w_0.$    
\end{assumption}
Assumption~\ref{ass:Source} with $\alpha=0$ simply means that $h_0$ is correctly specified by $\Hx$. For $\alpha>0$, it means that $h_0$ is a particularly smooth element of $\Hx$.

\begin{assumption}[Link condition]\label{assn:restricted-link}
    For all smooth functions $f \in \Hx$, the conditional expectation remains a smooth function.
    Specifically, there exist $\beta \in [1/2,1]$ and $r \in [0, \infty)$ such that the operator $S_z^{-(\frac{1}{2}+\beta)}S:\Hx \to \Hz$ is bounded, i.e.,
  $\|S_z^{-(\frac{1}{2}+\beta)}S\|_{\op} \le r.$
\end{assumption}

Assumption~\ref{assn:restricted-link} with $\beta=0$ simply means that the conditional expectation operator, when viewed as a mapping from $\Hx$ to $L^2(Z)$, is bounded. For $\beta=0$, the link condition automatically holds.\footnote{To see why, notice that $\Hx$ embeds continuously in $L^2(X)$, and appeal to the law of total variance.} With $\beta>0$, it means that the conditional expectation operator, when viewed as a mapping from $\Hx$ to $\Hz^{\frac{1}{2}+\beta}$, is bounded. In other words, for any smooth function $h \in \Hx$, its conditional expectation $g(Z) = \mathbb{E}[h(X)|Z]$ is also smooth in the sense that $g \in\Hz^{\frac{1}{2}+\beta}$. Here, $\beta = \frac{1}{2}$, corresponds to the natural assumption that $g \in \Hz$, implying that the reduced form function $\E(Y|Z)$ is well-specified by $\Hz$.

Sobolev spaces are special cases of RKHSs. In Sobolev spaces, the source and link conditions amount to assumptions that the number of square integrable derivatives is high enough relative to the dimension of the data.

\subsection{Technical assumption: Strong instrument}

Finally, we require that the instrumental variable $Z$ is strong enough, in the sense that it carries enough information about the covariate $X$.  Such an assumption is standard in the instrumental variable literature, where it is often stated as a certain rank being large enough. In our setting, it becomes a condition that there are enough directions in $\Hx$ that are well explained by $\Hz$. Let $\tilde{\mathfrak{m}}(\lambda, \mu) := \tr T_{\mu,\lambda}^{-1}S^*(S_z+\mu)^{-1}S_z(S_z+\mu)^{-1}S T_{\mu,\lambda}^{-1}$ be the effective dimension of the nonparametric instrumental variable regression problem, generalizing the standard effective dimension \citep{caponnetto2007optimal,fischer2020sobolev}.

\begin{assumption}[Strong instrument] \label{ass:strong instrument}
The instrument is strong enough: 
$ \tilde{\mathfrak{m}}(\lambda, \mu)  \gtrsim_{\rho_x, \omega_x} \lambda^{-\rho_x},$
where $\rho_x\in[1,2]$ is the polynomial rate of decay for $S_x$.
\end{assumption}

Formally, we assume that 
a specific effective dimension does not vanish too quickly as the regularization parameter $\lambda$ vanishes. Later, we prove that $\tilde{\mathfrak{m}}(\lambda, \mu) \lesssim_{\rho_x, \omega_x}\lambda^{-\rho_x}$. Assumption~\ref{ass:strong instrument} assumes the matching lower bound $\tilde{\mathfrak{m}}(\lambda, \mu) \gtrsim \lambda^{-\rho_x}$, which rules out the possibility that the instrument is too weak.

Throughout the paper, we maintain that $\mu\leq\lambda\leq1$, which means that we are not regularizing the first stage more than the second stage. This restriction on the regularization parameters is required for our bias argument. 
\section{Confidence bands with nonstandard data}\label{sec:algo_main}

Inference for KIV poses several challenges. 
First, recovering $h_0$ from the operator equation \eqref{eq:moment} is hard because it requires inverting a conditional expectation operator. 
Formally, the conditional expectation operator $S_z^{-1}S: h(\cdot)\to \E\{h(X)|Z=(\cdot)\}$ is an infinite-dimensional quantity that must be estimated from data and then inverted.
Second, analyzing $\hat{h}$ in Algorithm~\ref{alg:1} is hard because it involves two regularization parameters.
Compared to standard regression, additional regularization is unavoidable in order to non-parametrically approximate $S_z^{-1}$ in the ``first stage''. However, this additional regularization introduces complex bias, and it is unclear whether the regularization parameters can vanish quickly enough to control bias while also vanishing slowly enough to permit Gaussian approximation. 
Third, constructing a confidence set centered at $\hat{h}$ becomes harder if we impose practical constraints: the confidence set should not require more computation than the estimator, and the confidence set should remain tractable whenever the estimator is tractable, e.g. with high dimensional or nonstandard data.

To overcome these challenges, we propose an anti-symmetric Gaussian multiplier bootstrap that yields valid and sharp confidence sets for the KIV estimator. 
A virtue of our procedure (Algorithm~\ref{alg:2}) is that the costliest step of KIV estimation---inversion of the two kernel matrices---is performed only once. 
Theoretically, it yields confidence sets that are valid and sharp in $\Hx$-norm (Theorem~\ref{thm:valid-inference}), allowing the obvious choice of regularization $\lambda=\mu$.
These sharp $\Hx$-norm confidence sets imply uniform confidence bands because $\|h-h_0\|_{\infty}\leq \kappa_x \|h-h_0\|_{\Hx}$ by the Cauchy-Schwarz inequality.  
Overall, we characterize a range of regimes with low effective dimension and high smoothness in which our method works.

\subsection{This work: Bootstrap for KIV}\label{sec:algo_bootstrap}

We preview our inference procedure at a high level before filling in the details. 
First, we compute the KIV estimator, saving the kernel matrices and their regularized inverses, which is reused in inference. This ensures the same $\Ocal(n^3)$ computational complexity of estimation alone (Algorithm \ref{alg:1}). 
Second, for each bootstrap iteration, we draw Gaussian multipliers and compute the bootstrap function $\mathfrak{B}$. 
Third, across bootstrap iterations, we calculate the $(1-\chi)$-quantile $\hat t_\chi$ of $\|\mathfrak{B}\|_\Hx$. Our confidence set $\hat C_{1-\chi}$ is the point estimate $\hat h $ plus $n^{-1/2}\hat t_\chi$ inflated by an incremental factor $\{1+1/\log(n)\}$. 

Importantly, $1/\log(n)$ is not a tuning parameter. We use this device to guarantee valid inference in many settings, inspired by \cite{andrews2013inference}. Similar to \cite{singh2023kernel}, it is possible to replace $1/\log(n)$ with zero by placing stronger assumptions on the effective dimension of the data and then employing techniques of \cite{chernozhukov2014anticoncentration} and \cite{gotze2019large}.

Our method provides $\Hx$-norm valued confidence sets. We can translate  these confidence sets into uniform confidence bands, because a bounded kernel implies that for all $h \in \Hx$, $\sup_{x\in \mathcal{X}} |h(x)|=\sup_{x\in \mathcal{X}} |\langle h,\psi(x)\rangle_{\Hx}| \leq \|h\|_{\Hx}  \sup_{x\in \mathcal{X}}\|\psi(x)\|_{\Hx}\leq  \kappa_x\|h\|_\Hx$. 
\begin{algorithm}[Confidence set for kernel instrumental variable regression]
 \label{alg:2}  
  Given a sample $D=\{(Z_i, X_i,Y_i)\}_{i=1}^n$, kernels $\kx$ and $\kz$, and regularization parameters $\lambda,\mu>0$:  
  \begin{enumerate}   
  \item Compute the kernel matrices  
   $K_{XX}, K_{ZZ} \in \mathbb{R}^{n\times n}$ as before.  
      \item Compute matrices $K = K_{ZZ}(K_{ZZ}+n\mu I)^{-1}$ and $A = (KK_{XX} + n\lambda I)^{-1}$. 
      \item Compute the KIV residuals $\hat \ep \in \mathbb{R}^n$ by  $ \hat \ep = Y - K_{XX} AKY$. 
      \item Precompute $D = K\text{diag}(\hat \ep) + (I_n - K)\text{diag}(K\hat \ep) \in \mathbb{R}^{n \times n}$.\item For each bootstrap iteration,   \begin{enumerate}  \item draw multipliers $q \in \mathbb{R}^n$ from $\mathcal{N}(0, I - 11^\top/n)$, where $1 \in \mathbb{R}^n$ has entries equal to one;                 
          \item compute the vector  $\hat{\gamma} = n^{1/2} ADq;  $   \item compute the scalar $  M = (\hat{\gamma}^\top K_{XX}  \hat{\gamma})^{1/2}. $  \end{enumerate}  
      \item Across bootstrap iterations, compute the $(1-\chi)$-quantile, $\hat{t}_\chi$, of $M$.
       \item Calculate the $\Hx$ confidence set:  $  \hat{C}_{1-\chi} = \left\{ \hat{h} \pm \hat{t}_\chi  n^{-1/2} h :  \|h\|_{\Hx} \le 1 + 1/\log(n)\right\}.$  \item Calculate the uniform confidence band: $    \hat{C}_{1-\chi}(x) =\left[\hat{h}(x) \pm \hat{t}_\chi  n^{-1/2} \kappa_x\{1+1/\log(n)\} \right].$ 
       \end{enumerate}
  \end{algorithm}  
Within Algorithm~\ref{alg:2}, we implicitly calculate $M = \|\mathfrak{B}\|_\Hx$ for the $\Hx$ valued bootstrap function 
$
\mathfrak{B} = \frac{1}{n}\sum_{i = 1}^n\sum_{j = 1}^n\frac{1}{\sqrt{2}}\left(\hat{V}_i- \hat{V}_j\right)h_{ij}$. Here, $h_{ij}$ are independent and identically distributed standard Gaussians. Each summand $\hat{V}_i$ is the empirical analogue of the pre-Gaussian summand $U_i$. Intuitively, by taking the difference $\hat{V}_i-\hat{V}_j$, we cancel the complex bias of KIV. In this way, we leverage symmetry. See Appendix~\ref{sec:algo} for details on the implicit estimator $\hat{V}_i$ of $U_i$.

\subsection{Main result: Valid and sharp inference}\label{sec:algo_lead}

We present our main result: Theorem~\ref{thm:valid-inference} proves that the confidence set in Algorithm~\ref{alg:2} is valid and sharp in  $\Hx$-norm. Moreover, Corollary~\ref{cor:uniform} verifies that the confidence band in Algorithm~\ref{alg:2} is valid in $\sup$ norm; it is a valid uniform confidence band for kernel instrumental variable regression, filling a crucial gap in the literature.

Our proof technique is largely agnostic about how low the effective dimensions of the data are. Fundamentally, we require that the local width in Definition~\ref{def:local} vanishes for the pre-Gaussian term, i.e. that $\sigma^2(\Sigma,m)\downarrow 0$ for $\Sigma=\E(U\otimes U^*$) and $U$ stated below Definition~\ref{def:local}.

For exposition, we impose further structure on the problem, which implies the high-level condition. In particular, we assume that the eigenvalues of the covariance operators $S_x$ and $S_z$ decay polynomially. This is a standard regime for RKHS analysis \citep{caponnetto2007optimal,fischer2020sobolev}, which generalizes the Sobolev setting. Formally, to state Theorem~\ref{thm:valid-inference}, we impose that $
\nu_s(S_x)\asymp \omega_x s^{-1/(\rho_x-1)}
$ and $
\nu_s(S_z) \asymp \omega_z s^{-1/(\rho_z -1)}$,  where $\rho_x, \rho_z \in (1,2]$ quantify the rates of polynomial decay, and where $\omega_x$ and $\omega_z$ are constants. These assumptions imply that the key local widths vanish: $\sigma^2(S_x,m)\downarrow 0$, and therefore $\sigma^2(\Sigma,m)\downarrow 0$, as required.

A weak regularity condition throughout our paper is that the kernels $k_x$ and $k_z$ are bounded. This condition is satisfied for kernels commonly used in practice. In the regularization regimes we consider, this in turn implies that the summands in the pre-Gaussian term satisfy $\|U_i\|_{\Hx} = \Ocal(1/\sqrt{\lambda\mu})$ almost surely.\footnote{See Lemmas~\ref{lemma:iota} and~\ref{lemma:iota2}. We work with $(\lambda,\mu)$ regimes that guarantee $\|U_i\|_{\Hx} = \Ocal(1/\sqrt{\lambda\mu})$, by imposing that $\lambda$  and $\mu$  scale  similarly. For example, when $\lambda = \mu^{\iota}$, we require that $\iota \in(0,1]$ must be large enough.} Stronger assumptions could be imposed here, which we defer to future work.

\begin{theorem}[Valid and sharp confidence sets]\label{thm:valid-inference}
For $\chi \in(0,1)$, define $\hat t_\chi $ by $
\mathbb{P}\left(\|\mathfrak{B}\| > \hat t_\chi  \middle| D\right)= \chi $.
Suppose the data have low effective dimensions, i.e. $
\nu_s(S_x)\asymp \omega_x s^{-1/(\rho_x-1)}
$ and $
\nu_s(S_z) \asymp \omega_z s^{-1/(\rho_z -1)}$.
Suppose smoothness and strong instrument conditions hold, i.e Assumptions~\ref{ass:Source},~\ref{assn:restricted-link} and~\ref{ass:strong instrument}. 
Set $(\lambda,\mu)$ satisfying $\lambda \geq \mu$ and according to Table~\ref{table:restrictions}, e.g. $\lambda=\mu$.\footnote{See Assumption~\ref{ass:COLOBO} for details.}
Suppose the effective dimensions are low enough and the smoothness is high enough according to Table~\ref{table:restrictions combined}.
Then the $\Hx$ confidence set in Algorithm~\ref{alg:2} is $\Ocal(1/n)$-valid and $\{2/\log(n), \Ocal(1/n)\}$-sharp.
\end{theorem}

\begin{corollary}[Uniform confidence sets]\label{cor:uniform}
Under the assumptions of Theorem \ref{thm:valid-inference}, the uniform confidence band in Algorithm~\ref{alg:2} is $\Ocal(1/n)$-valid.

\end{corollary}

\subsection{Key intermediate results}\label{sec:algo_incremental}

Our main result, Theorem~\ref{thm:valid-inference}, ties together four intermediate results, which we present below: 
(i) a bias upper bound (Proposition~\ref{prop:Bias upper bound}),
(ii) a Gaussian coupling (Theorem~\ref{theorem:gaussian-approx1}),
(iii) a bootstrap coupling (Theorem~\ref{thm: Boostrap-approximation}),
and (iv) a variance lower bound (Proposition~\ref{prop:VALBO}).
We summarize these four intermediate results in the leading case of polynomial decay for $S_x$ and $S_z$ (Table~\ref{table:1}).
To validate inference, we must show that the errors arising from (i), (ii), and (iii) are dominated by (iv).
This leads to restrictions on the regularization parameters $(\lambda,\mu)$ (Table~\ref{table:restrictions}), and requirements that the effective dimension is low enough and smoothness is high enough (Table~\ref{table:restrictions combined}).

\begin{table*}[t]
\centering
\begin{threeparttable}
\caption{Intermediate results under polynomial decay of $S_x$ and $S_z$}
\label{table:1}
\setlength{\tabcolsep}{8pt}
\renewcommand{\arraystretch}{1.2}
\begin{tabular}{@{}c c@{}}
\toprule
Bias upper bound: $B$ & $n^{1/2}\lambda^{\alpha}$ \\[.3cm]
Gaussian coupling: $Q_{\bullet}$ & $\dfrac{1}{\lambda}\left(\dfrac{n\mu}{\lambda}\right)^{\frac{\rho_x-2}{2(3\rho_x-2)}}$ \\[.3cm]
Gaussian coupling:  $Q_{\mathrm{res}}$ &
$\dfrac{\lambda^{\alpha -1}}{\sqrt{n}\mu^2}   +  \dfrac{\lambda^{\alpha -1/2}}{\sqrt{n}\mu^{1/2 + \rho_z}}+ \dfrac{  1}{n\lambda^{3/2}\mu^{1+ \rho_z}}  + \dfrac{1}{\sqrt{n}\lambda \mu^{3/2}}$ \\[.3cm]
Bootstrap coupling: $R_{\bullet}$ & $\dfrac{1}{\lambda^{\rho_x/2}}(n\mu)^{\frac{2-\rho_x}{2(\rho_x-3)}}$ \\[.3cm]
Bootstrap coupling: $R_{\mathrm{res}}$ &
$\dfrac{\lambda^{\alpha-1}}{\mu^2\sqrt{n}}
+\dfrac{1}{\lambda\mu^{3/2}\sqrt{n}}
 + \dfrac{\lambda^{\alpha-1}}{n\mu^{3/2 + \rho_z}}+ \dfrac{ 1}{n^{3/2}\lambda^{2}\mu^{2+\rho_z}}$ \\[.3cm]
Variance lower bound: $L$ & $\lambda^{-\rho_x/2}$ \\[.2cm]
\bottomrule
\end{tabular}

\begin{tablenotes}[flushleft]
\footnotesize
\item 
\noindent The first row gives the bias upper bound $B$. The second and third rows give the Gaussian coupling bound $Q=Q_{\bullet}+Q_{\mathrm{res}}$. The fourth and fifth rows give the bootstrap coupling bound $R=R_{\bullet}+R_{\mathrm{res}}$.
The final row is the variance lower bound. Throughout, we suppress logarithmic factors and constants, and we impose $\lambda=\mu^{\iota}$ with $\iota\leq 1$.
\end{tablenotes}
\end{threeparttable}
\end{table*}

To begin, we characterize the bias of kernel instrumental variable regression under our smoothness assumptions. Under suitable regularity assumptions, our bias bound matches the well-known bias bound of kernel ridge regression \citep{smale2005shannon,caponnetto2007optimal,fischer2020sobolev}. We place two smoothness assumptions: a source condition ensuring that $h_0$ is smooth, and a link condition ensuring that the conditional expectation operator maps smooth functions to smooth functions. Both assumptions are in line with previous applications of RKHS methods to ill-posed inverse problems \citep{nashed1974regularization,singh2019kernel,meunier2024nonparametric}.

\begin{proposition}[Bias upper bound]\label{prop:Bias upper bound}
Suppose that Assumptions~\ref{ass:Source} and \ref{assn:restricted-link} hold,
and the regularization satisfies $\mu^\beta r  < \lambda^{1/2}$. 
Then
$n^{1/2}\|\hml - h_0\|_{\Hx} \le B=n^{1/2}\frac{\|h_{0,\lambda} - h_0\|_{\Hx}}{1 - C_\beta r\mu^{\beta}/\lambda^{\frac{1}{2}}} \le n^{1/2}\frac{C_\alpha \lambda^{\alpha} \|T^{-\alpha} h_0\|_{\Hx}}{1 - C_\beta r \mu^{\beta}/\lambda^{\frac{1}{2}}}. $
\end{proposition}

Under weak regularity conditions, the bias simplifies to $B=n^{1/2}\lambda^{\alpha}$, which is easy to interpret: for a smoother target function $h_0$, the smoothness parameter $\alpha$ is larger, and the bias vanishes more quickly. Overall, more smoothness translates into an easier estimation problem with better convergence rates. A simple and convenient regularity condition is that the conditional expectation operator is a bounded map from $\Hx$ to $\Hz$, i.e.~$\beta\ge\tfrac12$. Under this convenient regularity condition, it suffices to place a mild restriction on the regularization: $\mu \le r^{-2}\lambda$, which ensures that the denominator in Proposition~\ref{prop:Bias upper bound} is bounded away from zero. 
In light of the adversarial formulation in \eqref{eq:adversarial-definition}, 
this mild restriction on regularization ensures that the adversary's strategy space is not too constrained, so that the adversary may adequately detect correlation between the instrument $Z$ and the endogenous error $\varepsilon$. The mild restriction holds under natural choices of regularization: either $\lambda=\mu^\iota$ with $\iota\leq 1$, or $\mu=\lambda/C$ with $C>(C_\beta r)^{2}$.

Equipped with this bias bound, we present two key results that underpin Theorem~\ref{thm:valid-inference}: a Gaussian coupling (Theorem~\ref{theorem:gaussian-approx1}), and a bootstrap coupling (Theorem~\ref{thm: Boostrap-approximation}).

\begin{theorem}[Gaussian approximation]\label{theorem:gaussian-approx1}
Suppose the conditions of Proposition~\ref{prop:Bias upper bound} hold. 
Next, suppose the data have low effective dimension, i.e., $
\nu_s(S_x)\asymp \omega_x s^{-1/(\rho_x-1)}
$ and $
\nu_s(S_z) \asymp \omega_z s^{-1/(\rho_z -1)}$.
Set the regularization $(\lambda,\mu)$ so that $\|U_i\|_{\Hx}\lesssim\frac{1}{\sqrt{\lambda\mu}}$.\footnote{This mild condition requires $\lambda = \mu^{\iota/(2\beta)}$ for sufficiently large $\iota \in (0,1]$, see e.g.~Lemmas~\ref{lemma:iota} and~\ref{lemma:iota2}. }
Finally, assume that $n$ is sufficiently large.\footnote{See Assumption~\ref{ass:rate condition} for a precise statement.}
Then, there exists a Gaussian random element $Z\in \Hx$, with covariance $\Sigma$, such that with probability $1-\eta$, 
$\left\|n^{1/2}(\hat h - \hml) - Z\right\|_{\Hx} \lesssim  Q_\bullet  \widetilde M\log(36/\eta) +Q_{\mathrm{res}}$. In this compact notation, $\widetilde M$ is an absolute constant, and the key quantities are
\begin{align*}
    Q_\bullet &= \inf_{m\geq 1}\left\{
 \frac{\sigma(S_x, m)}{\lambda}
+ \frac{m^2 \log(m^2)}{\sqrt{n\mu\lambda}}\right\},\quad \nmu = \tr(S_z +\mu)^{-2}S_z, \\
Q_{\mathrm{res}} &= \left(\left\{\frac{\log(12/\eta)^2}{\sqrt{n}\lambda\mu^2} \vee\frac{\log(12/\eta)^2 \nmu}{\sqrt{n}\lambda^{1/2}\mu^{1/2}} \right\} \|\hml - h_0\|_\Hx  \vee \frac{ \nmu \log(12/\eta)^3}{n\lambda^{3/2}\mu}  \vee \frac{\log(12/\eta)^2}{\sqrt{n}\lambda \mu^{3/2}} \right).
\end{align*}
\end{theorem}

The quality of Gaussian approximation $Z$ ultimately depends on a handful of quantities: spectral decay of $S_x$ and $S_z$; regularization parameters $(\lambda,\mu)$; and the bias $n^{1/2}\|\hml - h_0\|_{\Hx}$.
Complexity of the covariate distribution is reflected by the local width $\sigma(S_x, m)$, and complexity of the marginal distribution of the instrument is reflected by $\nmu$. 
If the effective dimension is higher, then these quantities are larger, and the quality of Gaussian approximation is worse. 
Gaussian approximation also degrades when the regularization parameters $(\lambda,\mu)$ are too small; intuitively, with less regularization, both the estimator and its Gaussian approximation are less stable. 
Finally, a large bias also complicates the Gaussian approximation.

Proposition~\ref{prop:Bias upper bound} and Theorem~\ref{theorem:gaussian-approx1} preview an important tension.  Smaller $(\lambda,\mu)$ help estimation by limiting the bias, but possibly hurt inference by driving up complexity of the estimator. We use nonasymptotic analysis to thread this needle, achieving Gaussian approximation when $(\lambda,\mu)$ approach zero, even though no stable  Gaussian limit exists. To sample from the sequence of approximating Gaussians $Z$, this paper proposes a new bootstrap procedure $\mathfrak{B}$. The following result proves the validity of the bootstrap.

\begin{theorem}[Bootstrap approximation]\label{thm: Boostrap-approximation} 
Suppose that the conditions of Theorem~\ref{theorem:gaussian-approx1} hold. 
Further assume that $n$ is sufficiently large.\footnote{See Assumption~\ref{ass:combined_rate_condition} for a precise statement.}
Then, there exists
a Gaussian random element $Z'\in\Hx$ whose conditional distribution given $U$ has covariance $\Sigma$, such that with probability $1-\eta$, we have 
$
    \mathbb{P}\left[ \|Z'-\mathfrak{B} \| \lesssim \widetilde M \log(6/\eta)^{3/2} R_{\bullet} + R_{\mathrm{res}} |U\right]\geq 1-\eta
$. 
In this compact notation, $\widetilde M$ is an absolute constant, and the key quantities are
\begin{align*}
    R_\bullet&= \inf_{m\ge 1}\left[ m^{1/4}\left\{ \frac{ \tilde{\mathfrak{m}}(\lambda,\mu)}{n\mu\lambda } + \frac{1}{n^{2}\mu^2 \lambda^2} \right\}^{1/4} + \frac{ \sigma(S_x, m)}{\lambda} \right],\quad \tilde{\mathfrak{m}}(\lambda, \mu)= \tr T_{\mu,\lambda}^{-1}S^*(S_z+\mu)^{-1}S_z(S_z+\mu)^{-1}S T_{\mu,\lambda}^{-1},\\
R_{\mathrm{res}}&=\sqrt{2\log\left(15/\eta\right)}\Bigg[ \left\{
\frac{\log(30/\eta)}{\sqrt{n} \mu^2 \lambda}  + \frac{\nmu \log(30/\eta)^2 }{n\lambda\mu^{3/2 }}\right\} \|\hml - h_0\|_{\Hx} \
+\frac{\log(30/\eta)}{\sqrt{n}\lambda\mu^{3/2}}
 + \frac{ \nmu \log(30/\eta)^3}{n^{3/2}\lambda^{2}\mu^{2}}
\Bigg].
\end{align*}
\end{theorem}

Once again, a handful of quantities determine the quality of bootstrap approximation $\mathfrak{B}$: spectral decay of $T$ and $S_z$; regularization parameters $(\lambda,\mu)$; and the bias $n^{1/2}\|\hml - h_0\|_{\Hx}$. 
Now, the spectral quantities are the local width $\sigma(S_x, m)$, the first stage effective dimension parameter $\nmu$, and a new effective dimension parameter $\tilde{\mathfrak{m}}(\lambda,\mu)$. 
As before, if the effective dimensions are higher, then these quantities are larger, and the quality of bootstrap approximation is degrades. 
As before, bootstrap approximation also degrades if $(\lambda,\mu)$ are too small, reinforcing the trade-off between estimation and inference.

The procedure $\mathfrak{B}= \frac{1}{n}\sum_{i = 1}^n\sum_{j = 1}^n\frac{1}{\sqrt{2}}\left(\hat{V}_i- \hat{V}_j\right)h_{ij}$ that we analyze in Theorem~\ref{thm: Boostrap-approximation} is the empirical counterpart of an infeasible bootstrap $Z_{\mathfrak{B}}=\frac{1}{n}\sum_{i = 1}^n\sum_{j = 1}^n\left(\frac{V_i-V_j}{\sqrt{2}}\right)h_{ij}$. By briefly discussing $Z_{\mathfrak{B}}$, we clarify the role of anti-symmetry in our analysis. At a high level, $\hat{V}_i$ approximates $V_i$, but $V_i$ is biased for the desired $U_i$ in the pre-Gaussian term: $V_i=U_i+\theta$, where $\E(U_i)=0$ by construction and where $\theta\neq 0$ is a complex bias.\footnote{See Appendix~\ref{sec:bahadur2} for details.} Our key observation is that we can cancel out the bias $\theta$ by taking differences: $V_i-V_j=U_i-U_j$. This differencing is equivalent to generating the matrix of Gaussian multipliers $h\in\R^{n\times n}$ then constructing an anti-symmetric matrix of multipliers $h-h^{\top}$. Our technique expands a proposal of \cite{freedman1981bootstrapping}, who addresses bias arising from non-orthogonality of errors in homoscedastic, fixed design regression.

The results in Proposition~\ref{prop:Bias upper bound}, Theorem~\ref{theorem:gaussian-approx1}, and Theorem~\ref{thm: Boostrap-approximation} are all upper bounds that pertain to the fundamental decomposition
$$
n^{1/2}(\hat{h}-h_0)=\underbrace{n^{1/2}\{(\hat{h}-h_{\mu,\lambda})-\E_n(U)\}}_{\text{residual}}+\underbrace{n^{1/2}\E_n(U)}_{\text{pre-Gaussian}}+\underbrace{n^{1/2}(h_{\mu,\lambda}-h_0)}_{\text{bias}}.
$$
Proposition~\ref{prop:Bias upper bound} controls the nonrandom bias term. Theorem~\ref{theorem:gaussian-approx1} constructs a Gaussian coupling $Z$ for the pre-Gaussian term, and controls the residual. Theorem~\ref{thm: Boostrap-approximation} allows us to sample from the approximating Gaussian $Z$ via a feasible bootstrap $\mathfrak{B}$. 
Our final result, which completes the paper, is a lower bound of the variance of the approximating Gaussian $Z$. Such a lower bound allows us to construct valid confidence sets. Specifically, we require a lower bound on the variance that is still larger than the upper bounds on the bias, Gaussian coupling error, and bootstrap coupling error.

\begin{proposition}[Variance lower bound]\label{prop:VALBO}
Let $Z \in \Hx$ be a Gaussian with covariance $\Sigma$. Suppose $\bb{E}(\ep_i^2|Z_i) \ge \underline{\sigma}^2$ almost surely. 
Finally, set $(\lambda,\mu)$ according to weak regularity conditions.\footnote{See Assumption \ref{ass:COLOBO} for details.}
Then, with probability $1-\eta$,
$\norm{Z}_\Hx \ge \sqrt{\frac{1}{2 }\underline{\sigma}^2\mathfrak{\tilde m}(\lambda, \mu)} - \left \{ 2 + \sqrt{2 \log(1/\eta)}\right \}\sqrt{\frac{ 2\bar \sigma^2}{\lambda} }.$
\end{proposition}
Intuitively, for this lower bound to be meaningful, we require a strong instrument condition (Assumption~\ref{ass:strong instrument}). 
The strong instrument assumption prevents the effective dimension 
$\tilde{\mathfrak m}(\lambda,\mu)$ from collapsing as $\lambda,\mu \downarrow 0$. 
In other words, it imposes that a sufficiently rich set of directions in $\Hx$ is well explained by $\Hz$.
Viewed through this lens, our strong instrument assumption is an infinite-dimensional analogue of the familiar rank condition in classical instrumental variable analysis.
With our intermediate results in hand, as summarized by Table~\ref{table:1}, we prove Theorem~\ref{thm:valid-inference}. First, we demonstrate that the bias is dominated by the variance, i.e., $B\ll L$. Then, we demonstrate that the Gaussian and bootstrap coupling errors are dominated by the variance, i.e. $Q\ll L$ and $R\ll L$ where  $Q = Q_\bullet + Q_{\mathrm{res}}$ and $R = R_\bullet + R_{\mathrm{res}}$. For these orderings to hold, we require sensible restrictions on the regularization $(\lambda,\mu)$, low enough effective dimension, and high enough smoothness, as summarized by Tables~\ref{table:restrictions} and~\ref{table:restrictions combined}. Recall that as $(\lambda,\mu)$ vanish, $B$ converges yet $(Q,R)$ diverge. Our nonasymptotic analysis carefully navigates this tension between estimation and inference, in order to arrive at both uniform estimation and uniform inference for kernel instrumental variable regression.

\section{Discussion: Inference for causal functions}\label{sec:conclusion}
Our main contribution is to develop an inference procedure for kernel instrumental regression. 
The procedure retains the simple closed-form structure of kernel estimators while delivering strong statistical guarantees, even with complex or nonstandard data types. 
By complementing flexible nonparametric estimation with reliable nonparametric inference, we aim to broaden the use of kernel methods for causal analysis, in the social and biomedical sciences. 
Looking ahead, our techniques may extend beyond instrumental variables to a wider class of causal functions \citep{singh2020kernel,singh2025sequential}.

\bibliographystyle{hapalike_mod}
\spacingset{1}
{
\bibliography{0_refs}}
\spacingset{1.5}

 \clearpage 
 
 \appendix

\spacingset{1}
 \addcontentsline{toc}{section}{Appendix} 
\part{Appendix} 
\parttoc 
\spacingset{1.5}

 \section{Algorithm derivation}\label{sec:algo}
We derive the closed form expressions of $\hat h(x)$ and $\mathfrak{B}(x)$, justifying Algorithms \ref{alg:1} and \ref{alg:2}.
\subsection{Equivalent objectives}
\begin{lemma}[Dual characterization of the regularized KIV estimator]
\label{lemma:dual}
    For all $\lambda,\mu>0$, the solution to the adversarial objective coincides with the solution to the regression objective
    \begin{equation}
    \label{eq:reg_obj}
        h_{\mu, \lambda} = \argmin_{h \in \Hx} \left\|(S_z+\mu)^{-1/2}S(h_0-h)\right\|^2_{\Hz} + \lambda\|h\|^2_{\Hx}.
    \end{equation}
\end{lemma}
\begin{proof}
    The adversary $f_h$ is the solution to 
\begin{align*}
    f_h& =  \argmax_{f \in \Hz} 2\E[\{Y-h(X)\}f(Z)]-\|f\|_{2}^2-\mu\|f\|^2_{\Hz} \\
    &= \argmax_{f \in \Hz}  2\mathbb{E}[\{Y-h(X)\}\bk{f}{\phi(Z)}_{\Hz}]-\bk{f}{(S_z+\mu)  f}_{\Hz} ,
\end{align*}
where we used the definition of $S_z$ and the \( L^2 \)-norm. Boundedness of the feature map and its measurability imply Bochner integrability of the feature map \citep{steinwart2008support}, which allows us  to write the \( L^2 \)-norm as
\[
\|f\|_2^2 = \mathbb{E}[f(Z)^2] = \mathbb{E}[\langle f, \phi(Z) \rangle_{\Hz}^2] 
= \mathbb{E}[\langle f, \phi(Z) \rangle_{\Hz} \langle f, \phi(Z) \rangle_{\Hz}]
= \left\langle f, \mathbb{E}[\langle f, \phi(Z) \rangle_{\Hz} \, \phi(Z)] \right\rangle_{\Hz}
= \langle f, S_z f \rangle_{\Hz}.
\]
The first order condition implies that $f_h$ has to satisfy
\begin{align*}
2\mathbb{E}[\{Y-h(X)\}\phi(Z)]-2(S_z+\mu)f_h = 0 &\iff (S_z+\mu)f_h = \mathbb{E}[\{Y-h(X)\}\phi(Z)]=S(h_0-h) \\
& \implies f_h = (S_z+\mu)^{-1}S(h_0-h),
\end{align*}
where we used that
\begin{align*}
    \E[\{Y-h(X)\}\phi(Z)]&=\E[\{h_0(X)+\ep-h(X)\}\phi(Z)]\\
    &=\E[\{h_0(X)-h(X)\}\phi(Z)] \\
    &= \E[\phi(Z)\bk{\psi(X)}{h_0-h}] \\
    & = \E[\phi(Z)\otimes \psi(X)^*\{h_0-h\}] = S(h_0-h).
\end{align*}
Proceeding in a similar fashion and plugging in the first order condition $(S_z+\mu)f_h =S(h_0-h)$, we can rewrite the adversarial objective as
\begin{align*}
&2\E[\{Y-h(X)\}f_h(Z)]-\|f_h\|_2^2-\mu\|f_h\|_\Hz^2+\lambda\|h\|_\Hx^2 \\
&= 2\bk{S(h_0-h)}{f_h}_{\Hz} - \bk{f_h}{(S_z+\mu)f_h}_{\Hz} + \lambda\bk{h}{h}_{\Hx}
\\ &=  2\bk{S(h_0-h)}{f_h}_{\Hz} - \bk{f_h}{S(h_0-h)}_{\Hz} + \lambda\bk{h}{h}_{\Hx} \\
& = \bk{S(h_0-h)}{f_h}_{\Hz} + \lambda\bk{h}{h}_{\Hx} \\
& = \bk{S(h_0-h)}{(S_z+\mu)^{-1}S(h_0-h)}_{\Hz} + \lambda\bk{h}{h}_{\Hx} \\
& = \bk{(S_z+\mu)^{-1/2}S(h_0-h)}{(S_z+\mu)^{-1/2}S(h_0-h)}_{\Hz} + \lambda\bk{h}{h}_{\Hx} \\
& = \|(S_z+\mu)^{-1/2}S(h_0-h)\|^2_{\Hz}+\lambda\|h\|^2_{\Hx}. \quad \qedhere
\end{align*}
\end{proof}
\subsection{Closed form estimation}\label{sec:closed_form_Bootstrap}
We introduce vector notation of the form
\[
\Psi_X :=
\begin{bmatrix}
\psi(X_1)^\top\\[-2pt]
\vdots\\[-2pt]
\psi(X_n)^\top
\end{bmatrix},
\qquad
\Phi_Z :=
\begin{bmatrix}
\phi(Z_1)^\top\\[-2pt]
\vdots\\[-2pt]
\phi(Z_n)^\top
\end{bmatrix}
\]
and Gram matrices

\begin{tabular}{lll}
    $K_{XX}:=\Psi_X\Psi_X^\top\in\R^{n\times n}$,  & $(K_{XX})_{ij}=\psi(X_i)^\top\psi(X_j)$, & $K_{xX}:=\psi(x)^\top\Psi_X^\top$, \\
    $K_{ZZ}:=\Phi_Z\Phi_Z^\top\in\R^{n\times n}$,  & $(K_{ZZ})_{ij}=\phi(Z_i)^\top\phi(Z_j)$. & \\ 
\end{tabular}
\begin{lemma}[Point estimate]
\label{lemma:point}
    For all $\lambda, \mu >0$ the point estimator to (\ref{eq:reg_obj}) is given by 
    $$\hat{h}(x)=K_{xX} \{K_{ZZ}(K_{ZZ}+ n \mu)^{-1}K_{XX}+n\lambda\}^{-1}K_{ZZ}(K_{ZZ}+ n \mu)^{-1}Y.$$
    For $K_{ZZ}(K_{ZZ}+n\mu)^{-1} = I $, this estimator becomes the KRR solution \citep{kimeldorf1971some}.
\end{lemma}
\begin{proof}
Recall that $\E[\{Y-h(X)\}\phi(Z)] = S(h_0-h)$. With $\varepsilon_h=Y-h(X)$, the empirical objective of (\ref{eq:reg_obj}) becomes $
 \|(\hat{S}_z+\mu)^{-1/2}\E_n \{\varepsilon_h\phi(Z)\}\|_{\Hz}^2+\lambda\|h\|^2_{\Hx}.
$
With $\E_n \{\varepsilon_h\phi(Z)\}=n^{-1}\Phi_Z^* \ep_h$, the objective becomes
\begin{align*}
\begin{split}
    &\langle (\hat{S}_z+\mu)^{-1/2}n^{-1}\Phi_Z^* \ep_h,(\hat{S}_z+\mu)^{-1/2}n^{-1}\Phi_Z^* \ep_h\rangle_{\Hz} +\lambda \langle h,h\rangle_{\Hx}  \\
    &=n^{-2}  \ep_h^{\top}\Phi_Z(\hat{S}_z+\mu)^{-1}\Phi_Z^* \ep_h +\lambda \langle h,h\rangle_{\Hx}   \\
    &=n^{-1}  \ep_h^{\top}\Phi_Z(\Phi_Z^*\Phi_Z+n\mu)^{-1}\Phi_Z^* \ep_h +\lambda \langle h,h\rangle_{\Hx}   \\
    &=n^{-1}  \ep_h^{\top}K_{ZZ}(K_{ZZ}+n\mu)^{-1} \ep_h +\lambda \langle h,h\rangle_{\Hx}  \\
    &=n^{-1}  \ep_h^{\top}K \ep_h +\lambda \langle h,h\rangle_{\Hx} .
\end{split}
\end{align*}
    By the representer theorem, we can write $h = \Psi_X^*\alpha$ to express $$
\ep_h=Y-\Psi_X h=Y-\Psi_X \Psi_X^*\alpha=Y-K_{XX}\alpha.
$$
Plugging $\ep_h$ in the objective yields
$$
\frac{1}{n}(Y-K_{XX}\alpha)^{\top}K(Y-K_{XX}\alpha)+\lambda \alpha^{\top} K_{XX}\alpha.
$$ Setting the derivative with respect to $\alpha$ to zero yields the first order condition
$$
0=-\frac{2}{n}K_{XX}K(Y-K_{XX}\hat \alpha)+2\lambda K_{XX}\hat\alpha=\frac{2}{n}K_{XX}\{-K(Y-K_{XX}\hat\alpha)+n\lambda \hat\alpha\}.
$$
Setting the inner expression equal to zero gives
$$
K(Y-K_{XX}\hat\alpha)=n\lambda \hat\alpha\iff KY=(KK_{XX}+n\lambda) \hat\alpha \iff \hat\alpha=(KK_{XX}+n\lambda)^{-1}KY.
$$
Finally, we can write
\begin{align*}
\begin{split}
    \hat{h}(x)&=\psi(x)^* \Psi_X^*\hat \alpha\\
    &=K_{xX} (KK_{XX}+n\lambda)^{-1}KY \\
    &=K_{xX} \{K_{ZZ}(K_{ZZ}+ n \mu)^{-1}K_{XX}+n\lambda\}^{-1}K_{ZZ}(K_{ZZ}+ n \mu)^{-1}Y. \qquad \qedhere
\end{split} 
\end{align*}
 
\end{proof}
  \subsection{Closed form bootstrap}                                                                                                                             
  \begin{proposition}[Kernel bootstrap]
  For all $x \in \mathcal X$, the bootstrap process satisfies \[    \mathfrak{B}(x)    = K_{xX}\bigl(KK_{XX}+n\lambda I_n\bigr)^{-1}Dc,         
    \]    
   where    
   \[ 
  c := \frac{1}{\sqrt{2}}(h - h^{\top})1,
    \qquad   
      K := K_{ZZ}\bigl(K_{ZZ}+n\mu I_n\bigr)^{-1},  
     \qquad     
     D := K\diag(\hat\varepsilon) + (I_n - K)\diag(K\hat\varepsilon).  
     \]  
  Here $1 \in \mathbb{R}^n$ denotes the vector of ones, 
  $\hat\varepsilon = (\hat\varepsilon_1,\dots,\hat\varepsilon_n)^\top \in \R^n$ is the vector of residuals,                                                      
  and $h = (h_{ij}) \in \R^{n \times n}$ is a matrix whose entries $h_{ij}$ are i.i.d.\ standard Gaussian random variables.                                      
  \end{proposition}   
    \begin{proof}     
  $\mathfrak{B}$ admits a decomposition with      
   \begin{align*}                                                                    \mathfrak{B} = \frac{1}{n}\sum_{i = 1}^n\sum_{j = 1}^n\left(\frac{\hat V_i - \hat V_j}{\sqrt{2}} \right)h_{ij} &= \underbrace{\frac{1}{n\sqrt{2}}\sum_{i =     1}^n\sum_{j = 1}^n\left(\hat T_{\mu,\lambda}^{-1}\hat S^*( \hat S_z+\mu)^{-1}\{\phi(Z_i)\hat \epsilon_i -\phi(Z_j)\hat \epsilon_j\}\right)h_{ij}}_A \\         
  &+ \underbrace{\frac{1}{n\sqrt{2}}\sum_{i = 1}^n\sum_{j = 1}^n\left(\hat T_{\mu,\lambda}^{-1}\{S_i^*-S_j^*\}(\hat S_z+\mu)^{-1}\E_n[\hat       
  \varepsilon\phi(Z)]\right)h_{ij}}_B  \\       
  &-\underbrace{\frac{1}{n\sqrt{2}}\sum_{i = 1}^n\sum_{j = 1}^n\left(\hat T_{\mu,\lambda}^{-1} \hat S^*(\hat S_z+\mu)^{-1}(S_{z,i}-S_{z,j})(\hat
  S_z+\mu)^{-1}\E_n[\hat \varepsilon\phi(Z)]\right)h_{ij}}_C .      
  \end{align*}    
  We repeatedly use the following identity: for all integers $n \ge 1$, all vectors  
   $(A_i)_{i=1}^n$ and all matrices entries $(h_{ij})_{i,j=1}^n$,  
  \begin{align*} 
  \sum_{i = 1}^n\sum_{j = 1}^n(A_i-A_j)h_{ij}   
      &= \sum_{i = 1}^n\sum_{j = 1}^nA_i h_{ij}  - \sum_{i = 1}^n\sum_{j = 1}^nA_j h_{ij} \\ 
     &= \sum_{i = 1}^n A_i \sum_{j = 1}^n h_{ij}  - \sum_{j = 1}^n A_j \sum_{i = 1}^n h_{ij} \\  
     &= \sum_{i = 1}^n A_i \sum_{j = 1}^n h_{ij}  - \sum_{i = 1}^n A_i \sum_{j = 1}^n h_{ji} \\   
     &= \sum_{i = 1}^n A_i \sum_{j = 1}^n (h_{ij} - h_{ji}).  
     \end{align*} 
  Now, for $A$ this means that    
  \begin{align*}   
     \frac{1}{\sqrt{2}}\sum_{i = 1}^n\sum_{j = 1}^n\{\phi(Z_i)\hat \epsilon_i -\phi(Z_j)\hat \epsilon_j\}h_{ij} &= \sum_{i = 1}^n\phi(Z_i)\hat \epsilon_i    
  \sum_{j = 1}^n\frac{(h_{ij}  -h_{ji})}{\sqrt{2}} = \Phi_Z^*\diag(\hat\varepsilon)c.
  \end{align*}  
  Substituting this into A, 
  \begin{align*} 
  A      &=\{\hat{S}^*(\hat{S}_z+\mu)^{-1}\hat{S}+\lambda\}^{-1}\hat{S}^* (\hat{S}_z+\mu)^{-1}\left(\frac{1}{n}\Phi_Z^*\diag(\hat\varepsilon)c\right) \\             
     &=\left\{\frac{1}{n}\Psi_X^*\Phi_Z \left(\frac{1}{n}\Phi_Z^*\Phi_Z+\mu\right)^{-1} \frac{1}{n}\Phi_Z^*\Psi_X+\lambda\right\}^{-1}\frac{1}{n}\Psi_X^*\Phi_Z  
  \left(\frac{1}{n}\Phi_Z^*\Phi_Z+\mu\right)^{-1} \left(\frac{1}{n}\Phi_Z^*\diag(\hat\varepsilon)c\right)  \\ 
    &=\left\{\Psi_X^*\Phi_Z \left(\Phi_Z^*\Phi_Z+n\mu\right)^{-1} \Phi_Z^*\Psi_X+n\lambda\right\}^{-1}\Psi_X^*\Phi_Z  \left(\Phi_Z^*\Phi_Z+n\mu\right)^{-1}\Phi_Z^*\diag(\hat\varepsilon)c \\
  &=\left(\Psi_X^*K\Psi_X+n\lambda\right)^{-1}\Psi_X^*K\diag(\hat\varepsilon)c \\
       &=\Psi_X^*\left(K\Psi_X\Psi_X^*+n\lambda\right)^{-1}K\diag(\hat\varepsilon)c \\
    &= \Psi_X^*\left(KK_{XX}+n\lambda\right)^{-1}K\diag(\hat\varepsilon)c.\end{align*}
For $B$, applying the double-sum identity yields
\begin{align*}    
  &\frac{1}{n\sqrt{2}}\sum_{i=1}^n\sum_{j=1}^n  
  \left(\hat T_{\mu,\lambda}^{-1}\{S_i^*-S_j^*\}(\hat S_z+\mu)^{-1}\E_n[\hat \varepsilon\phi(Z)]\right)h_{ij} \\  
  &= \frac{1}{n}\sum_{i=1}^n c_i
  \hat T_{\mu,\lambda}^{-1}S_i^*(\hat S_z+\mu)^{-1}\E_n[\hat \varepsilon\phi(Z)].  
  \end{align*}   
  Since $S_i^* = \psi(X_i)\otimes\phi(Z_i)^*$ and $(\hat S_z+\mu)^{-1}\E_n[\hat\varepsilon\phi(Z)] = (\hat S_z+\mu)^{-1}\frac{1}{n}\Phi_Z^*\hat\varepsilon$, we have                                           
  \[ 
  S_i^*(\hat S_z+\mu)^{-1}\E_n[\hat\varepsilon\phi(Z)] = \langle (\hat S_z+\mu)^{-1}\tfrac{1}{n}\Phi_Z^*\hat\varepsilon,\,\phi(Z_i)\rangle_{\Hz}\,\psi(X_i) = [K\hat\varepsilon]_i\,\psi(X_i),                 
  \]  
  where $[K\hat\varepsilon]_i$ denotes the $i$th entry of $K\hat\varepsilon\in\R^n$. Substituting,
  \begin{align*}
  B &= \frac{1}{n}\sum_{i=1}^n c_i\,[K\hat\varepsilon]_i\hat T_{\mu,\lambda}^{-1}\psi(X_i)  
  = \hat T_{\mu,\lambda}^{-1}\left(\frac{1}{n}\Psi_X^*\diag(c)K\hat\varepsilon\right)     
  = \Psi_X^*\left(KK_{XX}+n\lambda\right)^{-1}\diag(c)K\hat\varepsilon,  
  \end{align*}   
  where the last step uses $\hat T_{\mu,\lambda}^{-1}(\frac{1}{n}\Psi_X^*w) = \Psi_X^*(KK_{XX}+n\lambda)^{-1}w$ as derived for term $A$.                                                                       
  Lastly, for $C$, the double-sum identity gives  
  \begin{align*} 
  & \frac{1}{n\sqrt{2}}\sum_{i=1}^n\sum_{j=1}^n     
  \left(\hat T_{\mu,\lambda}^{-1}\hat S^*(\hat S_z+\mu)^{-1}(S_{z,i}-S_{z,j})(\hat S_z+\mu)^{-1}\E_n[\hat \varepsilon\phi(Z)]\right)h_{ij} \\                                                                  
  &= \frac{1}{n}\sum_{i=1}^n c_i   
  \hat T_{\mu,\lambda}^{-1}\hat S^*(\hat S_z+\mu)^{-1}S_{z,i}(\hat S_z+\mu)^{-1}\E_n[\hat \varepsilon\phi(Z)]. 
  \end{align*}    
  Since $S_{z,i}=\phi(Z_i)\otimes\phi(Z_i)^*$, we have       
  \[ S_{z,i}(\hat S_z+\mu)^{-1}\E_n[\hat\varepsilon\phi(Z)] = \langle (\hat S_z+\mu)^{-1}\tfrac{1}{n}\Phi_Z^*\hat\varepsilon,\phi(Z_i)\rangle_{\Hz}\phi(Z_i) = [K\hat\varepsilon]_i\phi(Z_i).               
  \]    
  It follows that 
  \begin{align*}                                       \hat S^*(\hat S_z+\mu)^{-1}\bigl([K\hat\varepsilon]_i\phi(Z_i)\bigr) &= [K\hat\varepsilon]_i\hat S^*(\hat S_z+\mu)^{-1}\phi(Z_i) \\ 
      &= [K\hat\varepsilon]_i\frac{1}{n}\Psi_X^*\Phi_Z\left(\tfrac{1}{n}\Phi_Z^*\Phi_Z+\mu\right)^{-1}\Phi_Z^*e_i \\
      &= [K\hat\varepsilon]_i\Psi_X^*\Phi_Z(\Phi_Z^*\Phi_Z+n\mu)^{-1}\Phi_Z^*e_i \\
      &= [K\hat\varepsilon]_i\Psi_X^*K_{\cdot,i}, \end{align*}
where $K_{\cdot,i}$ is the $i$th column of $K$, $e_i$ the $i$th standard basis vector, and the third equality uses $\frac{1}{n}\left(\frac{1}{n}\Phi_Z^*\Phi_Z+\mu\right)^{-1} = (\Phi_Z^*\Phi_Z+n\mu)^{-1}$. Therefore,          \begin{align*} C &= \frac{1}{n}\sum_{i=1}^n c_i[K\hat\varepsilon]_i\hat T_{\mu,\lambda}^{-1}\left(\Psi_X^*K_{\cdot,i}\right)   
            = \sum_{i=1}^n c_i[K\hat\varepsilon]_i\Psi_X^*(KK_{XX}+n\lambda)^{-1}K_{\cdot,i} \\   
      &= \Psi_X^*(KK_{XX}+n\lambda)^{-1}\left(\sum_{i=1}^n c_i[K\hat\varepsilon]_iK_{\cdot,i}\right) = \Psi_X^*\left(KK_{XX}+n\lambda\right)^{-1}K\diag(K\hat\varepsilon)c,   
      \end{align*}      
      where the second step uses $\hat T_{\mu,\lambda}^{-1}(\Psi_X^*w) = n\Psi_X^*(KK_{XX}+n\lambda)^{-1}w$ (as derived for term $A$), and $\sum_{i=1}^n c_i[K\hat\varepsilon]_i K_{\cdot,i} = K\diag(K\hat\varepsilon)c$. To conclude, using $\diag(c)K\hat\varepsilon = \diag(K\hat\varepsilon)c$, 
  \begin{align*}
  \mathfrak{B} &= A + B - C \\
  &= \Psi_X^*\left(KK_{XX}+n\lambda\right)^{-1}\left[K\diag(\hat\varepsilon)c + \diag(c)K\hat\varepsilon - K\diag(c)K\hat\varepsilon\right] \\ 
  &= \Psi_X^*\left(KK_{XX}+n\lambda\right)^{-1}\left[K\diag(\hat\varepsilon) + \diag(K\hat\varepsilon) - K\diag(K\hat\varepsilon)\right]c \\  
  &= \Psi_X^*\left(KK_{XX}+n\lambda\right)^{-1}\left[K\diag(\hat\varepsilon) + (I_n - K)\diag(K\hat\varepsilon)\right]c \\ 
  &= \Psi_X^*\left(KK_{XX}+n\lambda\right)^{-1}Dc,\end{align*}   and $\mathfrak{B}(x) = K_{xX}\left(KK_{XX}+n\lambda\right)^{-1}Dc$.  For $K=I_n$ we have $D = \diag(\hat\varepsilon)$, so $Dc = \diag(\hat\varepsilon)c = \beta$, which exactly equals the closed-form expression of \cite{singh2023kernel}. \qedhere  
  \end{proof}   

\section{Technical lemmas}\label{sec:technical}
We provide the necessary technical lemmas below. For completeness, we also include auxiliary results from prior work, retaining their original notation.
\subsection{Analysis}

\begin{lemma}[Higher-order resolvent, cf. \citealp{singh2023kernel}, Lemma E.10] \label{lemma:highorder}Let \( V \) be a vector space and \( A, B: V \to V \) be invertible linear operators. Then,  \( \forall \ell \geq 1 \), it holds
\[
A^{-1} - B^{-1} = A^{-1} \left\{ (B - A) B^{-1} \right\}^{\ell} + \sum_{r=1}^{\ell - 1} B^{-1} \left\{ (B - A) B^{-1} \right\}^{r}.
\]
When $\ell = 1$, this reduces to the resolvent identity
\[
A^{-1} - B^{-1} = A^{-1}(B - A)B^{-1} \iff A^{-1} = A^{-1}(B - A)B^{-1} + B^{-1}.
\]
\end{lemma}

\begin{lemma}[Non-commutative product rule]
\label{lemma:ABC}
    Let define $\Delta A = \hat A - A$, $\Delta 
    B= \hat B- B$, and $\Delta C = \hat C - C$. Then,
    $$\hat A \hat B \hat C = ABC + AB\Delta C + A\Delta B C + A \Delta B \Delta C + \Delta A B C + \Delta A B \Delta C + \Delta A \Delta B C + \Delta A  \Delta B \Delta C.
    $$
\end{lemma}
\begin{proof}
    \begin{align*}
\hat{A} \hat{B} \hat{C} 
&= (A + \Delta A)(B + \Delta B)(C + \Delta C) \\
&= A B C + A B \Delta C + A \Delta B C + A \Delta B \Delta C \\
&\quad + \Delta A B C + \Delta A B \Delta C + \Delta A \Delta B C + \Delta A \Delta B \Delta C. \qquad \qedhere
\end{align*}
\end{proof}
\begin{lemma}[Polar decomposition bound (cf.\ \citealp{de2005learning})]
\label{lemma:Polar}
Let $\mathcal{H}_1,\mathcal{H}_2$ be Hilbert spaces and let 
$A : \mathcal{H}_1 \to \mathcal{H}_2$ be a bounded linear operator. 
For all $\lambda > 0$, define $|A| := (A^*A)^{1/2}$. Then
\[
  \bigl\| A (A^*A+\lambda I)^{-1} \bigr\|_{\mathrm{op}}
  = \bigl\|\, |A| (|A|^2+\lambda I)^{-1} \bigr\|_{\mathrm{op}}
  = \sup_{t \in \mathrm{spec}(|A|)} \frac{t}{t^2 + \lambda}
  \le \frac{1}{2\lambda^{1/2}}.
\]
Moreover,
\[
  \bigl\|(A^*A+\lambda I)^{-1} A^*\bigr\|_{\mathrm{op}}
  = \bigl\| A (A^*A+\lambda I)^{-1} \bigr\|_{\mathrm{op}}.
\]
\end{lemma}
\begin{proof}
The first statment comes from writing $A$ with its polar decomposition $A = U|A|$, with $\|U\|_{\mathrm{op}} = 1$.  The second statement holds since the operator norm is invariant under taking the adjoint, i.e.
\[
\bigl\|(A^*A+\lambda I)^{-1}A^*\bigr\|_{\mathrm{op}}
= \left\|\bigl(A(A^*A+\lambda I)^{-1}\bigr)^*\right\|_{\mathrm{op}}
= \bigl\|A(A^*A+\lambda I)^{-1}\bigr\|_{\mathrm{op}}. \qquad \qedhere
\]
\end{proof}

\begin{lemma}[Generalized parallelogram law]
\label{lemma:parallelogram_law}
Let $\mathcal{H}$ be a Hilbert space. For all bounded linear operators
$A,B : \mathcal{H} \to \mathcal{H}$ it holds that
\[
  (A+B)(A+B)^* \preceq 2AA^* + 2BB^*,
\]
where $\preceq$ denotes the Loewner order on self-adjoint operators.
\end{lemma}
\begin{proof}
    It suffices to show that
    \begin{align*}
        AA^*+AB^*+BA^*+BB^* &\preceq 2AA^*+2BB^* \\
        \iff AB^*+BA^*      &\preceq AA^*+BB^* \\
        \iff 0              &\preceq (A-B)(A-B)^*.  \quad \qedhere
    \end{align*}
\end{proof}

\begin{lemma}[Young's inequality for scalars]\label{lemma:young}
For all $\bar a, \bar b \ge 0$ and all $t>0$ it holds that
\[
  2\bar a \bar b \le \frac{\bar a^2}{t} + t \bar b^2.
\]
  In particular, let $C \ge 1$. If $a, b$ are real numbers satisfying                    
  \[  |a| \le \bar a, \quad \underline b \le |b| \le \bar b,  \quad \bar b^2 \le C\,\underline b^2,  \quad\text{and}\quad \bar a^2 \le \frac{1}{8C^2}\,\bar b^2,    \] then
  \[(a+b)^2 \ge \frac{1}{4C}\,\bar b^2.\]                                                  
\end{lemma}
\begin{proof}
For the first claim, observe that $(\sqrt{t}\bar{b}-\bar{a}/\sqrt{t})^2\ge 0$, i.e.
\[
t\bar{b}^2+\frac{\bar{a}^2}{t}-2\bar{a}\bar{b}\ge 0
\quad\Rightarrow\quad
2\bar{a}\bar{b}\le \frac{\bar{a}^2}{t}+t\bar{b}^2.
\]
For the “in particular” part, start with
\[
(a+b)^2 = a^2+b^2+2ab \ge a^2+b^2-2|a||b|
\ge \underline{b}^2-2\bar{a}\bar{b},
\]
where we used $|a|\le \bar{a}$ and $|b|\ge \underline{b}$ (and $a^2\ge 0$).
By Young's inequality (applied to $\bar{a},\bar{b}$),
\[
\underline{b}^2-2\bar{a}\bar{b}
\ge \underline{b}^2-\frac{\bar{a}^2}{t}-t\bar{b}^2.
\]
Using $\underline{b}^2\ge \bar{b}^2/C$,
\[
\underline{b}^2-\frac{\bar{a}^2}{t}-t\bar{b}^2
\ge \Bigl(\frac{1}{C}-t\Bigr)\bar{b}^2-\frac{\bar{a}^2}{t}.
\]
Choose $t=\frac{1}{2C}$ to get
\[
\Bigl(\frac{1}{C}-t\Bigr)\bar{b}^2-\frac{\bar{a}^2}{t}
= \frac{1}{2C}\bar{b}^2-2C\bar{a}^2.
\]
Finally, $\bar{a}^2\le \frac{1}{8C^2}\bar{b}^2$ implies
$2C\bar{a}^2\le \frac{1}{4C}\bar{b}^2$, hence
\[
(a+b)^2 \ge \frac{1}{2C}\bar{b}^2-\frac{1}{4C}\bar{b}^2
= \frac{1}{4C}\bar{b}^2. \quad \qedhere
\]
\end{proof}

\begin{lemma}[Young's inequality for operators]
\label{lemma:Young_Operators}
Let $\mathcal{H}$ be a Hilbert space and let $A,B : \mathcal{H} \to \mathcal{H}$ be bounded linear operators.
Let $\underline{\Sigma}_A,\bar{\Sigma}_A,\underline{\Sigma}_B,\bar{\Sigma}_B$ be bounded self-adjoint operators on $\mathcal{H}$ and let $C \ge 1$.
Suppose that
\begin{enumerate}
    \item $0 = \underline{\Sigma}_A \preceq AA^* \preceq \bar{\Sigma}_A$;
    \item $\underline{\Sigma}_B \preceq BB^* \preceq \bar{\Sigma}_B \preceq C\,\underline{\Sigma}_B$;
    \item $\bar{\Sigma}_A \preceq \tfrac{1}{8C^2}\,\bar{\Sigma}_B$.
\end{enumerate}
Then
\[
  (A+B)(A+B)^* \succeq \frac{1}{4C}\,\bar{\Sigma}_B.
\]
\end{lemma}

\begin{proof}
Fix $x$ in the Hilbert space. Then,
\begin{align*}
\|(A+B)^*x\|^2
&= \|A^*x\|^2+\|B^*x\|^2+2\langle A^*x,B^*x\rangle \\
&\ge \|A^*x\|^2+\|B^*x\|^2-2\|A^*x\|\|B^*x\|
\\
&\ge \|B^*x\|^2-2\|A^*x\|\|B^*x\|.
\end{align*}
Define the scalars
\[
\bar{a}^2:=\langle \bar{\Sigma}_A x,x\rangle,\qquad
\underline{b}^2:=\langle \underline{\Sigma}_B x,x\rangle,\qquad
\bar{b}^2:=\langle \bar{\Sigma}_B x,x\rangle.
\]
By assumptions (1) and (2),
\[
\|A^*x\|^2=\langle AA^*x,x\rangle\le \bar{a}^2,\qquad
\underline{b}^2\le \|B^*x\|^2=\langle BB^*x,x\rangle\le \bar{b}^2.
\]
Hence
\[
\|(A+B)^*x\|^2 \ge \underline{b}^2-2\bar{a}\bar{b}.
\]
Assumption (2) gives $\bar{b}^2\le C\underline{b}^2$, and (3) gives
$\bar{a}^2\le \tfrac{1}{8C^2}\bar{b}^2$. Thus Lemma~\ref{lemma:young} applies to
$\bar{a},\underline{b},\bar{b}$ and yields
\[
\|(A+B)^*x\|^2 \ge \frac{1}{4C}\bar{b}^2
= \frac{1}{4C}\langle \bar{\Sigma}_B x,x\rangle.
\]
Equivalently,
\[
\langle\bigl[(A+B)(A+B)^*-\tfrac{1}{4C}\bar{\Sigma}_B\bigr]x,x\rangle \ge 0
\quad\text{for all }x,
\]
which is precisely $(A+B)(A+B)^*\succeq \tfrac{1}{4C}\bar{\Sigma}_B$.
\end{proof}

\subsection{Probability}

\begin{lemma}[Bernstein inequality (Proposition 2 of \citealp{caponnetto2007optimal})]
\label{lemma:Bernstein} Suppose that $\xi_i$ are i.i.d. random elements of a Hilbert space, which $ \forall \ell \geq 2$ satisfy 
\[
\mathbb{E} \|\xi_i - \mathbb{E}(\xi_i)\|^\ell \leq \frac{1}{2} \ell! B^2 \left(\frac{A}{2}\right)^{\ell-2}.
\]
Then, $ \forall \eta \in (0,1)$ it holds with probability at least $1 - \eta$ that
\[
\left\| \frac{1}{n} \sum_{i=1}^n \xi_i - \mathbb{E}(\xi_i) \right\| 
\leq 2 \left\{ \sqrt{ \frac{B^2 \log(2/\eta)}{n} } + \frac{A \log(2/\eta)}{n} \right\} 
\leq 2 \log(2/\eta) \left( \frac{A}{n} + \sqrt{ \frac{B^2}{n} } \right).
\]
In particular, this holds if $\mathbb{E}(\|\xi_i\|^2) \leq B^2$ and $\|\xi_i\| \leq A/2$ almost surely.

\end{lemma}

\begin{lemma}[Borell’s inequality, Theorem 2.5.8 of \citealp{gine2021mathematical} ] \label{lemma:Borell}
Let $G_t$ be a centered Gaussian process, a.s.\ bounded on $T$. Then for $\forall u>0$,
\[
\mathbb{P}\!\left( \sup_{t \in T} G_t - \mathbb{E}\sup_{t \in T} G_t > u \right)
\vee
\mathbb{P}\!\left( \sup_{t \in T} G_t - \mathbb{E}\sup_{t \in T} G_t < -u \right)
\leq
\exp\!\left( -\frac{u^2}{2\sigma_T^2} \right),
\]
where $\sigma_T^2 = \sup_{t \in T} \mathbb{E} G_t^2$.
\end{lemma}

\begin{lemma}[Gaussian norm bound, cf.~\citealp{singh2023kernel}]
\label{lemma:Gaussianborel}
Let $Z$ be a Gaussian random element in a Hilbert space $H$ such that 
$\mathbb{E}\|Z\|^2 < \infty$. Then, with probability $1-\eta$,
\[
\|Z\| \leq \left\{ 1 + \sqrt{2 \log(1/\eta)} \right\} 
\sqrt{\mathbb{E}\|Z\|^2}.
\]
In particular, if $A:H \to H$ is a trace-class operator, then, with probability 
$1-\eta$ with respect to $g$
\[
\|Ag\| \leq \left\{ 1 + \sqrt{2 \log(1/\eta)} \right\} \|A\|.
\]
\end{lemma}

\section{Bias upper bound}

To lighten notation, let $h_{\lambda}:=h_{0,\lambda}$ and $T_{\lambda}:=T_{0,\lambda}$. For all $\nu \in [0,1]$, define the spectral constant $C_\nu := (1-\nu)^{1-\nu}\nu^\nu$.

\subsection{Regression bias}
\begin{lemma}[General operator norm bound]\label{lem:general_bounds}
    Let $A$ be a positive semidefinite, self-adjoint operator on a Hilbert space $\mathcal{H}$. For any scalar $\gamma > 0$ and exponent $\nu \in [0,1]$
    \[
    \bigl\|\gamma A^{\nu}(A+\gamma I)^{-1}\bigr\|_{\mathrm{op}}
    \;\le\; C_\nu\,\gamma^{\nu}.
    \]
\end{lemma}

\begin{proof}
By the spectral theorem, the operator norm is equal to the supremum of the corresponding scalar function over the spectrum of $A$. Let $\sigma(A) \subseteq [0, \infty)$ denote the spectrum. Then,
\begin{align*}
\|\gamma A^{\nu} (A + \gamma I)^{-1}\|_{\mathrm{op}}    
    &= \sup_{t \in \sigma(A)} \frac{\gamma t^{\nu} }{t + \gamma} \le \sup_{t \ge 0} \frac{\gamma t^{\nu} }{t + \gamma}.
\end{align*}
We define the function $f(t) := \frac{\gamma t^\nu}{t + \gamma}$. Setting the derivative $f'(t) = 0$ yields the maximizer $t^* = \frac{\nu}{1-\nu}\gamma$. Evaluating the function at this maximum gives:
\[
f(t^*) = \frac{\gamma \left(\frac{\nu}{1-\nu}\gamma\right)^\nu}{\frac{\nu}{1-\nu}\gamma + \gamma} 
= \gamma^\nu \left[ (1-\nu)^{1-\nu}\nu^\nu \right] 
= C_\nu \gamma^\nu.
\]
Thus, the operator norm is bounded by $C_\nu \gamma^\nu$.
\end{proof}

\begin{lemma}[Regression bias bound] 
\label{lemma:First_stage_bias}

Suppose Assumption \ref{ass:Source} holds. Then, 
    $$ \|h_{\lambda}-h_0\|_\Hx \leq C_\alpha \lambda^{\alpha} \|T^{-\alpha}h_0\|_\Hx.$$
\end{lemma}

\begin{proof}
Let $T = S^*S_z^{-1}S$. Under Assumption \ref{ass:Source}, there exists $w_0$ such that $h_0 = T^\alpha w_0$.
The bias can be written as:
\[
    h_{\lambda}-h_0 = -\lambda(T+\lambda)^{-1}h_0 = -\lambda(T+\lambda)^{-1}T^\alpha w_0.
\]
Taking the norm, we have:
\begin{align*}
    \|h_{\lambda}-h_0\|_\Hx 
    &= \|\lambda T^\alpha (T+\lambda)^{-1} w_0\|_\Hx \\
    &\le \|\lambda T^\alpha (T+\lambda)^{-1}\|_{\mathrm{op}} \|w_0\|_\Hx.
\end{align*}
Lemma \ref{lem:general_bounds} demonstrated that 
\[
    \|\lambda T^\alpha (T+\lambda)^{-1}\|_{\mathrm{op}} \le C_\alpha \lambda^\alpha.
\]
Combining these gives the result directly.
\end{proof}
\subsection{Instrumental variable regression bias}
\begin{lemma}[Decomposition]\label{lem:decomp}
Define $b_\lambda := h_0 - h_\lambda$. Then, for all integers $k \ge 1$,
\[h_0 - T_{\mu,\lambda}^{-1}T_\mu h_0 = T_{\mu,\lambda}^{-1}(T-T_\mu)[T_\lambda^{-1}(T - T_\mu)]^{k-1}b_\lambda + \sum_{l=0}^{k-1}[T_\lambda^{-1}(T - T_\mu)]^{l}b_\lambda.\]
Furthermore, for all $\mu > 0$,
$T - T_\mu = S^*(S_z^{-1} - (S_z + \mu)^{-1})S = \mu S^*S_z^{-1}  (S_z + \mu)^{-1}S.$
\end{lemma}
\begin{proof}
    We begin by writing
    \[h_0 - T_{\mu,\lambda}^{-1}T_\mu h_0 = T_{\mu,\lambda}^{-1}(T_{\mu,\lambda} - T_\mu)h_0 = \lambda T_{\mu,\lambda}^{-1}h_0.\]
    By the resolvent identity $A^{-1} = B^{-1} + A^{-1}(B-A)B^{-1}$ with $A = T_{\mu,\lambda}$ and $B = T_\lambda$, we further obtain
    \begin{align*}
        \lambda T_{\mu,\lambda}^{-1}h_0
        &= \lambda T_\lambda^{-1}h_0 + T_{\mu,\lambda}^{-1}(T_\lambda - T_{\mu,\lambda})(\lambda T_\lambda^{-1}h_0) \\
        &= b_\lambda + T_{\mu,\lambda}^{-1}(T - T_\mu)b_\lambda
    \end{align*}
    using $\lambda T_\lambda^{-1} h_0 = b_\lambda$ and $T_\lambda-T_{\mu,\lambda} = T - T_\mu$. This proves our claim with $k = 1$. To complete the proof by induction, note that the claim for $k+1$ follows from the claim for $k$ by applying the resolvent identity to substitute
    \[T_{\mu,\lambda}^{-1} = T_\lambda^{-1} + T_{\mu,\lambda}^{-1}(T - T_\mu)T_\lambda^{-1}.\]
The last statement follows from the resolvent identity with $l = 1.$
\end{proof}

\begin{lemma}[Bias upper bound]\label{prop:general-bound} Suppose that Assumptions \ref{ass:Source} and \ref{assn:restricted-link} hold for $(\alpha, \beta,r)$, and $r\mu^\beta / \lambda^{1/2} < 1$. Then, 
\[\bml : = \|T_{\mu,\lambda}^{-1}T_\mu h_0 - h_0\|_{\Hx}  \le \frac{\|h_\lambda - h_0\|_{\Hx}}{1 - C_\beta r\mu^{\beta}/\lambda^{\frac{1}{2}}} \le \frac{C_\alpha \lambda^{\alpha} \|T^{-\alpha} h_0\|_{\Hx}}{1 - C_\beta r \mu^{\beta}/\lambda^{\frac{1}{2}}} .\]
\end{lemma}
\proof[Proof]{
We begin by writing that
\[T^{-1/2}S^*S_z^{-1}ST^{-1/2} = T^{-1/2}TT^{-1/2} = I,\]
so $ T^{-1/2}S^*S_z^{-1/2}$ is a unitary operator with norm one. 
Lemma demonstrated that we may substitute
\begin{align*}
\|T_\lambda^{-1}(T_\mu-T)x\|_\Hx 
&=\|T_\lambda^{-1}S^*[\mu S_z^{-1} (S_z + \mu)^{-1}]Sx\|_\Hx  \\
&=\|[T_\lambda^{-1}T^{1/2}][T^{-1/2}S^*S_z^{-1/2}][\mu S_z^{\beta} (S_z + \mu)^{-1}][S_z^{-1/2-\beta}S]x\|_\Hx  \\
&\le \|T_\lambda^{-1}T^{1/2}\|_{\mathrm{op}}\|T^{-1/2}S^*S_z^{-1/2}\|_{\mathrm{op}}\|\mu S_z^{\beta} (S_z + \mu)^{-1}\|_{\mathrm{op}}\|S_z^{-1/2-\beta}S\|_{\mathrm{op}}\|x\|_\Hx .
\end{align*}
The first norm is bounded by $\lambda^{-1/2}$ by construction. The second norm is bounded by $1$ from our earlier observation. The third is bounded by $C_\beta\mu^\beta$ as demonstrated by Lemma \ref{lem:general_bounds}. The fourth is bounded by $r$ by Assumption \ref{assn:restricted-link}.
Thus, we obtain
\[\|T_\lambda^{-1}(T_\mu-T)\|_{\mathrm{op}} \le \frac{r\mu^\beta}{\sqrt{\lambda}}.\]
Thus,
\[\|[T_\lambda^{-1}(T_\mu-T)]^lb_\lambda\|_{\Hx} \le \left( \frac{r\mu^\beta}{\sqrt{\lambda}}\right)^l\|b_\lambda\|_{\Hx}\]
and, using $\|T_\mu - T\|_{\mathrm{op}} \le \|T\|_{\mathrm{op}} \le \kappa_x^2$ and $\|T_{\mu,\lambda}^{-1}\|_{\mathrm{op}} \le \lambda^{-1}$, we also have that
\[\|T_{\mu,\lambda}^{-1}(T_\mu-T)[T_\lambda^{-1}(T_\mu-T)]^lb_\lambda\|_{\Hx} \le \frac{\kappa_x^2}{\lambda}\left( \frac{r\mu^\beta}{\sqrt{\lambda}}\right)^l\|b_\lambda\|_{\Hx}.\]
Plugging these results into our decomposition of Lemma \ref{lem:decomp}, we have for any $k \ge 1$ that 
\[\|h_0 - T_{\mu,\lambda}^{-1}T_\mu h_0\|_\Hx \le \left[\frac{\kappa_x^2}{\lambda}\delta^{k-1} +  \sum_{l=0}^{k-1}\delta^l\right]\|b_\lambda\|_\Hx; \quad \delta = \left( \frac{r\mu^\beta}{\sqrt{\lambda}}\right).\]
Since $\delta < 1$, we recover the bound 
\[\|h_0 - T_{\mu,\lambda}^{-1}T_\mu h_0\|_\Hx \le \lim_{k \uparrow \infty} \left[\frac{\kappa_x^2}{\lambda}\delta^{k-1} +  \sum_{l=0}^{k-1}\delta^l\right]\|b_\lambda\|_\Hx = \left(\frac{1}{1-\delta}\right)\|b_\lambda\|_\Hx.\]
Finally, we use the standard regression bias bound (Lemma \ref{lemma:First_stage_bias}) to bound $\|b_\lambda\|_\Hx$.
}

\begin{lemma}[Projected bias]\label{lemma:pj_bias}
    Assume that Assumptions \ref{ass:Source} and \ref{assn:restricted-link} hold with $(\alpha,\beta,r)$.
    Let $\lambda, \mu > 0$ be such that
    $
      \delta = \frac{r \mu^\beta}{\sqrt{\lambda}} < 1.
    $
    Then,
    \[
      \bigl\|(S_z+\mu)^{-1/2} S\bigl(h_0 - T_{\mu,\lambda}^{-1} T_\mu h_0\bigr)\bigr\|_{\Hz}
      \;\le\; \frac{1}{1 - \delta}\, K_\alpha\, \lambda^{(\alpha + 1/2)\wedge 1},
    \]
  where $K_\alpha = \sqrt{C_{2\alpha}}\|T^{-\alpha}h_0\|_\Hx$ if $\alpha \le 1/2$, and $K_\alpha = \kappa_x^{2\alpha - 1} \|T^{-\alpha}h_0\|_\Hx$ if $\alpha > 1/2$.
\end{lemma}
\begin{proof}
    Define the bias element $b_{\lambda}\coloneqq h_{\lambda}-h_{0}=-\lambda T_{\lambda}^{-1}h_{0}$ and $
    A\coloneqq(S_{z}+\mu I)^{-1/2}S: \Hx \to \Hz.$  We use that 
    \begin{equation*}
    \norm{AT_{\lambda}^{-1/2}}_{\mathrm{op}} \le \norm{S_z^{-1/2}ST_{\lambda}^{-1/2}}_{\mathrm{op}} \le 1,
   \quad \text{and}
    \norm{AT_{\mu,\lambda}^{-1}}_{\mathrm{op}}\le \frac{1}{2\sqrt{\lambda}}.
    \end{equation*}
    For all integers $l\ge 0$,
    \begin{equation*}\label{eq:similarity}
    \big[T_{\lambda}^{-1}(T_{\mu}-T)\big]^{l}
    =
    T_{\lambda}^{-1/2}\Big(T_{\lambda}^{-1/2}(T_{\mu}-T)T_{\lambda}^{-1/2}\Big)^{l}T_{\lambda}^{1/2}.
    \end{equation*}
    Hence, for all $l\ge 0$,
    \begin{align*}
    \norm{A\big[T_{\lambda}^{-1}(T_{\mu}-T)\big]^{l}b_{\lambda}}_\Hz
    &\le
    \underbrace{\norm{AT_{\lambda}^{-1/2}}_{\mathrm{op}}{}}_{\le1}
    \underbrace{\norm{T_{\lambda}^{-1/2}(T_{\mu}-T)T_{\lambda}^{-1/2}}_{\mathrm{op}}^{l}}_{=\delta^{l}}
    \norm{T_{\lambda}^{1/2}b_{\lambda}}_\Hx
    \notag= \delta^{l}\norm{T_{\lambda}^{1/2}b_{\lambda}}_\Hx,
    \end{align*}
    where we set
    \begin{align*}
    \delta\ &= \norm{T_{\lambda}^{-1/2}(T_{\mu}-T)T_{\lambda}^{-1/2}}_{\mathrm{op}}\\
    & = \|T_{\lambda}^{-1/2} S^* S_z^{-1}\mu \Szmu S T_{\lambda}^{-1/2}\|_{\mathrm{op}}\leq \frac{1}{\sqrt{\lambda}}\| T_{\lambda}^{-1/2}S^* S_z^{-1/2 }\mu S_z^\beta \Szmu S_z^{-1/2 - \beta }S \|_{\mathrm{op}}\\
    &\leq \frac{r\mu^\beta}{\sqrt{\lambda}},
    \end{align*}
    which is smaller than $1$ by hypothesis. Lastly, for all $\alpha \in[0,1/2]$, we have
    \begin{align*}
        \norm{T_{\lambda}^{1/2}b_{\lambda}}_\Hx
        = \lambda \|T_{\lambda}^{-1/2} h_0\|_\Hx
        = \lambda \|T_{\lambda}^{-1/2}T^{\alpha} w_0\|_\Hx
        \le \sqrt{C_{2\alpha}} \lambda^{\alpha + 1/2} \|w_0\|_\Hx.
    \end{align*}
    In contrast, for all $\alpha \in (1/2,1]$,
    \begin{align*}
        \norm{T_{\lambda}^{1/2}b_{\lambda}}_\Hx = \lambda \|T_{\lambda}^{-1/2} h_0\|_\Hx = \lambda \|T_{\lambda}^{-1/2} T^\alpha w_0\|_\Hx \le \lambda \kappa_x^{2\alpha - 1} \|w_0\|_\Hx.
    \end{align*}
    According to Lemma \ref{lem:decomp}, for all integers $k\ge 1$,
    \[
    A\bigl(h_0 - T_{\mu,\lambda}^{-1}T_\mu h_0\bigr)
    = AT_{\mu,\lambda}^{-1}(T_\mu - T)\bigl[T_\lambda^{-1}(T_\mu - T)\bigr]^{k-1} b_\lambda
    +
    \sum_{l=0}^{k-1} A\bigl[T_\lambda^{-1}(T_\mu - T)\bigr]^{l} b_\lambda.
    \]
    Furthermore, for all integers $k \ge 1$,
    \begin{align*}
      \bigl\| A(h_0 - T_{\mu,\lambda}^{-1}T_\mu h_0)\bigr\|_{\Hz}
      &\le \frac{1}{2\sqrt{\lambda}} \delta^{k-1} \,\|T_\lambda^{1/2} b_\lambda\|_{\Hx}
         + \sum_{l=0}^{k-1} \delta^l \,\|T_\lambda^{1/2} b_\lambda\|_{\Hx} \\
      &= \left(\frac{1}{2\sqrt{\lambda}}\delta^{k-1} + \sum_{l=0}^{k-1}\delta^l\right)
         \|T_\lambda^{1/2} b_\lambda\|_{\Hx}.
    \end{align*}
    In particular, for $k \to \infty $,
    \[
        \| A\bigl(h_0 - T_{\mu,\lambda}^{-1}T_\mu h_0\bigr)\|_\Hz\leq \frac{1}{1 - \delta} K_\alpha \lambda^{(\alpha + 1/2)\wedge1 }. \quad\qedhere
    \]
\end{proof}
 \section{Matching symbols}

\subsection{Covariance upper bound}
We introduce further notation to simplify the analysis. Let $u := h_0-h_{\mu, \lambda} \in \Hx$ and define the scalar evaluation $u(X_i) := \psi(X_i)^*u$. Consequently, we can write
    $S_i(h_0-h_{\mu, \lambda}) = \{h_0(X_i)-h_{\mu, \lambda}(X_i)\}\phi(Z_i) = u(X_i)\phi(Z_i).$
Since $h_0-h_{\mu, \lambda}=(I-T_{\mu, \lambda}^{-1}T_\mu)h_0$ and the operator norm satisfies $\|I-T_{\mu, \lambda}^{-1}T_\mu\|_{\op} \le 1$, the scalar values are bounded by
    $|u(X_i)| = |h_0(X_i)-h_{\mu, \lambda}(X_i)| \leq \kappa_x\|h_0-h_{\mu, \lambda}\|_\Hx =: \bar{u},$
where the bound for $\|h_0-h_{\mu, \lambda}\|_\Hx$ is established in Lemma \ref{prop:general-bound}. Additionally, define $v := (S_z+\mu)^{-1/2}Su \in \Hz$ and $w := (S_z+\mu)^{-1/2}v \in \Hz$. These elements satisfy
$
S_i^*(S_z+\mu)^{-1}S\hvec = \psi(X_i)w(Z_i)
$
and
$
S_{z,i}(S_z+\mu)^{-1}S\hvec = \phi(Z_i)w(Z_i)
$. The norm of $v$ is bounded by $\bar{v}$ as shown in Lemma \ref{lemma:pj_bias}. Furthermore, the evaluations of $w$ satisfy $|w(Z_i)|\leq \kappa_z\|w\|_\Hz=:\bar{w}.$
 Finally, the norms of $w$ and $v$ are related by $\|w\|_\Hz \leq \|v\|_\Hz\frac{1}{\sqrt{\mu}}.$ We abbreviate $l(\eta) = \log(2/\eta)$.
\begin{lemma}[Covariance upper bound]\label{lemma:cov_upper}
The covariance $\Sigma = \E(U_i \otimes U_i^*)$ is upper bounded by 
     \[\Sigma \preceq(6\bar{w}^2 + 6\bar{u}^2 +2\bar \sigma^2)\Tml^{-1} S^* \Szmu S_z \Szmu S \Tml^{-1}  + 6\bar{v}^2 \kappa_x^2\Tml^{-2}.\]
\end{lemma}
\begin{proof}
Write $U_i=U_{i1}+U_{i2}$, where
\begin{align*}
U_{i1} &= T_{\mu,\lambda}^{-1}\{S^*(S_z+\mu)^{-1}(S_i-S)+(S_i-S)^*(S_z+\mu)^{-1}S + S^*(S_z + \mu)^{-1}(S_z - S_{z,i})(S_z + \mu)^{-1}S\}(h_0-h_{\mu,\lambda}), \\
U_{i2} &=T_{\mu,\lambda}^{-1}S^*(S_z+\mu)^{-1}\phi(Z_i)\ep_i.
\end{align*}
Then, $\Sigma=\Sigma_{11}+\Sigma_{12}+\Sigma_{21}+\Sigma_{22}$ where $\Sigma_{\ell m}=\E (U_{i \ell}\otimes U_{im}^*)$. In particular by Lemma \ref{lemma:parallelogram_law},
$
\Sigma\preceq 2(\Sigma_{11}+\Sigma_{22}).
$ Consider the parts individually.
\begin{enumerate}
\item $\Sigma_{11}:$ Define
\begin{align*}
U_{i11} &= \Tml^{-1} S^* \Szmu \dSi \hvec ,\\
U_{i12} &=  \Tml^{-1} \dSi^{*} \Szmu S \hvec ,\\
U_{i13} &= \Tml^{-1} S^* \Szmu \dSzi \Szmu S \hvec .
\end{align*}
Now, $\Sigma_{11} = \sum_{i=1}^{3}\sum_{j=1}^{3} \Sigma_{1ij}$
where $\Sigma_{1ij} = \E (U_{1ij}\otimes U_{1ij}^*)$ and 
\begin{align*}
\Sigma_{111} &= \mathbb{E}\Big(\{\Tml^{-1} S^* \Szmu \dSi \hvec\}\Big)
  \otimes
  \Big(\{\Tml^{-1} S^* \Szmu \dSi \hvec\}\Big)^{*} \\
\Sigma_{112} &
= \mathbb{E}\Big(\{\Tml^{-1} S^* \Szmu \dSi \hvec\}\Big)
  \otimes
  \Big(\{\Tml^{-1} \dSi^{*} \Szmu S \hvec\}\Big)^{*} \\
\Sigma_{113} &
= \mathbb{E}\Big(\{\Tml^{-1} S^* \Szmu \dSi \hvec\}\Big)
  \otimes
  \Big(\{\Tml^{-1} S^* \Szmu \dSzi \Szmu S \hvec\}\Big)^{*} \\
\Sigma_{121} &
= \mathbb{E}\Big(\{\Tml^{-1} \dSi^{*} \Szmu S \hvec\}\Big)
  \otimes
  \Big(\{\Tml^{-1} S^* \Szmu \dSi \hvec\}\Big)^{*} \\
\Sigma_{122} &
= \mathbb{E}\Big(\{\Tml^{-1} \dSi^{*} \Szmu S \hvec\}\Big)
  \otimes
  \Big(\{\Tml^{-1} \dSi^{*} \Szmu S \hvec\}\Big)^{*} \\
\Sigma_{123} &=
 \mathbb{E}\Big(\{\Tml^{-1} \dSi^{*} \Szmu S \hvec\}\Big)
  \otimes
  \Big(\{\Tml^{-1} S^* \Szmu \dSzi \Szmu S \hvec\}\Big)^{*} \\
\Sigma_{131} &= \mathbb{E}\Big(\{\Tml^{-1} S^* \Szmu \dSzi \Szmu S \hvec\}\Big)
  \otimes
  \Big(\{\Tml^{-1} S^* \Szmu \dSi \hvec\}\Big)^{*} \\
\Sigma_{132} &= \mathbb{E}\Big(\{\Tml^{-1} S^* \Szmu \dSzi \Szmu S \hvec\}\Big)
  \otimes
  \Big(\{\Tml^{-1} \dSi^{*} \Szmu S \hvec\}\Big)^{*} \\
\Sigma_{133} &=\mathbb{E}\Big(\{\Tml^{-1} S^* \Szmu \dSzi \Szmu S \hvec\}\Big)\\
  &\quad\otimes 
  \Big(\{\Tml^{-1} S^* \Szmu \dSzi \Szmu S \hvec\}\Big)^{*}.
\end{align*}
We have that $\Sigma_{11} \preceq 3 (\Sigma_{111} +\Sigma_{122} + \Sigma_{133}).$
Thus, looking at the individual terms,
\begin{align*}
\Sigma_{111}&\preceq \E \Big(\{\Tml^{-1} S^* \Szmu S_i \hvec\}\Big)
  \otimes
  \Big(\{\Tml^{-1} S^* \Szmu S_i \hvec\}\Big)^{*} \\
  &=  \E\Big(\Tml^{-1} S^* \Szmu u(X_i)\phi(Z_i)\phi(Z_i)^*u(X_i)\Szmu S \Tml^{-1} \Big )
 \\ 
 &\preceq \bar{u}^2\Tml^{-1} S^* \Szmu S_z\Szmu S \Tml^{-1}. 
\end{align*}
The other two diagonal terms require more care. First,
\begin{align*}
\Sigma_{122} & \preceq\E\Big(\{\Tml^{-1} S_i^{*} \Szmu S \hvec\}\Big)
  \otimes
  \Big(\{\Tml^{-1} S_i^{*} \Szmu S \hvec\}\Big)^{*}\\
  &= \E \Big(\Tml^{-1}\psi(X_i)w(Z_i)^2 \psi(X_i)^*\Tml^{-1}\Big).
\end{align*}
Recall that, for general $u,v$ $$\bk{u}{S_z v}_{\Hz} = \mathbb{E}_Z[u(Z)v(Z)].$$
Then,
\begin{align*}
    &\hphantom{=} \E \Big(\Tml^{-1}\psi(X_i)w(Z_i)^2 \psi(X_i)^*\Tml^{-1}\Big) =  \E_Z\left(w(Z_i)^2 \E\left\{\Tml^{-1}\psi(X_i) \otimes \psi(X_i)^*\Tml^{-1}| Z_i\right\}\right ) \\
    &\preceq \E_Z\left(\bk{\phi(Z_i)}{w}^2_{\Hz}\right ) \Tml^{-2}\kappa_x^2   = \bk{\Szmu S \hvec}{S_z\Szmu S \hvec}_{\Hz}\Tml^{-2}\kappa_x^2 \\
       & = \bk{(S_z+\mu)^{-1/2} S \hvec}{(S_z+\mu)^{-1/2}S_z(S_z+\mu)^{-1/2}(S_z+\mu)^{-1/2} S \hvec}_{\Hz}\Tml^{-2}\kappa_x^2 \\
       &\preceq \bk{(S_z+\mu)^{-1/2} S \hvec}{(S_z+\mu)^{-1/2} S \hvec}_{\Hz}\Tml^{-2}\kappa_x^2 \\
       &= \|(S_z+\mu)^{-1/2} S \hvec\|^2_{\Hz}\Tml^{-2}\kappa_x^2 \preceq \bar v^2\Tml^{-2}\kappa_x^2.
\end{align*}  
Thus, $\Sigma_{122} \preceq \bar{v}^2 \kappa_x^2\Tml^{-2}$. Lastly, for  $\Sigma_{133}$ we can write
\begin{align*}
    \Sigma_{133} & \preceq \E\Big(\{\Tml^{-1} S^* \Szmu S_{z,i} \Szmu S \hvec\}\Big)
  \otimes
  \Big(\{\Tml^{-1} S^* \Szmu S_{z,i} \Szmu S \hvec\}\Big)^{*} \\
  &= \E\Big(\{\Tml^{-1} S^* \Szmu S_{z,i} w\}\Big)
  \otimes
  \Big(\{\Tml^{-1} S^* \Szmu S_{z,i} w\}\Big)^{*} \\
    &= \E\Big(\{\Tml^{-1} S^* \Szmu  \phi(Z_i) w(Z_i)\}\Big)
  \otimes
  \Big(\{\Tml^{-1} S^* \Szmu \phi(Z_i) w(Z_i)\}\Big)^{*} \\
  &\preceq  \bar{w}^2 \Tml^{-1} S^* \Szmu S_z\Szmu S \Tml^{-1} .
\end{align*}
We can now combine the three diagonal terms, yielding 
\begin{align*}
\Sigma_{11}&\preceq 3\bar{u}^2 \Tml^{-1} S^* \Szmu S_z \Szmu S \Tml^{-1} + 3\bar{w}^2 \Tml^{-1} S^* \Szmu S_z\Szmu S \Tml^{-1}  + 3 \bar{v}^2 \kappa_x^2\Tml^{-2} \\
&\preceq 3(\bar{u}^2 +\bar{w}^2) \Tml^{-1} S^* \Szmu S_z \Szmu S \Tml^{-1} + 3 \bar{v}^2 \kappa_x^2\Tml^{-2}.
\end{align*}
\item For $\Sigma_{22}$, we use that $\ep_i\leq\bar \sigma$. Then, 
\begin{align*} 
0\preceq\Sigma_{22}&\preceq\bar \sigma^2 \E \{T_{\mu,\lambda}^{-1}S^*(S_z+\mu)^{-1}\phi(Z_i)\} \otimes\{T_{\mu,\lambda}^{-1}S^*(S_z+\mu)^{-1}\phi(Z_i)\}^* \\
& =\bar \sigma^2 \E \{T_{\mu,\lambda}^{-1}S^*(S_z+\mu)^{-1}\phi(Z_i)\otimes\phi(Z_i)^*(S_z+\mu)^{-1}S T_{\mu,\lambda}^{-1}\}
\\ &= \bar \sigma^2T_{\mu,\lambda}^{-1}S^*(S_z+\mu)^{-1}S_z(S_z+\mu)^{-1}S T_{\mu,\lambda}^{-1}.
\end{align*}
\end{enumerate}
We collect results as
\begin{align*}
    \Sigma &\preceq 2(\Sigma_{11} + \Sigma_{22})\\
    &\preceq 6(\bar{w}^2 + \bar{u}^2) \Tml^{-1} S^* \Szmu S_z \Szmu S \Tml^{-1}  + 6\bar{v}^2 \kappa_x^2\Tml^{-2}+2\bar \sigma^2 T_{\mu,\lambda}^{-1}S^*(S_z+\mu)^{-1}S_z(S_z+\mu)^{-1}S T_{\mu,\lambda}^{-1}\\
    &=(6\bar{w}^2 + 6\bar{u}^2 +2\bar \sigma^2)\Tml^{-1} S^* \Szmu S_z \Szmu S \Tml^{-1}  + 6\bar{v}^2 \kappa_x^2\Tml^{-2}. \quad\qedhere
\end{align*}
\end{proof}

\subsection{Covariance lower bound}\label{sec:cov-lb}
We derive the variance lower bound of $Z$, where $Z$ is a Gaussian element in ${\Hx}$ with covariance $\Sigma$.
\begin{assumption}[Local width bound assumptions]\label{ass:COLOBO}
    For our analysis, we require that the regularization parameters $\mu, \lambda \in (0,1)$ satisfy the following relationships. There exist constants $c_1, c_2 > 0$ such that
    \begin{enumerate}
        \item $\lambda^{(2\alpha+1)\wedge 2} \le c_1 \mu$,
        \item $\mu \le c_2 \lambda$.
    \end{enumerate}
\end{assumption}
\begin{lemma}[Covariance trace lower bound]\label{lem:Covariance-lower-bound}
Suppose that Assumption~\ref{ass:COLOBO} holds and that $\bb{E}(\varepsilon_i^2\mid Z_i)\geq \underline{\sigma}^2$ almost surely.
Let $\Sigma = \E(U_{i} \otimes U_i^*)$ be the covariance operator. Then
\[
  \tr(\Sigma) \geq \frac{1}{2}\underline{\sigma}^2
  \tilde{\mathfrak{m}}(\lambda,\mu),
\]
where $\tilde{\mathfrak{m}}(\lambda,\mu) = \tr T_{\mu,\lambda}^{-1} S^*(S_z+\mu I)^{-1} S_z (S_z+\mu I)^{-1} S T_{\mu,\lambda}^{-1}$.
\end{lemma}
\begin{proof}
We work at the trace level, decomposing $\tr(\Sigma) = \tr(\Sigma_{11}) + 2\tr(\Sigma_{12}) + \tr(\Sigma_{22})$.  Looking at  $\Sigma_{22}$, since 
$\bb{E}(\varepsilon_i^2\mid Z_i)\geq \underline{\sigma}^2$, by the tower property we obtain
\begin{align*}
  \tr(\Sigma_{22}) = \E\|U_{i2}\|_\Hx^2
  &\geq
  \underline{\sigma}^2
    \tr\!\Bigl[
      T_{\mu,\lambda}^{-1} S^*(S_z+\mu I)^{-1} S_z (S_z+\mu I)^{-1} S T_{\mu,\lambda}^{-1}
    \Bigr]
  = \underline{\sigma}^2\tilde{\mathfrak{m}}(\lambda,\mu).
\end{align*}
For $\Sigma_{11}$, by the parallelogram law, $\Sigma_{11} \preceq 3(\Sigma_{111} + \Sigma_{122} + \Sigma_{133})$, hence $\tr(\Sigma_{11}) \leq 3(\tr(\Sigma_{111}) + \tr(\Sigma_{122}) + \tr(\Sigma_{133}))$. From the bounds established in Lemma~\ref{lemma:cov_upper}: $\tr(\Sigma_{111}) \leq \bar{u}^2\tilde{\mathfrak{m}}(\lambda,\mu)$ and $\tr(\Sigma_{133}) \leq \bar{w}^2\tilde{\mathfrak{m}}(\lambda,\mu)$, obtained by taking the trace of the Loewner bounds $\Sigma_{111} \preceq \bar{u}^2\Tml^{-1} S^* \Szmu S_z \Szmu S \Tml^{-1}$ and $\Sigma_{133} \preceq \bar{w}^2\Tml^{-1} S^* \Szmu S_z \Szmu S \Tml^{-1}$. For $\tr(\Sigma_{122})$, we compute directly from the definition:
    \[
    \tr(\Sigma_{122}) = \E\|U_{i12}\|_\Hx^2 = \E\bigl[w(Z_i)^2\|\Tml^{-1}\psi(X_i)\|_\Hx^2\bigr] \leq \frac{\kappa_x^2}{\lambda^2}\E[w(Z_i)^2] \leq \frac{\kappa_x^2\bar{v}^2}{\lambda^2},
    \]
    where we used $\|\Tml^{-1}\|_{\op} \leq 1/\lambda$ and $\E[w(Z_i)^2]  \leq \|v\|_\Hz^2 \leq \bar{v}^2$.
Combining this yields
\begin{equation}\label{eq:trace-Sigma11}
\tr(\Sigma_{11}) \leq 3(\bar{u}^2 + \bar{w}^2)\tilde{\mathfrak{m}}(\lambda,\mu) + \frac{3\kappa_x^2\bar{v}^2}{\lambda^2}.
\end{equation}
By the Cauchy--Schwarz inequality:
\[
|\tr(\Sigma_{12})| = |\E\langle U_{i1}, U_{i2}\rangle_\Hx| \leq \sqrt{\E\|U_{i1}\|_\Hx^2\E\|U_{i2}\|_\Hx^2} = \sqrt{\tr(\Sigma_{11})\tr(\Sigma_{22})}.
\]
Combining via the scalar inequality we have
\[
\tr(\Sigma) \geq \tr(\Sigma_{22}) - 2\sqrt{\tr(\Sigma_{11})\tr(\Sigma_{22})} = \tr(\Sigma_{22})\Bigl[1 - 2\sqrt{\frac{\tr(\Sigma_{11})}{\tr(\Sigma_{22})}}\Bigr].
\]
We verify that $\tr(\Sigma_{11}) \leq \frac{1}{16}\tr(\Sigma_{22})$. This gives $1 - 2\sqrt{1/16} = 1/2$, hence $\tr(\Sigma) \geq \frac{1}{2}\tr(\Sigma_{22}) \geq \frac{1}{2}\underline{\sigma}^2\tilde{\mathfrak{m}}(\lambda,\mu)$. It suffices to check two conditions:
\begin{enumerate}
    \item[(a)] $3(\bar{u}^2 + \bar{w}^2)\tilde{\mathfrak{m}} \leq \frac{1}{32}\underline{\sigma}^2\tilde{\mathfrak{m}}$, i.e., $3(\bar{u}^2 + \bar{w}^2) \leq \frac{\underline{\sigma}^2}{32}$.
    \item[(b)] $\frac{3\kappa_x^2\bar{v}^2}{\lambda^2} \leq \frac{1}{32}\underline{\sigma}^2\tilde{\mathfrak{m}}(\lambda,\mu)$.
\end{enumerate}
We now verify both conditions under Assumption~\ref{ass:COLOBO}. For (a) using the bound from Lemma \ref{prop:general-bound},
    $
      b_{\mu,\lambda} := C_\alpha \lambda^{\alpha} \|w_0\|_{\Hx}/(1 - \delta).
    $
    Recalling that $\bar{u} = \kappa_x b_{\mu,\lambda}$, we have
    $
      \bar{u}^2 = \kappa_x^2b_{\mu,\lambda}^2.
    $
    Using Lemma \ref{lemma:pj_bias}, we write $K_\alpha = \tilde{K}_\alpha \|w_0\|_\Hx$ and define
    \[
      \bar{w}^2
      := \frac{\kappa_z^2}{\mu} \bar{v}^2 = \frac{\kappa_z^2}{\mu}\frac{1}{(1-\delta)^2}K_\alpha^2
        \lambda^{(2\alpha+1)\wedge 2}.
    \]
   We compare $\bar{u}^2$ and $\bar{w}^2$. The condition $\bar{u}^2 \le \bar{w}^2$ requires
    $
      \mu \le \frac{\kappa_z^2 K_\alpha^2}{\kappa_x^2 C_\alpha^2 \|w_0\|_{\Hx}^2} \lambda^{\gamma(\alpha)},
    $
    where $\gamma(\alpha) = ((2\alpha+1)\wedge 2) - 2\alpha \leq 1$.
    Since Assumption \ref{ass:COLOBO} guarantees $\mu \le c_2 \lambda$, this holds for $c_2$ chosen appropriately.
    To ensure $6\bar{w}^2 \leq \underline{\sigma}^2/32$, define $C_\sigma := \underline{\sigma}^4/(192\bar{\sigma}^2)$ and require $\bar{w}^2 \leq C_\sigma$. This amounts to
    \[
      \mu
      \ge \frac{\kappa_z^2K_\alpha^2}{C_\sigma(1-\delta)^2}
          \lambda^{(2\alpha+1)\wedge 2},
    \]
    which holds by Assumption \ref{ass:COLOBO}. Since $\underline{\sigma} \leq \bar{\sigma}$, we have $C_\sigma \leq \underline{\sigma}^2/192$, and thus $6\bar{w}^2 \leq 6C_\sigma \leq \underline{\sigma}^2/32$. For the verification of (b), using $\bar{v}^2 = \mu\bar{w}^2/\kappa_z^2$ and $\mu \leq c_2\lambda$ (Assumption~\ref{ass:COLOBO}):
\[
\frac{3\kappa_x^2\bar{v}^2}{\lambda^2} = \frac{3\kappa_x^2\mu\bar{w}^2}{\kappa_z^2\lambda^2} \leq \frac{3c_2\kappa_x^2\bar{w}^2}{\kappa_z^2\lambda} \leq \frac{3c_2\kappa_x^2 C_\sigma}{\kappa_z^2\lambda}.
\]
Under the strong instrument assumption $\tilde{\mathfrak{m}}(\lambda,\mu) \gtrsim \lambda^{-\rho_x}$ with $\rho_x > 1$, the condition becomes
\[
\frac{3c_2\kappa_x^2 C_\sigma}{\kappa_z^2\lambda} \leq \frac{\underline{\sigma}^2}{32}\tilde{\mathfrak{m}}(\lambda,\mu) \quad\Longleftarrow\quad \frac{3c_2\kappa_x^2 C_\sigma}{\kappa_z^2} \leq \frac{\underline{\sigma}^2}{32}c_{\mathfrak{m}}\lambda^{1-\rho_x},
\]
where $c_{\mathfrak{m}}$ is the constant in $\tilde{\mathfrak{m}} \geq c_{\mathfrak{m}}\lambda^{-\rho_x}$. Since $\rho_x > 1$, the right-hand side diverges as $\lambda \to 0$, so this holds for all $\lambda$ below a threshold depending only on the constants $(\kappa_x, \kappa_z, \bar{\sigma}, \underline{\sigma}, c_2, \rho_x, c_{\mathfrak{m}})$.
\end{proof}

\begin{lemma}[Variance lower bound]
\label{lemma:valbo}
Suppose Assumption \ref{ass:COLOBO} holds.
Let $Z$ be a Gaussian random element of $\Hx$ with covariance $\Sigma$, and suppose $\bb{E}(\ep_i^2|Z_i) \ge \underline{\sigma}^2$ almost surely. Then with probability $1-\eta$,
$$\norm{Z}_\Hx \ge \sqrt{\frac{1}{2 }\underline{\sigma}^2\mathfrak{\tilde m}(\lambda, \mu)} - \left \{ 2 + \sqrt{2 \log(1/\eta)}\right \}\sqrt{\frac{ C_{\mathrm{op}}\bar \sigma^2}{\lambda} },$$
where $\tilde{\mathfrak{m}}(\lambda, \mu)= \tr T_{\mu,\lambda}^{-1}S^*(S_z+\mu)^{-1}S_z(S_z+\mu)^{-1}S T_{\mu,\lambda}^{-1}$ and $C_{\mathrm{op}} = 1 + c_2\kappa_x^2/\kappa_z^2$.
\end{lemma}
\begin{proof}
    We lower bound $\bb{E}\snorm{Z}_\Hx$ via the identity
$
    \{\bb{E}(\snorm{Z}_\Hx)\}^2 = \bb{E}(\snorm{Z}_\Hx^2) - \bb{E}\{\snorm{Z}_\Hx - \bb{E}(\snorm{Z}_\Hx)\}^2$
then appeal to Lemma \ref{lemma:Gaussianborel}. Let $B_{\Hx}$ be the unit ball in $\Hx$.
\begin{enumerate}
    \item To upper bound $ \bb{E}(\snorm{Z}_\Hx - \bb{E}\snorm{Z}_\Hx)^2$, we express $\norm{Z}_\Hx$ as the supremum of a Gaussian process: $\norm{Z}_\Hx=\sup_{t\in B_{\Hx}}\bk{Z}{t}_\Hx=\sup_{t\in B_{\Hx}}G_t$.  By the variance upper bound
$$
\sigma^2_T=\sup_{t \in B_{\Hx}} \bb{E} \bk{Z}{t}^2 = \norm{\Sigma}_{\op} \leq \frac{ C_{\mathrm{op}}\bar \sigma^2}{\lambda},
$$
where $C_{\mathrm{op}} = 1 + c_2\kappa_x^2/\kappa_z^2$ and the bound follows from the covariance upper bound (Lemma~\ref{lemma:cov_upper}):
\begin{align*}
    \Sigma  &\preceq (6\bar{w}^2 + 6\bar{u}^2 +2\bar \sigma^2)\Tml^{-1} S^* \Szmu S_z \Szmu S \Tml^{-1}  + 6\bar{v}^2 \kappa_x^2\Tml^{-2} .
\end{align*}
Taking operator norms of both terms separately. For the first term, using Assumption \ref{ass:COLOBO} and the fact that   $\bar{w}^2  \leq C_\sigma = \frac{\underline{\sigma}^4}{192 \bar \sigma^2}$:
since $\underline{\sigma} \leq \bar{\sigma}$, we have $12\bar{w}^2 \leq 12C_\sigma \leq \underline{\sigma}^2/16 \leq  \bar \sigma^2/16 \leq 2\bar \sigma^2$. Using $\bar u^2 \leq \bar w^2$:
\begin{align*}
(6\bar{w}^2 + 6\bar{u}^2 +2\bar \sigma^2)\|\Tml^{-1} S^* \Szmu S_z \Szmu S \Tml^{-1}\|_{\op} &\leq (12\bar{w}^2 +2\bar \sigma^2)\frac{1}{4\lambda} \leq \frac{4\bar \sigma^2}{4\lambda} = \frac{\bar \sigma^2}{\lambda}.
\end{align*}
For the second term, using $\bar{v}^2 = \mu\bar{w}^2/\kappa_z^2$ and $\mu \leq c_2\lambda$:
\begin{align*}
6\bar{v}^2 \kappa_x^2\|\Tml^{-2}\|_{\op} \leq \frac{6\bar{v}^2 \kappa_x^2}{\lambda^2} = \frac{6\mu\bar{w}^2 \kappa_x^2}{\kappa_z^2\lambda^2} \leq \frac{6c_2\bar{w}^2 \kappa_x^2}{\kappa_z^2\lambda} \leq \frac{c_2\kappa_x^2\bar \sigma^2}{\kappa_z^2\lambda}.
\end{align*}
Combining gives $\norm{\Sigma}_{\op} \leq \frac{C_{\mathrm{op}}\bar \sigma^2}{\lambda}$ where $C_{\mathrm{op}} = 1 + c_2\kappa_x^2/\kappa_z^2$.
Thus, $\bk{Z}{t}$ is almost surely bounded on $B_{\Hx}$ and by combining the two inequalities of Borell's inequality with a union bound, we have
\begin{align*}
\mathbb{P}\left\{ (\norm{Z}_\Hx - \bb{E}\norm{Z}_\Hx)^2 \ge u \right\} &= \mathbb{P}\left( |\norm{Z}_\Hx - \bb{E}\norm{Z}_\Hx| \ge \sqrt u \right) \le 2\exp\left(-\frac{u}{2 \frac{ C_{\mathrm{op}}\bar \sigma^2}{\lambda}}\right).
\end{align*}
By integrating the tail,
\begin{align*}
\bb{E}(\norm{Z}_\Hx - \bb{E}\norm{Z}_\Hx)^2
&= \int_{0}^\infty \mathbb{P}\left\{ (\norm{Z}_\Hx - \bb{E}\norm{Z}_\Hx)^2 \ge u \right\}du
\le \int_0^\infty 2\exp\left(-\frac{u}{\frac{ 2C_{\mathrm{op}}\bar \sigma^2}{\lambda}}\right)  du
= \frac{4C_{\mathrm{op}}\bar \sigma^2}{\lambda}.
\end{align*}
\item  We lower bound $\bb{E}(\snorm{Z}_\Hx^2)$ by the covariance trace lower bound (Lemma~\ref{lem:Covariance-lower-bound}):
\begin{align*}
    \bb{E}\snorm{Z}_\Hx^2 &= \tr \Sigma
     \geq \frac{1}{2 }\underline{\sigma}^2\mathfrak{\tilde m}(\lambda, \mu).
\end{align*}
\item
Combining the bounds and using $\sqrt{a-b} \ge \sqrt{a} - \sqrt{b}$ (valid for $a \ge b \ge 0$), we get that
$$\E\snorm{Z}_\Hx\geq  \sqrt{\frac{1}{2 }\underline{\sigma}^2\mathfrak{\tilde m}(\lambda, \mu) - \frac{4C_{\mathrm{op}}\bar \sigma^2}{\lambda}}\geq \sqrt{\frac{1}{2 }\underline{\sigma}^2\mathfrak{\tilde m}(\lambda, \mu)} - \sqrt{\frac{4C_{\mathrm{op}}\bar \sigma^2}{\lambda}}.$$
In particular, choosing
$
u = \sqrt{2\|\Sigma\|_{\op}\log(1/\eta)},
$
we conclude that with probability at least $1-\eta$,
\[
\|Z\|_\Hx \geq \mathbb{E}\|Z\|_\Hx - \sqrt{2\|\Sigma\|_{\op}\log(1/\eta)}
.\]
Therefore,
\begin{align*}
    \|Z\|_\Hx &\ge \sqrt{\frac{1}{2 }\underline{\sigma}^2\mathfrak{\tilde m}(\lambda, \mu)} - \sqrt{\frac{4C_{\mathrm{op}}\bar \sigma^2}{\lambda}}  - \sqrt{\frac{2C_{\mathrm{op}}\bar \sigma^2}{\lambda} \log(1/\eta)} \\
    &= \sqrt{\frac{1}{2 }\underline{\sigma}^2\mathfrak{\tilde m}(\lambda, \mu)} - \left \{ 2 + \sqrt{2 \log(1/\eta)}\right \}\sqrt{\frac{ C_{\mathrm{op}}\bar \sigma^2}{\lambda} }. \qquad\qedhere
\end{align*}
\end{enumerate}
\end{proof}

\subsection{Local width upper bound}
\begin{lemma}[Local width upper bound]
  \label{lemma:local_width_bound}
    Suppose Assumptions \ref{ass:COLOBO} and \ref{ass:strong instrument} hold. Then, the local width of $\Sigma$ is bounded as follows. For $m = 0$,
    $
        \sigma(\Sigma,0)\leq \bar \sigma \sqrt{C_{\mathrm{tr}}\tilde{\mathfrak{m}}(\lambda, \mu)},
    $
    where $C_{\mathrm{tr}} = 4 + c_2\kappa_x^2/(c_{\mathfrak{m}}\kappa_z^2)$ and $c_{\mathfrak{m}}$ is the constant in Assumption~\ref{ass:strong instrument}. For all integers $m \geq 1$,
    $
        \sigma(\Sigma, m)\leq ({2 \bar \sigma}/
        {\lambda} )\sigma(S_x, m),
    $
    where $\tilde{\mathfrak{m}}(\lambda, \mu) = \tr  T_{\mu,\lambda}^{-1} S^* (S_z+\mu)^{-1} S_z (S_z+\mu)^{-1} S T_{\mu,\lambda}^{-1}.$
\end{lemma}
\begin{proof}
We derive the two bounds separately. For $m = 0$,
recall $\sigma^2(\Sigma, 0) = \tr(\Sigma)$. By the parallelogram law (Lemma~\ref{lemma:parallelogram_law}), $\Sigma \preceq 2(\Sigma_{11} + \Sigma_{22})$ and $\Sigma_{11} \preceq 3(\Sigma_{111} + \Sigma_{122} + \Sigma_{133})$, so
\[
\tr(\Sigma) \leq 2\bigl(3\tr(\Sigma_{111}) + 3\tr(\Sigma_{122}) + 3\tr(\Sigma_{133}) + \tr(\Sigma_{22})\bigr).
\]
From the Loewner bounds in Lemma~\ref{lemma:cov_upper}, taking traces gives
\begin{align*}
\tr(\Sigma_{111}) &\leq \bar{u}^2\tilde{\mathfrak{m}}(\lambda,\mu), \\
\tr(\Sigma_{133}) &\leq \bar{w}^2\tilde{\mathfrak{m}}(\lambda,\mu), \\
\tr(\Sigma_{22}) &\leq \bar\sigma^2\tilde{\mathfrak{m}}(\lambda,\mu).
\end{align*}
For $\tr(\Sigma_{122})$, we cannot take the trace of the Loewner bound $\Sigma_{122} \preceq \bar{v}^2\kappa_x^2\Tml^{-2}$. Instead, we compute $\tr(\Sigma_{122})$ directly from the definition with
\[
\tr(\Sigma_{122}) = \E\|U_{i12}\|_\Hx^2 = \E\bigl[w(Z_i)^2\|\Tml^{-1}\psi(X_i)\|_\Hx^2\bigr] \leq \frac{\kappa_x^2}{\lambda^2}\E[w(Z_i)^2] \leq \frac{\kappa_x^2\bar{v}^2}{\lambda^2},
\]
where we used $\|\Tml^{-1}\|_{\op} \leq 1/\lambda$ and $\E[w(Z_i)^2] \leq \bar{v}^2$. Combining:
\[
\tr(\Sigma) \leq (6\bar{u}^2 + 6\bar{w}^2 + 2\bar\sigma^2)\tilde{\mathfrak{m}}(\lambda,\mu) + \frac{6\kappa_x^2\bar{v}^2}{\lambda^2}.
\]
Under Assumption~\ref{ass:COLOBO}, the analysis in Lemma~\ref{lemma:valbo} establishes $12\bar{w}^2 \leq 2\bar\sigma^2$ and $\bar{u}^2 \leq \bar{w}^2$. Hence
$6\bar{u}^2 + 6\bar{w}^2 + 2\bar\sigma^2 \leq 12\bar{w}^2 + 2\bar\sigma^2 \leq 4\bar\sigma^2$. For the remainder term, using $\bar{v}^2 = \mu\bar{w}^2/\kappa_z^2$ and $\mu \leq c_2\lambda$
\[
\frac{6\kappa_x^2\bar{v}^2}{\lambda^2} = \frac{6\kappa_x^2\mu\bar{w}^2}{\kappa_z^2\lambda^2} \leq \frac{6c_2\kappa_x^2\bar{w}^2}{\kappa_z^2\lambda} \leq \frac{c_2\kappa_x^2\bar\sigma^2}{\kappa_z^2\lambda},
\]
where the last step uses $6\bar{w}^2 \leq \bar\sigma^2$. By Assumption~\ref{ass:strong instrument}, $\tilde{\mathfrak{m}}(\lambda,\mu) \geq c_{\mathfrak{m}}\lambda^{-\rho_x}$ with $\rho_x \geq 1$, which gives $1/\lambda \leq \tilde{\mathfrak{m}}(\lambda,\mu)/c_{\mathfrak{m}}$. Therefore
\[
\frac{c_2\kappa_x^2\bar\sigma^2}{\kappa_z^2\lambda} \leq \frac{c_2\kappa_x^2\bar\sigma^2}{c_{\mathfrak{m}}\kappa_z^2}\tilde{\mathfrak{m}}(\lambda,\mu).
\]
Combining both contributions gives
\[
\tr(\Sigma) \leq 4\bar\sigma^2\tilde{\mathfrak{m}}(\lambda,\mu) + \frac{c_2\kappa_x^2\bar\sigma^2}{c_{\mathfrak{m}}\kappa_z^2}\tilde{\mathfrak{m}}(\lambda,\mu)
= C_{\mathrm{tr}}\bar\sigma^2\tilde{\mathfrak{m}}(\lambda,\mu),
\]
where $C_{\mathrm{tr}} = 4 + c_2\kappa_x^2/(c_{\mathfrak{m}}\kappa_z^2)$.
For $m \geq 1$, we derive a bound for $\Sigma$ involving the compact operator $\Tml^{-1}S_x\Tml^{-1}$. By Lemma~\ref{lemma:parallelogram_law}, $\Sigma \preceq 2(\Sigma_{11} + \Sigma_{22})$, and $\Sigma_{11} \preceq 3(\Sigma_{111} + \Sigma_{122} + \Sigma_{133})$. For the Bound on $\Sigma_{111}$ and $\Sigma_{133}$, we use the proof of Lemma~\ref{lemma:cov_upper}: $\Sigma_{111} \preceq \bar{u}^2\Tml^{-1}S^*\Szmu S_z\Szmu S\Tml^{-1}$ and $\Sigma_{133} \preceq \bar{w}^2\Tml^{-1}S^*\Szmu S_z\Szmu S\Tml^{-1}$. Since $S^*\Szmu S_z\Szmu S \preceq S^*S_z^{-1}S = T \preceq S_x$ (Lemma~\ref{lemma:X>T}), the conjugation rule gives
\[
\Sigma_{111} + \Sigma_{133} \preceq (\bar{u}^2 + \bar{w}^2)\Tml^{-1}S_x\Tml^{-1}.
\]
For the bound on $\Sigma_{122}$, recall from the proof of Lemma~\ref{lemma:cov_upper} that $\Sigma_{122} \preceq \E[\Tml^{-1}\psi(X_i)w(Z_i)^2\psi(X_i)^*\Tml^{-1}]$. Since $|w(Z_i)| \leq \bar{w}$ almost surely,
\[
\Sigma_{122} \preceq \bar{w}^2\Tml^{-1}\E[\psi(X_i)\psi(X_i)^*]\Tml^{-1} = \bar{w}^2\Tml^{-1}S_x\Tml^{-1}.
\]
Similarly, for the bound of $\Sigma_{22}$, we have that $\Sigma_{22} \preceq \bar\sigma^2\Tml^{-1}S^*\Szmu S_z\Szmu S\Tml^{-1} \preceq \bar\sigma^2\Tml^{-1}S_x\Tml^{-1}$.
Combining this yields
\begin{align*}
\Sigma &\preceq 2\bigl(3(\Sigma_{111} + \Sigma_{122} + \Sigma_{133}) + \Sigma_{22}\bigr) \\
&\preceq 2\bigl(3(\bar{u}^2 + 2\bar{w}^2) + \bar\sigma^2\bigr)\Tml^{-1}S_x\Tml^{-1}.
\end{align*}
Under Assumption~\ref{ass:COLOBO}: $\bar{u}^2 \leq \bar{w}^2$ and $12\bar{w}^2 \leq 2\bar\sigma^2$, so $3(\bar{u}^2 + 2\bar{w}^2) \leq 9\bar{w}^2 \leq (3/4)\bar\sigma^2$. Hence
\[
\Sigma \preceq 2(3/4 + 1)\bar\sigma^2\Tml^{-1}S_x\Tml^{-1} \preceq 4\bar\sigma^2\Tml^{-1}S_x\Tml^{-1}.
\]
The operator $\Tml^{-1}S_x\Tml^{-1}$ is compact since $S_x$ is compact and $\Tml^{-1}$ is bounded.
By Weyl's inequality, $\nu_s(\Sigma) \leq 4\bar\sigma^2\nu_s(\Tml^{-1}S_x\Tml^{-1})$. Applying the singular value inequality $\nu_s(ABA^*) \leq \|A\|_{\op}^2\nu_s(B)$ with $A = \Tml^{-1}$ and $B = S_x$:
\[
\nu_s(\Tml^{-1}S_x\Tml^{-1}) \leq \|\Tml^{-1}\|_{\op}^2\nu_s(S_x) = \frac{1}{\lambda^2}\nu_s(S_x).
\]
Combining and    summing over $s > m$:
\[\sigma^2(\Sigma, m) \leq \frac{ 4\bar \sigma^2}{\lambda^2} \sigma^2(S_x,m). \qedhere\]
\end{proof}

\begin{lemma}[Rate condition]
\label{lemma:rate_delta}
Let $N_\delta:=\max\{n_1, n_2\}$.  Suppose $N_\delta\leq n $. Then,
    $
        \left\|T_{\mu,\lambda}^{-1}(\hat{T}_\mu-T_\mu)\right\|_{\mathrm{op}} \leq \delta \in [0,1/2]
    $
    where the threshold sample sizes are defined as
    \begin{align*}
        n_1 &= 16 \kappa_z^2\mathfrak{n}_z(\mu)^{-1}\mu^{-2}, \\
        n_2 &= 128 l(\eta)^2 \max \left\{ \frac{\kappa^2(\kappa_z \sqrt{\mathfrak{n}_z(\mu)} + 4)}{\lambda\mu}, \frac{8\kappa_z^4 \kappa_x^2}{\lambda \mu^2}, \frac{8\kappa^2}{\lambda\mu}, \frac{32\kappa_x^2 \kappa^2}{\lambda^2 \mu} \right\}.
    \end{align*}
    Here, $\mathfrak{n}_z(\mu) = \tr((S_z+\mu)^{-2}S_z)$ is the effective dimension.
\end{lemma}
\begin{proof}
We study how the high probability bounds comprising $\delta$ scale with $n$, $\mu$, and $\lambda$. Lemma \ref{lemma:delta} derives $\delta$ as
    $$\delta = \underbrace{\frac{\kappa \delta_z \tilde\gamma_1 + \tilde\gamma_1^2}{\lambda\mu} }_A+ \underbrace{\frac{\tilde \gamma_z {\kappa_x}}{2\sqrt{\lambda}\mu}}_B+ \underbrace{\frac{ \tilde\gamma_1}{2\sqrt{\lambda\mu}}}_C + \underbrace{\frac{\tilde \gamma_1 {\kappa_x}}{\lambda\sqrt{\mu}}}_D$$
The individual components are
\begin{align*}
    \tilde{\gamma}_1  &= 8l(\eta)\frac{\kappa}{n^{1/2}},\\
      \delta_z &= 2\kappa_zl(\eta)\left\{\sqrt{\frac{\mathfrak{n}_z(\mu)}{n}} \vee \frac{4\kappa_z}{n\mu}   \right\},\\
      \tilde \gamma_z &=8l(\eta)\frac{\kappa_z^2}{n^{1/2}}.
\end{align*}
For all $n\geq 16 \kappa_z^2\nmu^{-1}\mu^{-2} $ the square root part of $\delta_z$ dominates.
Thus,
\begin{align*}
    \tilde{\gamma}_1 &= 8l(\eta)\frac{\kappa}{n^{1/2}}, \\
      \delta_z &= 2\kappa_zl(\eta)\sqrt{\frac{\mathfrak{n}_z(\mu)}{n}}, \\
      \tilde \gamma_z &=8l(\eta)\frac{\kappa_z^2}{n^{1/2}}.
\end{align*}
For $A$,
\begin{align*}
\frac{\kappa(2\kappa_z l(\eta)\sqrt{\nmu})(8l(\eta)\kappa) + (8l(\eta)\kappa)^2}{n\lambda\mu} \leq 1/8 
&\iff \frac{16l(\eta)^2\kappa^2(\kappa_z\sqrt{\nmu} + 4)}{n\lambda\mu} \leq 1/8 \
\\&\iff n_A \geq \frac{128l(\eta)^2\kappa^2(\kappa_z\sqrt{\nmu} + 4)}{\lambda\mu}.
\end{align*}
For $B$,
\begin{align*}
\frac{8l(\eta)\kappa_z^2{\kappa_x}}{2\sqrt{\lambda}\mu\sqrt{n}} \leq 1/8 \iff \sqrt{n} \geq \frac{32 l(\eta)\kappa_z^2{\kappa_x}}{\sqrt{\lambda}\mu} \
\iff n_B \geq \frac{1024 l(\eta)^2 \kappa_z^4 \kappa_x^2}{\lambda \mu^2}.
\end{align*}
For $C$,
\begin{align*}
\frac{8l(\eta)\kappa}{2\sqrt{\lambda\mu}\sqrt{n}} \leq 1/8 \iff \sqrt{n} \geq \frac{32l(\eta)\kappa}{\sqrt{\lambda\mu}} \
\iff n_C \geq \frac{1024 l(\eta)^2 \kappa^2}{\lambda \mu}.
\end{align*}
For $D$,
\begin{align*}
\frac{8l(\eta)\kappa{\kappa_x}}{\lambda\sqrt{\mu}\sqrt{n}} \leq 1/8 \iff \sqrt{n} \geq \frac{64l(\eta)\kappa{\kappa_x}}{\lambda\sqrt{\mu}} \
\iff n_D \geq \frac{4096 l(\eta)^2 \kappa^2 \kappa_x^2}{\lambda^2 \mu}.
\end{align*}
\end{proof}
\begin{assumption}[Rate condition]\label{ass:rate condition}
Let $\mu, \lambda > 0$. Assume the sample size $n$ satisfies $n \ge \max\{n_1, n_2, N_{\delta_z}\}$, where
   \begin{align*}
     n_1 &= 16 \kappa_z^2\nmu^{-1}\mu^{-2},\\
        n_2 &= 128 l(\eta)^2 \max \left\{ \frac{\kappa^2(\kappa_z \sqrt{\nmu} + 4)}{\lambda\mu}, \frac{8\kappa_z^4 \kappa_x^2}{\lambda \mu^2}, \frac{8\kappa^2}{\lambda\mu}, \frac{32\kappa^2 \kappa_x^2}{\lambda^2 \mu} \right\},\\
        N_{\delta_z} &=64 \kappa_z^2 l(\eta) \left\{ 
 l(\eta)\mathfrak{n}_z(\mu)  \vee 
\mu^{-1}
\right\} .
    \end{align*}

\end{assumption}

\subsection{Bounded summand}
\label{sec:bounded summand}

\begin{lemma}[Bounded $U_i$]
    \label{lemma:bounded_Ui} $U_i$ is $a$-bounded almost surely, meaning for all $\mu,\lambda>0$ 
    $$\|U_i\|_\Hx \leq a :=\frac{2\kappa\bml  + 4\kappa\lambda^{-1/2} \vml + 2\kappa_z^2\mu^{-1/2} \vml + \bar \sigma \kappa_z} {2\sqrt{\lambda\mu}},
 $$
where $\bml = \|h_0 -\hml \|_\Hx$ and $\vml = \|(S_z + \mu)^{-1/2} S(h_0-\hml)\|_\Hz$ are bias and projected bias respectively.
\end{lemma}
\begin{proof} By the triangle inequality and Lemmas \ref{lemma:SzSsubpace} and \ref{lemma:Polar}, we decompose the norm as: 
    \begin{align*}
\|U_i\|_\Hx &=\|\Tml^{-1}S^*(S_z+\mu)^{-1}(S_i-S)\hvec\|_\Hx\\
&+\|\Tml^{-1}(S_i-S)^*(S_z+\mu)^{-1}S\hvec\|_\Hx\\
&+ \|\Tml^{-1}S^*(S_z + \mu)^{-1}(S_z - S_{z,i})(S_z + \mu)^{-1}S\hvec\|_\Hx \\ &+\|T_{\mu,\lambda}^{-1}S^*(S_z+\mu)^{-1}\phi(Z_i)\ep_i\|_\Hx \\
&\leq \frac{\kappa}{\sqrt{\lambda\mu}} \bml + \frac{2\kappa}{\lambda}\sqrt{\frac{1}{\mu}}\vml+ \frac{1}{\sqrt{\lambda\mu}}  \kappa_z^2\sqrt{\frac{1}{\mu}}\vml + \frac{1}{2\sqrt{\lambda\mu}}\bar{\sigma}\kappa_z\\
&= \frac{2\kappa\bml  + 4\kappa\lambda^{-1/2} \vml + 2\kappa_z^2\mu^{-1/2} \vml + \bar \sigma \kappa_z} {2\sqrt{\lambda\mu}}.
\end{align*}
\end{proof}
\begin{lemma}[Dominating factor in $\|U_i\|_\Hx$ under the regime $\lambda=\mu^{\iota}$]\label{lemma:iota}
Suppose there exists $\alpha \in [0,1]$ such that $\vml \in \Ocal\big(\lambda^{(\alpha+\frac12)\wedge 1}\big)$  and the chosen regime is $\lambda=\mu^{\iota}$, with
\[
\iota \ge
\begin{cases}
\dfrac{1}{2\alpha+1}, & \alpha\in[0,\tfrac12]\\[8pt]
\dfrac{1}{2}, & \alpha\in(\tfrac12,1].
\end{cases}
\]
Then,
$
   \|U_i\|_\Hx \le 
   a = \Ocal\Big(\frac{1}{\sqrt{\mu\lambda}}\Big).
$
\end{lemma}
\begin{proof}
Since $\mu \leq \lambda$, the upper bound is dominated by either
  $B_1 = \frac{\vml}{\sqrt{\lambda}\mu}$ or 
  $B_2 = \frac{1}{\sqrt{\mu\lambda}}.$
Let \(p := (\alpha+\tfrac12)\wedge 1\in[\tfrac12,1]\). Then, \(B_1=\Ocal\big(\lambda^{p-\frac12}\mu^{-1}\big)\), so
\[
  \frac{B_1}{B_2}
  = \Ocal\Big(\lambda^{p-\frac12}\mu^{-1}\sqrt{\mu\lambda}\Big)
  = \Ocal\big(\lambda^{p}\mu^{-1/2}\big).
\]
With the regime \(\lambda=\mu^{\iota}\), this becomes
\[
  \frac{B_1}{B_2}
  = \Ocal\big(\mu^{\iota p-\frac12}\big).
\]
Hence \(B_1=\Ocal(B_2)\) provided \(\iota p-\tfrac12\ge 0\), i.e.
\[
  \iota \ge \frac{1}{2p}
  = \frac{1}{2\big((\alpha+\tfrac12)\wedge 1\big)}
  \Longleftrightarrow
  \begin{cases}
    \iota \ge \dfrac{1}{2\alpha+1}, & \alpha\in[0,\tfrac12]\\[8pt]
    \iota \ge \dfrac{1}{2}, & \alpha\in(\tfrac12,1].
  \end{cases}
\]
Under these conditions \(B_1\lesssim B_2\) meaning the bound for $\|U_i\|_\Hx $ is dominated by \(B_2\) and
\[
  \|U_i\|_\Hx \le a = \Ocal\Big(\frac{1}{\sqrt{\mu\lambda}}\Big). \qedhere
\]
\end{proof}
\begin{lemma}[Dominating factor in $\|U_i\|_\Hx$ under the regime $\mu=\lambda/C$]\label{lemma:iota2}
Suppose $\mu=\lambda/C$ with $C>1$ and there exists $\alpha\in[0,1]$ such that $\vml \in \Ocal\big(\lambda^{(\alpha+\frac12)\wedge 1}\big)$, then
\[
   \|U_i\|_\Hx \le a = \Ocal\Big(\frac{1}{\sqrt{\mu\lambda}}\Big).
\]
\end{lemma}
\begin{proof}
By the projected-bias decomposition, the upper bound is controlled by either
$
  B_1 = \frac{\vml}{\sqrt{\lambda}\mu}$ or $
  B_2 = \frac{1}{\sqrt{\mu\lambda}}.$
With $\mu=\lambda/C$ we have
\[
  B_1 = \frac{C\vml}{\lambda^{3/2}},
  \qquad
  B_2 = \frac{\sqrt{C}}{\lambda},
  \qquad
  \frac{B_1}{B_2} = \sqrt{C}\frac{\vml}{\sqrt{\lambda}}.
\]
Let \(p := (\alpha+\tfrac12)\wedge 1\). By hypothesis, \(\vml=\Ocal(\lambda^{p})\), hence
$
  \frac{B_1}{B_2}
  = \sqrt{C}\Ocal\big(\lambda^{p-\frac12}\big).
$ Since \(\alpha\in[0,1]\) implies \(p\in[\tfrac12,1]\), we have \(p-\tfrac12\ge 0\). 
Consequently, \(\|U_i\|_\Hx\le a=\Ocal\big(1/\sqrt{\mu\lambda}\big)\).
\end{proof}
\begin{corollary}[Simplified bound for $\|U_i\|_\Hx$]
\label{col:const bound a}
    Suppose that either $\lambda = \mu^\iota$ or $\mu = \lambda/C$ with $C\geq1$. If $\iota$ is set according to Lemma \ref{lemma:iota}, then $\|U_i\|_\Hx\leq a$ with
    \[
        a \leq \frac{\widetilde{M}}{\sqrt{\mu\lambda}},
    \]
    where $\widetilde{M}$ is a constant depending on the regime. For $\lambda = \mu^\iota$,
    \[
        \widetilde{M} := \tfrac12\left(
        \frac{2\kappa C_\alpha + 4\kappa K_\alpha + 2\kappa_z^2 K_\alpha}{1-c_\delta}\|T^{-\alpha}h_0\|_\Hx
        +\bar{\sigma}\kappa_z\right),
    \]
    and for $\lambda/C = \mu$,
    \[
        \widetilde{M} := \tfrac12\left(
        \frac{2\kappa C_\alpha + 4\kappa K_\alpha + 2\kappa_z^2 K_\alpha C^{1/2}}{1-c_\delta}\|T^{-\alpha}h_0\|_\Hx
        +\bar{\sigma}\kappa_z\right).
    \]
    Here, $K_\alpha$ is the projected bias constant defined in Lemma \ref{lemma:pj_bias}, and $1-c_\delta$ is a constant lower bounding the bias denominator. 
\end{corollary}

\section{Bahadur representation}\label{sec:bahadur}

We derive a Bahadur representation, as a key step towards deriving $Q$. In particular, we show the difference between $n^{1/2}(\hat{h}-h_{\mu,\lambda})$ and $n^{1/2}\E_n(U)$ is small, i.e. the residual vanishes. Here, 
\begin{align*}
U_i&=T_{\mu,\lambda}^{-1}\{S^*(S_z+I\mu)^{-1}(S_i-S)+(S_i-S)^*(S_z+I\mu)^{-1}S \\
&\quad + S^*(S_z+I\mu)^{-1}(S_z - S_{z,i})(S_z+I\mu)^{-1}S\}(h_0-h_{\mu,\lambda})  +T_{\mu,\lambda}^{-1}S^*(S_z+I\mu)^{-1}\phi(Z_i)\ep_i.
\end{align*}

\subsection{Helpful orderings}

\begin{lemma}[Covariance operator bound]\label{lemma:X>T}
Suppose $S_z$ is invertible. Then, the block matrix admits a decomposition 
\begin{equation*} 
\begin{pmatrix} S_x & S^* \\ S & S_z \end{pmatrix} = \begin{pmatrix} I & S^* \\ 0 & S_z \end{pmatrix} \begin{pmatrix} S_x-S^*S_z^{-1}S & 0 \\ S_z^{-1}S & I \end{pmatrix}.
\end{equation*}
Moreover, $S_x \succeq S^*S_z^{-1}S \succeq 0$.
\end{lemma}

\begin{proof}
The decomposition can be checked by matrix multiplication under our maintained assumptions. Next, the lower bound $S^*S_z^{-1}S \succeq 0$ holds as $S^*S_z^{-1}S$ is the symmetric square of $S^*S_z^{-1/2}$. For the upper bound, note that $S_x - S^*S_z^{-1}S$ is symmetric, so it suffices to check non-negativity of eigenvalues. Suppose, to the contrary, that it has an eigenvector $u_\nu$ with negative eigenvalue $\nu < 0$. A straightforward computation using the derived decomposition  shows
\[\begin{pmatrix} u_\nu \\ -S_z^{-1}Su_\nu \end{pmatrix}^*\begin{pmatrix} S_x & S^* \\ S & S_z \end{pmatrix}\begin{pmatrix} u_\nu \\  -S_z^{-1}Su_\nu \end{pmatrix} = \begin{pmatrix} u_\nu \\ 0 \end{pmatrix}^*\begin{pmatrix} \{S_x - S^*S_z^{-1}S\}u_\nu \\ 0 \end{pmatrix} = \nu u_\nu^* u_\nu < 0.\]
This contradicts the hypothesized positive-definiteness.
\end{proof}
\begin{lemma}[Bound on the regularized cross-covariance]\label{lemma:SzSsubpace}
    The operator norm of the regularized conditional expectation operator is bounded, meaning
    \[
        \left\| (S_z + \mu I)^{-1} S \right\|_{\mathrm{op}} \leq \frac{\kappa_x}{\sqrt{\mu}}.
    \]
\end{lemma}

\begin{proof}
    Note that $S_z \preceq S_z + \mu I$ implies $(S_z+\mu I)^{-1} \preceq S_z^{-1}$. Thus,
    \[
        S^*(S_z+I\mu)^{-1}S \preceq S^*S_z^{-1}S \preceq S_x,
    \]
    where the last inequality is shown in Lemma \ref{lemma:X>T}. This implies
    \[
        \|(S_z+I\mu)^{-1/2}S\|_{\mathrm{op}}^2 = \|S^*(S_z+I\mu)^{-1}S\|_{\mathrm{op}} \le \|S_x\|_{\mathrm{op}} \le \kappa_x^2.
    \]
    Therefore, $\|(S_z+I\mu)^{-1/2}S\|_{\mathrm{op}} \le \kappa_x$. Finally, we conclude
    \[
        \left\| (S_z + \mu I)^{-1} S \right\|_{\mathrm{op}}
        \le \left\| (S_z + \mu I)^{-1/2} \right\|_{\mathrm{op}} \, \left\| (S_z + \mu I)^{-1/2} S \right\|_{\mathrm{op}}
        \le \frac{1}{\sqrt{\mu}} \kappa_x. \qedhere
    \]
\end{proof}
\subsection{High probability events}

\begin{lemma}[Hilbert Schmidt bounds for primitive events]\label{lemma:primitive_bounds} 
For all $n\geq 4$ and $\eta \in (0,1)$ $\|\hat{S}-S\|_{\HS}$ and $\|\E_n(\phi_{Z_i}\ep_i)\|_{\HS}$ are each bounded  with probability at least $1-\eta$ by
$$
\|\hat{S}-S\|_{\HS}  \leq 
2l(\eta)\left(\frac{4\kappa}{n}+\frac{2\kappa}{n^{1/2}}\right)
\leq  8l(\eta)\frac{\kappa}{n^{1/2}} =: \tilde{\gamma}_1 
$$
and
$$
\|\E_n(\phi_{Z_i}\ep_i)\|_{\HS} 
\leq 2l(\eta)\left(\frac{2\kappa_z\bar{\sigma}}{n}+\frac{\kappa_z \bar{\sigma}}{n^{1/2}}\right)
\leq 4l(\eta)\frac{\kappa_z \bar{\sigma}}{n^{1/2}} =: \tilde{\gamma}_2.$$
\end{lemma}
\begin{proof}
For 
$\|\hat{S}-S\|_{\HS}$, it holds that $\E(S_i-S)=0$. By the boundedness of the kernel $
    \|S_i-S\|_{\HS}\leq 2\kappa $ and $ 
    \E\|S_i-S\|^2_{\HS}\leq (2\kappa)^2.$
Let $A = 4\kappa$ and $B = 2\kappa$; the bound follows by applying Lemma \ref{lemma:Bernstein}. For $\|\E_n(\phi_{Z_i}\ep_i)\|_{\HS}$ it holds that  $\E(\phi_{Z_i}\ep_i)=0$. Also,
$
\|\phi_{Z_i}\ep_i\|_{\HS}\leq \kappa_z\bar{\sigma}$ and $
\E(\|\phi_{Z_i}\ep_i\|_{\HS}^2)\leq \kappa_z^2 \bar{\sigma}^2.$ Let $A = 2\kappa_z \bar \sigma$ and $B =\kappa_z \bar \sigma $; the bound follows by the same lemma.
\end{proof}

\begin{lemma}[Estimation error bound of the $Z$ covariance operator]
\label{lemma:gammaz}
    $\|\hat{S}_z-S_z\|_{\HS}$ is bounded with probability $1-\eta$ for $n\geq 4$ by 
    
   $$
\|\hat{S}_z -S_z\|_{\HS}  \leq
2l(\eta)\left(\frac{4\kappa_z^2}{n}+\frac{2\kappa_z^2}{n^{1/2}}\right)
\leq  8l(\eta)\frac{\kappa_z^2}{n^{1/2}} =:\tilde{\gamma}_z.
$$
\end{lemma}
\begin{proof}
    The proof follows from the same arguments as in Lemma \ref{lemma:primitive_bounds}.
\end{proof}
\begin{lemma}[High-probability bound for preconditioned $Z$-covariances]
    \label{lemma:sz_hat}
    The Hilbert–Schmidt norm of the projected estimation error of $S_z$ is bounded with high probability: For all $\eta \in (0,1)$ with probability $1-\eta$  it holds that $\|(S_z+I\mu)^{-1}(\hat S_z - S_z)\|_{\HS}\leq\delta_z$, where $$
    \delta_z := 2\kappa_zl(\eta)\left\{\sqrt{\frac{\mathfrak{n}_z(\mu)}{n}} + \frac{4\kappa_z}{n\mu}   \right\}.
$$
Let $N_{\delta_z} := 64 \kappa_z^2 l(\eta) \left\{ 
l(\eta)\mathfrak{n}_z(\mu)  \vee 
\mu^{-1}
\right\}$. 
 A sufficient condition for ensuring \( \delta_z \leq  \frac{1}{2} \) is $N_{\delta_z} \leq n$.
\end{lemma}
\begin{proof}
    The proof follows directly from matching symbols to Lemma F.2 of \cite{singh2023kernel}. The rate condition is the result of solving for $n$ in $\delta_z \leq 1/2$.
\end{proof}
\begin{lemma}[High-probability bound on the projected feature–noise]  \label{lemma:SzPhi} The projected empirical mean of the feature-noise is bounded with high probability meaning  
    $\| (S_z + I\mu)^{-1} \E_n(\phi_{Z_i}\ep_i) \|_{\HS}  \leq \tilde \gamma_3$ with probability $1-\eta$ for all $\eta \in(0,1)$,  where
    $$\tilde \gamma_3  := 2l(\eta)\left\{\frac{2\bar{\sigma} \kappa_z}{n\mu} + \sqrt{\frac{\bar{\sigma}^2\mathfrak{n}_z(\mu) }{n}}\right \}.$$
\end{lemma}
\begin{proof}
    It holds that $\E \{(S_z + I\mu)^{-1} \phi_{Z_i}\ep_i\}= 0$ and also $\| (S_z + I\mu)^{-1} \E_n(\phi_{Z_i}\ep_i) \| _{\HS}\leq\frac{\bar{\sigma}\kappa_z}{\mu}$. For the second moment, let $(\nu_j^z, e_j^z)$ be the eigendecomposition. Then,  
    $$\E\|(S_z + I\mu)^{-1} \phi_{Z_i}\ep_i\|_{\HS}^2 \leq \bar{\sigma}^2 \sum_{j = 1}^\infty \frac{\E \langle\phi_{Z_i},e_j ^z\rangle^2}{(\nu_j^z + \mu)^2} = \bar{\sigma}^2 \sum_{j = 1}^\infty\frac{\nu_j^z}{(\nu_j^z +\mu)^2} = \bar{\sigma}^2\mathfrak{n}_z(\mu).$$
Now, fix $\eta\in (0,1)$. Take  $A = \frac{2\bar{\sigma}\kappa_z}{\mu}$ and $\text{B} = \bar{\sigma}\sqrt{\mathfrak{n}_z(\mu)}$, Lemma \ref{lemma:Bernstein} demonstrates that with probability $1-\eta$
\[
 \| (S_z + I\mu)^{-1} \E_n(\phi_{Z_i}\ep_i) \|_{\HS} \leq 2l(\eta)\left\{\frac{2\bar{\sigma} \kappa_z}{n\mu} + \sqrt{\frac{\bar{\sigma}^2\mathfrak{n}_z(\mu) }{n}}\right \} =:\tilde \gamma_3. 
\qedhere\]
\end{proof}
\begin{lemma} [Compounded operator bound]\label{lemma:gamma1} With probability  $1-2\eta$ the projected compounded feature noise $T^{-1}_{\mu,\lambda}(\hat{S}-S)^*(\hat{S}_z+I\mu)^{-1}\mathbb{E}_n(\phi_Z\ep)$ is bounded, meaning 
$$\left\|T^{-1}_{\mu,\lambda}(\hat{S}-S)^*(\hat{S}_z+I\mu)^{-1}\mathbb{E}_n(\phi_Z\ep)\right\|_\Hx\leq  \gamma_1 :=  \frac{\tilde{\gamma}_1\tilde{\gamma}_2}{\lambda \mu}.$$
\end{lemma}
\begin{proof}
By the triangle inequality and simple bounds $\|\Tml^{-1}\|_{\mathrm{op}}\leq 1/\lambda$ and $\|(S_z+I\mu)^{-1}\|_{\mathrm{op}}\leq 1/\mu$
$$\left\|T^{-1}_{\mu,\lambda}(\hat{S}-S)^*(\hat{S}_z+I\mu)^{-1}\mathbb{E}_n(\phi_Z\ep)\right\|_\Hx\leq \frac{1}{\lambda\mu}\|(\hat{S}-S)^* \|_{\mathrm{op}}
\|\mathbb{E}_n(\phi_Z\ep) \|_\Hz. $$
Invoking Lemma \ref{lemma:primitive_bounds} plus a union bound gives the result.
\end{proof}
\begin{lemma}[Inverse estimation error]
\label{lemma:hat_sz_inv}
    Suppose  $\|(S_z+I\mu)^{-1}(\hat S_z - S_z)\|_{\mathrm{op}} \leq \delta_z \leq \frac{1}{2} $, then, $$\|(\hat{S}_z + I\mu)^{-1} - (S_z+I\mu)^{-1}\|_{\mathrm{op}} \leq \frac{\delta_z(1+2\delta_z)}{\mu}. $$ 
\end{lemma}
\begin{proof}
    By the resolvent identity, for all integers $l \geq 1$, 
\begin{align*}
     (\hat{S}_z + I\mu)^{-1} - (S_z+I\mu)^{-1} 
    = (\hat S_z+I\mu)^{-1} \{(S_z - \hat S_z)( S_z+I\mu)^{-1}\}^l  + \sum_{r = 1}^{l-1} ( S_z+I\mu)^{-1}  \{(S_z - \hat S_z)( S_z+I\mu)^{-1}\}^r.
\end{align*}
Let $l \to \infty $.
Since $\|(S_z+I\mu)^{-1}(\hat S_z - S_z)\|_{\mathrm{op}} =  \delta_z  \leq \frac{1}{2} $ it holds that $\|(\hat S_z+I\mu)^{-1} \{(S_z - \hat S_z)( S_z+I\mu)^{-1}\}^l\|_{\mathrm{op}}\to 0.$
Thus, 
$$ \underbrace{( S_z+I\mu)^{-1}  (S_z - \hat S_z)( S_z+I\mu)^{-1}}_{A} + \underbrace{\sum_{r = 2}^{\infty} ( S_z+I\mu)^{-1}  \{(S_z - \hat S_z)( S_z+I\mu)^{-1}\}^r}_{B}.
 $$
By hypothesis, $\| A\|_{\mathrm{op}} \leq \frac{\delta_z}{\mu}$ and $\| B\|_{\mathrm{op}} \leq \frac{\delta_z^2}{\mu (1 - \delta_z)} \leq \frac{2\delta_z^2}{\mu}$. Thus, $\|A + B \|_{\mathrm{op}} \leq \frac{\delta_z(1+2\delta_z)}{\mu}.$
\end{proof}

\begin{lemma}[Sandwich error bound]
    With probability $1-\eta$, the projected error $T^{-1}_{\mu,\lambda}(\hat{S}-S)^*(\hat{S}_z+I\mu)^{-1}(\hat{S}-S)$ is bounded by
    \[
        \left\|T^{-1}_{\mu,\lambda}(\hat{S}-S)^*(\hat{S}_z+I\mu)^{-1}(\hat{S}-S)\right\|_{\mathrm{op}} \leq \delta_1 := \frac{\tilde{\gamma}_1^2}{\lambda \mu}.
    \]
\end{lemma}
\begin{proof}
    Note that,
$$\left\|T^{-1}_{\mu,\lambda}(\hat{S}-S)^*(\hat{S}_z+I\mu)^{-1}(\hat{S}-S)\right\|_{\mathrm{op}} \leq \frac{1}{\lambda\mu}\|(\hat{S}-S)^* \|_{\mathrm{op}}^2. $$
Invoking Lemma \ref{lemma:primitive_bounds} gives the result.
\end{proof}
\begin{lemma}[High-probability bound for the operator deviation] Suppose   $\|(S_z+I\mu)^{-1}(\hat S_z - S_z)\|_{\mathrm{op}} \leq \delta_z \in [0,1/2]$.  The scaled estimation error of the regularized covariance operator $T_\mu$ is bounded with probability at least $1-2\eta $, 
\[\left\|T_{\mu,\lambda}^{-1}(\hat{T}_\mu-T_\mu)\right\|_{\mathrm{op}}\leq \delta :=\frac{\kappa \delta_z \tilde\gamma_1 + \tilde\gamma_1^2}{\lambda\mu} +\frac{\tilde \gamma_z {\kappa_x}}{2\sqrt{\lambda}\mu}+ \frac{ \tilde\gamma_1}{2\sqrt{\lambda\mu}} + \frac{\tilde \gamma_1 {\kappa_x}}{\lambda\sqrt{\mu}}.
\]
\label{lemma:delta}
\end{lemma}
\begin{proof}
We can write 
\begin{align*}
T_{\mu,\lambda}^{-1}(\hat{T}_\mu - T_\mu )
&=T_{\mu,\lambda}^{-1}\{ \hat{S}^*(\hat{S}_z + I\mu)^{-1} \hat{S} - S^*(S_z+I\mu)^{-1} S\} \\
=&\underbrace{T_{\mu,\lambda}^{-1}\{\hat{S}^*(\hat{S}_z + I\mu)^{-1} \hat{S} - S^*(\hat{S}_z + I\mu)^{-1} \hat{S}\}}_{\text{(I)}}  \\&+ \underbrace{T_{\mu,\lambda}^{-1}\{S^*(\hat{S}_z + I\mu)^{-1} \hat{S} - S^*(S_z+I\mu)^{-1} \hat{S}\}}_{\text{(II)}}
 \\&+ \underbrace{T_{\mu,\lambda}^{-1} \{ S^*(S_z+I\mu)^{-1} \hat{S} - S^*(S_z+I\mu)^{-1} S}_{\text{(III)}}\} .
\end{align*} 
We bound each of the three terms individually with
\begin{align*}
\text{(I)} \quad  \left\|T_{\mu,\lambda}^{-1}(\hat{S} - S)^* (\hat{S}_z + I\mu)^{-1} \hat{S} \right\|_\op &= \left\| T_{\mu,\lambda}^{-1}(\hat{S} - S)^* (\hat{S}_z + I\mu)^{-1}(S_z - \hat{S}_z)(S_z+I\mu)^{-1}\hat{S} \right. \\
& \quad \left. +  T_{\mu,\lambda}^{-1}(\hat{S} - S)^* (S_z+I\mu)^{-1} \hat{S}
\pm T_{\mu,\lambda}^{-1}(\hat{S} - S)^* (S_z+I\mu)^{-1} S \right\|_\op \\
&=  \left\| T_{\mu,\lambda}^{-1}(\hat{S} - S)^* (\hat{S}_z + I\mu)^{-1}(S_z - \hat{S}_z)(S_z+I\mu)^{-1}\hat{S} \right. \\
 &\quad  +  T_{\mu,\lambda}^{-1}(\hat{S} - S)^* (S_z+I\mu)^{-1} (\hat{S}-S) +\left. T_{\mu,\lambda}^{-1}(\hat{S} - S)^* (S_z+I\mu)^{-1} S \right\| _\op
 \\
&\leq 
\underbrace{\left\| T_{\mu,\lambda}^{-1}(\hat{S} - S)^* (\hat{S}_z + I\mu)^{-1}(S_z - \hat{S}_z)(S_z+I\mu)^{-1}\hat{S} \right\|_\op}_{\text{(I.1)}} \\
&\quad + \underbrace{\left\| T_{\mu,\lambda}^{-1}(\hat{S} - S)^* (S_z+I\mu)^{-1} (\hat{S} - S) \right\|_\op}_{\text{(I.2)}} \\
&\quad + \underbrace{\left\| T_{\mu,\lambda}^{-1}(\hat{S} - S)^* (S_z+I\mu)^{-1} S \right\|_\op}_{\text{(I.3)}}.
\end{align*}
We decompose term \textup{(I)} further and bound each piece separately:
\begin{enumerate}
  \item[(I.1)] Using the hypothesis 
    $\bigl\|(S_z+\mu I)^{-1}(\hat S_z - S_z)\bigr\|_{\mathrm{op}} \le \delta_z$
    and Lemma~\ref{lemma:primitive_bounds}, we obtain, with probability at least $1-\eta$,
    \[
      \bigl\| T_{\mu,\lambda}^{-1}(\hat{S} - S)^* (\hat{S}_z + \mu I)^{-1}
             (S_z - \hat{S}_z)(S_z + \mu I)^{-1}\hat{S} \bigr\|_{\mathrm{op}}
      \le \frac{\kappa\,\delta_z}{\lambda\mu}\,\|\hat S - S\|_{\mathrm{op}}
      \le \frac{\kappa\,\delta_z\,\tilde\gamma_1}{\lambda\mu}.
    \]
  \item[(I.2)] By Lemma~\ref{lemma:primitive_bounds}, with probability at least $1-\eta$,
    \[
      \bigl\| T_{\mu,\lambda}^{-1}(\hat{S} - S)^* (S_z + \mu I)^{-1} (\hat{S} - S) \bigr\|_{\mathrm{op}}
      \le \frac{\tilde \gamma_1^2}{\lambda\mu}.
    \]
  \item[(I.3)] Using the hypothesis, Lemma~\ref{lemma:primitive_bounds}, and
    Lemma~\ref{lemma:SzSsubpace}, we have, with probability at least $1-\eta$,
    \[
      \bigl\| T_{\mu,\lambda}^{-1}(\hat{S} - S)^* (S_z + \mu I)^{-1} S \bigr\|_{\mathrm{op}}
      \le \frac{\tilde \gamma_1 {\kappa_x}}{\lambda\sqrt{\mu}}.
    \]
\end{enumerate}
By a union bound, the events in \textup{(I.1)}–\textup{(I.3)} hold simultaneously on the same event
with probability at least $1-\eta$. On this event,
\begin{align*}
  \bigl\| T_{\mu,\lambda}^{-1}(\hat{S} - S)^* (\hat{S}_z + \mu I)^{-1} \hat{S} \bigr\|_{\mathrm{op}}
  &\le \frac{\kappa\,\delta_z\,\tilde\gamma_1}{\lambda\mu}
       + \frac{\tilde \gamma_1^2}{\lambda\mu}
       + \frac{\tilde \gamma_1 {\kappa_x}}{\lambda\sqrt{\mu}}.
\end{align*}
For term \textup{(II)}, we use the resolvent identity together with Lemma~\ref{lemma:gammaz} and Lemma~\ref{lemma:emp_suhas}. With probability at least $1-\eta$,
\[
  \bigl\| T_{\mu,\lambda}^{-1} S^* (S_z+\mu I)^{-1} (\hat S_z - S_z) (\hat S_z+\mu I)^{-1} \hat{S} \bigr\|_{\mathrm{op}}
  \le \frac{\tilde \gamma_z  {\kappa_x}}{2\sqrt{\lambda}\,\mu}.
\]
Lastly, for \textup{(III)} we use
\begin{align*}
  \bigl\|T_{\mu,\lambda}^{-1}S^*(S_z+\mu I)^{-1}(\hat{S}-S)\bigr\|_{\mathrm{op}}
  &\le \bigl\|T_{\mu,\lambda}^{-1}S^*(S_z+\mu I)^{-1}\bigr\|_{\mathrm{op}}
        \,\bigl\|\hat{S}-S\bigr\|_{\mathrm{op}} 
  \le \frac{1}{2\sqrt{\lambda\mu}}\;\tilde{\gamma}_1.
\end{align*}
The last inequality holds on the event from Lemma~\ref{lemma:primitive_bounds}
and follows from Lemma~\ref{lemma:Polar} applied with
$A^* = S^*(S_z+\mu I)^{-1/2}$, so that
\[
  \bigl\|T_{\mu,\lambda}^{-1}S^*(S_z+\mu I)^{-1/2}\bigr\|_{\mathrm{op}}
  \le \frac{1}{2\sqrt{\lambda}}
  \quad\text{and}\quad
  \bigl\|(S_z+\mu I)^{-1/2}\bigr\|_{\mathrm{op}} \le \frac{1}{\sqrt{\mu}}.
\]
Combining the three bounds for \textup{(I)}–\textup{(III)} and taking the
intersection of the two distinct high-probability events involved, we obtain,
for any $0<\eta<1/2$, that with probability at least $1-2\eta$,
\[
  \bigl\|T_{\mu,\lambda}^{-1}(\hat{T}_\mu - T_\mu)\bigr\|_{\mathrm{op}}
  \le \delta
  := \frac{\kappa\,\delta_z\,\tilde{\gamma}_1 + \tilde{\gamma}_1^2}{\lambda\mu}
     + \frac{\tilde{\gamma}_z{\kappa_x}}{2\sqrt{\lambda}\,\mu}
     + \frac{\tilde{\gamma}_1}{2\sqrt{\lambda\mu}}
     + \frac{\tilde{\gamma}_1{\kappa_x}}{\lambda\sqrt{\mu}}.
\qedhere
\]

\end{proof}

\begin{lemma}[Linearization]
\label{lemma:linearization}
    Suppose $\|T_{\mu,\lambda}^{-1}(\hat T_\mu - T_\mu)\|_{\op} \le \d < 1$.
Then, for all $\lambda,\mu>0$, integers $k\geq 1$ and $u\in \Hx$
$$
(\hat{T}_{\mu,\lambda}^{-1}-T_{\mu,\lambda}^{-1})u=A_1u+A_2 T_{\mu,\lambda}^{-1}u+A_3 T^{-1}_{\mu,\lambda} u,\quad 
\|A_1\|_{\op}\leq \frac{\delta^k}{\lambda},\quad \|A_2\|_{\op}\leq \delta,\quad \|A_3\|_{\op}\leq  \frac{\delta^2}{1-\delta}.
$$
Furthermore, suppose $\delta \in [0,1/2]$. Then, $
\|(\hat{T}^{-1}_{\mu,\lambda}-T^{-1}_{\mu,\lambda})u\|_\Hx\leq 2\delta \|T^{-1}_{\mu,\lambda}u\|_\Hx.
$
\end{lemma}
\begin{proof}
By the iterated resolvent identity Lemma \ref{lemma:highorder},
 \begin{align*}
      \hat{T}_{\mu,\lambda}^{-1} -  T_{\mu,\lambda}^{-1} &= \hat T_{\mu,\lambda}^{-1}\{(T_\mu - \hat T_\mu)T_{\mu,\lambda}^{-1}\}^k + \sum_{r = 1}^{k-1} T_{\mu,\lambda}^{-1}\{(T_\mu - \hat T_\mu)T_{\mu,\lambda}^{-1}\}^r\\
      &= \hat T_{\mu,\lambda}^{-1}\{(T_\mu - \hat T_\mu)T_{\mu,\lambda}^{-1}\}^k + T_{\mu,\lambda}^{-1}(T_\mu - \hat T_\mu)T_{\mu,\lambda}^{-1} + \sum_{r = 2}^{k-1} T_{\mu,\lambda}^{-1}\{(T_\mu - \hat T_\mu)T_{\mu,\lambda}^{-1}\}^r\\
      &= A_1 + A_2T_{\mu,\lambda}^{-1} + A_3 T_{\mu,\lambda}^{-1}.
 \end{align*}
 Now, using $\|T_{\mu,\lambda}^{-1}(\hat T_\mu - T_\mu)\|_{\op} \le \d < 1$,
\begin{align*}
    \|A_1\|_{\op} &=  \|\hat T_{\mu,\lambda}^{-1}\{(T_\mu - \hat T_\mu)T_{\mu,\lambda}^{-1}\}^k \|_{\op} \leq \frac{\delta^k}{\lambda}\\
    \|A_2\|_{\op} &= \|T_{\mu,\lambda}^{-1}(T_\mu - \hat T_\mu)\|_{\op} \leq\delta \\
    \|A_3\|_{\op} &= \| \sum_{r = 2}^{k-1} \{T_{\mu,\lambda}^{-1}(T_\mu - \hat T_\mu)\}^r\|_{\op} \leq    \sum_{r = 2}^{k-1}  \delta^r \leq  \sum_{r = 2}^{\infty} \delta^r\leq \frac{\delta^2}{1-\delta}.
\end{align*}
If in addition  $\delta \in [0,1/2]$, then, 
\begin{align*}
 \|(\hat{T}_{\mu,\lambda}^{-1} -  T_{\mu,\lambda}^{-1}) u \|_\Hx & \leq  \|A_1\|_{\op} \|u\|_\Hx +  \|A_2 \|_{\op} \|T_{\mu,\lambda}^{-1}u\| _\Hx
 + \|A_3 \|_{\op} \|T_{\mu,\lambda}^{-1}u\| _\Hx \\
 & \leq \frac{\delta^k}{\lambda} \|u\| _\Hx+ (\delta + \frac{\delta^2}{1-\delta})\|T_{\mu,\lambda}^{-1}u\|_\Hx  \\
 &= \frac{\delta^k}{\lambda} \|u\|_\Hx + \frac{\delta}{1-\delta} \|T_{\mu,\lambda}^{-1}u\|_\Hx \underset{k \to \infty, \delta \in [0,1/2]}{\leq} 2\delta\|T_{\mu,\lambda}^{-1}u\|_\Hx. \qquad \qedhere
\end{align*}
\end{proof}
\begin{lemma}[Projected inverse estimation error with noise bound]
\label{lemma:xi1}
    Suppose $\|(S_z+I\mu)^{-1}(\hat S_z - S_z)\|_{\HS}\leq\delta_z  \in [0,1/2]$. Then, for all $\lambda,\mu>0$ and $\eta \in (0,1)$  with probability $1 - \eta$, 
    $\|T^{-1}_{\mu,\lambda}S^*\{(\hat{S}_z+I\mu)^{-1} -( S_z+I\mu)^{-1}\}\mathbb{E}_n(\phi_Z\ep)\|_\Hx \leq  \xi_1 := \frac{\delta_z \tilde{\gamma}_2}{\sqrt{\lambda\mu}}.$ 
\end{lemma}
\begin{proof}
     By the resolvent identity, for all $l\geq 1$ \begin{align*}
    (\hat{S}_z+I\mu)^{-1} -( S_z+I\mu)^{-1}&= (\hat{S}_z+I\mu)^{-1}\{(\hat S_z - S_z) ( S_z+I\mu)^{-1} \}^l+ \sum_{r=1}^{l-1} (S_z+I\mu)^{-1} \{(\hat S_z - S_z) ( S_z+I\mu)^{-1} \}^r.
    \end{align*} 
For $l \to \infty$, using the same argument as in Lemma \ref{lemma:linearization},
\begin{align*}
&\|  T^{-1}_{\mu,\lambda}S^*
\{(\hat{S}_z+I\mu)^{-1} -( S_z+I\mu)^{-1}\}\mathbb{E}_n(\phi_Z\ep) \|_\Hx\\ 
&\leq \|T^{-1}_{\mu,\lambda}S^* (S_z+I\mu)^{-1} \|_\op \| \sum_{r=1}^{\infty} \{(\hat S_z - S_z) ( S_z+I\mu)^{-1} \}^r\mathbb{E}_n(\phi_Z\ep)\|_\Hx \leq\frac{1}{2\sqrt{\lambda\mu}} 2  \delta_z \tilde \gamma_2 = \frac{\delta_z \tilde \gamma_2 }{\sqrt{\lambda\mu}}. \qquad \qedhere
 \end{align*}
\end{proof}

\begin{lemma}[Empirical projected noise bound]
\label{lemma:gamma}
    Suppose $\|(S_z+I\mu)^{-1}(\hat S_z - S_z)\|_{\HS}\leq\delta_z \leq \frac{1}{2}$. Then, with probability at least $1-3\eta$ it holds that
$$\left\| T_{\mu,\lambda}^{-1}\hat{S}^*(\hat{S}_z+I\mu)^{-1}\E_n\{\phi(Z)\ep\} \right\|_\Hx\leq \gamma := \frac{\kappa2\delta_z\tilde \gamma_3}{\lambda} + \frac{\tilde\gamma_1\tilde \gamma_3}{\lambda} + \frac{\tilde \gamma_2}{2\sqrt{\lambda\mu}}.$$
\end{lemma}
\begin{proof}
Consider the decomposition, 
\begin{align*}
&T_{\mu,\lambda}^{-1}\hat{S}^*(\hat{S}_z+I\mu)^{-1}\E_n\{\phi_Z\ep\} \pm T_{\mu,\lambda}^{-1}\hat{S}^*(S_z+I\mu)^{-1}\E_n\{\phi_Z\ep\} \\
&= T_{\mu,\lambda}^{-1}\hat{S}^* \{(\hat{S}_z+I\mu)^{-1} - (S_z+I\mu)^{-1} \} \E_n(\phi_Z\ep) \\&+T_{\mu,\lambda}^{-1}\hat{S}^*(S_z+I\mu)^{-1}\E_n(\phi_Z\ep) 
\pm T_{\mu,\lambda}^{-1}S^*(S_z+I\mu)^{-1}\E_n(\phi_Z\ep) \\
&= \underbrace{T_{\mu,\lambda}^{-1}\hat{S}^* \{(\hat{S}_z+I\mu)^{-1} - (S_z+I\mu)^{-1} \} \E_n(\phi_Z\ep)}_{\text{I}}   \\ 
&+\underbrace{T_{\mu,\lambda}^{-1}(\hat{S}-S)^*(S_z+I\mu)^{-1}\E_n(\phi_Z\ep)}_{\text{II}} \\
&+ \underbrace{T_{\mu,\lambda}^{-1}S^*(S_z+I\mu)^{-1}\E_n(\phi_Z\ep)}_{\text{III}}.
\end{align*}
\begin{enumerate}
    \item[(I)] In the first term, taking the high-probability event from Lemma \ref{lemma:SzPhi}, using the assumption and the same application of the resolvent identity as in Lemma \ref{lemma:xi1}, one can write
    \begin{align*}
       \|T_{\mu,\lambda}^{-1}\hat{S}^* \{(\hat{S}_z+I\mu)^{-1} - (S_z+I\mu)^{-1} \} \E_n(\phi_Z\ep)\|_\Hx&\leq \frac{\kappa}{\lambda} \|\{(\hat{S}_z+I\mu)^{-1} - (S_z+I\mu)^{-1} \} \E_n(\phi_Z\ep)\|_\Hz \\ &\leq\frac{2\kappa\delta_z\tilde \gamma_3}{\lambda}.
    \end{align*}
\item[(II)] For the second term, on the events from Lemmas \ref{lemma:primitive_bounds} and \ref{lemma:SzPhi},
\begin{align*}
    \|T_{\mu,\lambda}^{-1}(\hat{S}-S)^*(S_z+I\mu)^{-1}\E_n(\phi_Z\ep)\|_\Hx \leq \frac{1}{\lambda}\|(\hat{S}-S)^*\|_\op  \| (S_z+I\mu)^{-1}\E_n(\phi_Z\ep)\|_\Hz \leq \frac{1}{\lambda}\tilde \gamma_1\tilde\gamma_3 .
\end{align*}
\item[(III)] For the third term, on the event from Lemma \ref{lemma:primitive_bounds}, using Lemma \ref{lemma:Polar} with $A^* = S^*(S_z+I\mu)^{-1/2}$, 
\begin{align*}
\|
T_{\mu,\lambda}^{-1}S^*(S_z+I\mu)^{-1}\E_n(\phi_Z\ep)\|_\Hx \leq \|
T_{\mu,\lambda}^{-1}S^*(S_z+I\mu)^{-1}\|_\op \|\E_n(\phi_Z\ep)\|_\Hz \leq \frac{\tilde \gamma_2}{2\sqrt{\lambda\mu}}.
\end{align*}
\end{enumerate}
Combining all three bounds, we get with probability at least $1-3\eta$
\[\gamma := \frac{\kappa2\delta_z\tilde \gamma_3}{\lambda} + \frac{\tilde\gamma_1\tilde \gamma_3}{\lambda} + \frac{\tilde \gamma_2}{2\sqrt{\lambda\mu}}. \qedhere\]
\end{proof}

\subsection{Main result}

\begin{lemma}[Abstract Bahadur representation] 
\label{lemma:abstractbahadur}
Suppose
$$
\left\| T_{\mu,\lambda}^{-1}\hat{S}^*(\hat{S}_z+I\mu)^{-1}\E_n\{\phi(Z)\ep\} \right\|_\Hx\leq  \gamma, \quad \|(S_z+I\mu)^{-1}(\hat S_z - S_z)\|_\op\leq\delta_z \leq \frac{1}{2}, \quad  
\left\|T_{\mu,\lambda}^{-1}(\hat{T}_\mu-T_\mu)\right\|_{\mathrm{op}}\leq \delta\leq \frac{1}{2}, 
$$
$$
\left\|T^{-1}_{\mu,\lambda}(\hat{S}-S)^*(\hat{S}_z+I\mu)^{-1}\mathbb{E}_n(\phi_Z\ep)\right\|_\Hx\leq \gamma_1, \quad \|\hat{S}-S\|_\op  \leq \tilde{\gamma}_1 ,
$$
and  $$
\|T^{-1}_{\mu,\lambda}S^*\{(\hat{S}_z+I\mu)^{-1} -( S_z+I\mu)^{-1}\}\mathbb{E}_n(\phi_Z\ep)\|_\Hx \leq  \xi_1.
$$ 
Then, for all $\mu,\lambda>0$, 
$$
\hat{h}-h_{\mu,\lambda}=\E_n(U_i)+u,\quad \|u\|_\Hx \leq 2\delta\gamma + \gamma_1 + \xi_1 + \left( 2\delta^2 +  R \right) \| h_0 - h_{\mu,\lambda} \|_\Hx
$$
with \[
R := \frac{\delta_z}{\sqrt{\lambda\mu}}\left\{ \delta_z(\kappa +  \tilde \gamma_1) + \tilde\gamma_1/2 \right \} + \frac{\tilde \gamma_1}{\lambda\mu} \big\{\tilde \gamma_1 + \delta_z(1+2\delta_z)(\kappa+ \tilde \gamma_1)\big\}
\]
and 
\begin{align*}
U_i&=T_{\mu,\lambda}^{-1}\{S^*(S_z+I\mu)^{-1}(S_i-S)+(S_i-S)^*(S_z+I\mu)^{-1}S \\
&+ S^*(S_z+I\mu)^{-1}(S_z - S_{z,i})(S_z+I\mu)^{-1}S\}(h_0-h_{\mu,\lambda}) +T_{\mu,\lambda}^{-1}S^*(S_z+I\mu)^{-1}\phi(Z_i)\ep_i.
\end{align*}
Subsequently, 
\begin{align*}
 \E_n(U) &=  
T_{\mu,\lambda}^{-1}\{S^*(S_z+I\mu)^{-1}(\hat S - S) + (\hat S - S)^*(S_z+I\mu)^{-1}S\}(h_0 - h_{\mu,\lambda})
\\
&+
T_{\mu,\lambda}^{-1}S^*(S_z+I\mu)^{-1}\mathbb{E}_n(\phi_Z\varepsilon)
 \\
&+    T_{\mu,\lambda}^{-1} \{S^*(S_z+I\mu)^{-1}(S_z - \hat S_z)(S_z+I\mu)^{-1}S\} (h_0 - h_{\mu, \lambda}).
\end{align*}
\end{lemma}
\begin{proof}
    We proceed in steps. First, decomposing the difference between $
\hat{h}=\hat{T}_{\mu,\lambda}\hat{S}^* (\hat{S}_z+I\mu)^{-1} \E_n \{Y\phi(Z)\}$ and $ h_{\mu,\lambda}=T_{\mu,\lambda}^{-1}T_{\mu} h_0$ yields 
\begin{align*}
\hat{h}-h_{\mu,\lambda}
    &=\hat{T}^{-1}_{\mu,\lambda} \hat{S}^* (\hat{S}_z+I\mu)^{-1}\mathbb{E}_n\{Y\phi(Z)\}-T_{\mu,\lambda}^{-1}T_{\mu}h_0 \\
    &=\hat{T}^{-1}_{\mu,\lambda} \hat{S}^* (\hat{S}_z+I\mu)^{-1}\mathbb{E}_n[\{h_0(X)+\ep\}\phi(Z)]-T_{\mu,\lambda}^{-1}T_{\mu} h_0  \\
    &=\hat{T}^{-1}_{\mu,\lambda}\hat{S}^*  (\hat{S}_z+I\mu)^{-1}\mathbb{E}_n(\phi_Z\ep)+\hat{T}^{-1}_{\mu,\lambda} \hat{T}_{\mu}h_0-T_{\mu,\lambda}^{-1}T_{\mu} h_0\\
    &=:(i)+(ii).
    \end{align*}
Here, $(i) := \hat{T}^{-1}_{\mu,\lambda}\hat{S}^*  (\hat{S}_z+I\mu)^{-1}\mathbb{E}_n(\phi_Z\ep)$ and $(ii) := \hat{T}^{-1}_{\mu,\lambda} \hat{T}_{\mu}h_0-T_{\mu,\lambda}^{-1}T_{\mu} h_0.$
Next,  $(i)=(\hat{T}^{-1}_{\mu,\lambda}-T^{-1}_{\mu,\lambda})P_n+T^{-1}_{\mu,\lambda}P_n$ where $P_n=\hat{S}^*(\hat{S}_z+I\mu)^{-1}\mathbb{E}_n(\phi_Z\ep).$ This follows by writing 
$
    \hat{T}^{-1}_{\mu,\lambda}
    =(\hat{T}^{-1}_{\mu,\lambda}-T^{-1}_{\mu,\lambda})+T^{-1}_{\mu,\lambda}.
    $
Invoking Lemma \ref{lemma:linearization} and the assumed high probability events, the first term can be bounded by 
$$
\|(\hat{T}^{-1}_{\mu,\lambda}-T^{-1}_{\mu,\lambda})P_n\|_\Hx\leq 2\delta\|T_{\mu,\lambda}^{-1}P_n\|_\Hx\leq 2\delta\gamma.
$$
Developing the remaining term,
\begin{align*}
T^{-1}_{\mu,\lambda}P_n&=T^{-1}_{\mu,\lambda}(\hat{S}-S)^*(\hat{S}_z+I\mu)^{-1}\mathbb{E}_n(\phi_Z\ep)+
T^{-1}_{\mu,\lambda}S^*(\hat{S}_z+I\mu)^{-1}\mathbb{E}_n(\phi_Z\ep) \\
&= T^{-1}_{\mu,\lambda}(\hat{S}-S)^*(\hat{S}_z+I\mu)^{-1}\mathbb{E}_n(\phi_Z\ep)+
T^{-1}_{\mu,\lambda}S^*(\hat{S}_z+I\mu)^{-1}\mathbb{E}_n(\phi_Z\ep) \pm T^{-1}_{\mu,\lambda}S^*(S_z+I\mu)^{-1}\mathbb{E}_n(\phi_Z\ep)
 \\
&= T^{-1}_{\mu,\lambda}(\hat{S}-S)^*(\hat{S}_z+I\mu)^{-1}\mathbb{E}_n(\phi_Z\ep)+
T^{-1}_{\mu,\lambda}S^*\{(\hat{S}_z+I\mu)^{-1} -( S_z+I\mu)^{-1}\}\mathbb{E}_n(\phi_Z\ep) \\
&+ T^{-1}_{\mu,\lambda}S^*(S_z+I\mu)^{-1}\mathbb{E}_n(\phi_Z\ep),
\end{align*} 
where by hypothesis $\|T^{-1}_{\mu,\lambda}(\hat{S}-S)^*(\hat{S}_z+I\mu)^{-1}\mathbb{E}_n(\phi_Z\ep)\|_\Hx\leq \gamma_1$ and $\|T^{-1}_{\mu,\lambda}S^*\{(\hat{S}_z+I\mu)^{-1} -( S_z+I\mu)^{-1}\}\mathbb{E}_n(\phi_Z\ep)\|_\Hx \leq \xi_1.$
Next, we show that $(ii)$ is equal to $(\hat{T}^{-1}_{\mu,\lambda}-T^{-1}_{\mu,\lambda})Q_n 
+T^{-1}_{\mu,\lambda}Q_n$ with $Q_n:=(\hat{T}_{\mu}-T_{\mu})(h_0-h_{\mu,\lambda})$. This argument is more involved. Expanding 
    \begin{align*}
    \hat{T}^{-1}_{\mu,\lambda} \hat{T}_{\mu}h_0-T^{-1}_{\mu,\lambda}T_{\mu}h_0
    &=(\hat{T}^{-1}_{\mu,\lambda}-T^{-1}_{\mu,\lambda})\hat{T}_{\mu}h_0
    +T_{\mu,\lambda}^{-1}\hat{T}_{\mu}h_0
    -T^{-1}_{\mu,\lambda}T_{\mu}h_0 \\
    &=(\hat{T}^{-1}_{\mu,\lambda}-T^{-1}_{\mu,\lambda})(\hat{T}_{\mu}-T_{\mu})h_0
    +(\hat{T}^{-1}_{\mu,\lambda}-T^{-1}_{\mu,\lambda})T_{\mu}h_0
    +T_{\mu,\lambda}^{-1}\hat{T}_{\mu}h_0
    -T^{-1}_{\mu,\lambda}T_{\mu}h_0 \\
    &=
      (\hat{T}^{-1}_{\mu,\lambda}-T^{-1}_{\mu,\lambda})(\hat{T}_{\mu}-T_{\mu})h_0-\hat{T}^{-1}_{\mu,\lambda}(\hat{T}_{\mu}-T_{\mu})h_{\mu,\lambda}+T_{\mu,\lambda}^{-1}(\hat{T}_{\mu}-T_{\mu})h_0,
\end{align*}
where in the last line we use the resolvent identity to write
\begin{align*}
    (\hat{T}^{-1}_{\mu,\lambda}-T^{-1}_{\mu,\lambda})T_{\mu}h_0
    =\hat{T}^{-1}_{\mu,\lambda}(T_{\mu}-\hat{T}_{\mu})T_{\mu,\lambda}^{-1}T_{\mu}h_0 
    =\hat{T}^{-1}_{\mu,\lambda}(T_{\mu}-\hat{T}_{\mu})h_{\mu,\lambda} 
    =-\hat{T}^{-1}_{\mu,\lambda}(\hat{T}_{\mu}-T_{\mu})h_{\mu,\lambda}.
\end{align*}
To conclude the argument, we write
\begin{align*}
    &-\hat{T}^{-1}_{\mu,\lambda}(\hat{T}_{\mu}-T_\mu)h_{\mu,\lambda}+T_{\mu,\lambda}^{-1}(\hat{T}_{\mu}-T_\mu)h_0\pm T^{-1}_{\mu,\lambda}(\hat{T}_{\mu}-T_\mu)h_{\mu,\lambda}  \\
&=(T^{-1}_{\mu,\lambda}-\hat{T}^{-1}_{\mu,\lambda})(\hat{T}_{\mu}-T_\mu)h_{\mu,\lambda}+T_{\mu,\lambda}^{-1}(\hat{T}_{\mu}-T_\mu)(h_0-h_{\mu,\lambda}) \\
&=(\hat{T}^{-1}_{\mu,\lambda}-T^{-1}_{\mu,\lambda})(\hat{T}_{\mu}-T_\mu)(-h_{\mu,\lambda})+T_{\mu,\lambda}^{-1}(\hat{T}_{\mu}-T_\mu)(h_0-h_{\mu,\lambda})\\
&=(\hat{T}^{-1}_{\mu,\lambda}-T^{-1}_{\mu,\lambda})(\hat{T}_{\mu}-T_\mu)(-h_{\mu,\lambda})+T_{\mu,\lambda}^{-1} Q_n.
\end{align*}
By Lemma \ref{lemma:linearization} and the high probability events
$$
\|(\hat{T}^{-1}_{\mu,\lambda}-T^{-1}_{\mu,\lambda})Q_n\|_\Hx\leq 2\delta \|T_{\mu,\lambda}^{-1}Q_n\|_\Hx\leq 2\delta^2\|h_0-h_{\mu,\lambda}\|_\Hx.
$$
Developing the remaining term, 
\begin{align*}
    T_{\mu,\lambda}^{-1} Q_n &= T_{\mu,\lambda}^{-1} (\hat{T}_\mu - T_\mu)(h_0 - h_{\mu,\lambda})= T_{\mu,\lambda}^{-1}\{\hat S^*(\hat S_z+I\mu)^{-1} \hat S - S^*( S_z+I\mu)^{-1}  S \}(h_0 - h_{\mu,\lambda}). 
    \end{align*}
We apply the decomposition in Lemma \ref{lemma:ABC} to the term $\hat{S}^*(\hat{S}_z + I\mu)^{-1} \hat{S}$. We set $
    \hat{A} = \hat{S}^*$, $\hat{B} = (\hat{S}_z + I\mu)^{-1}$ and $ \hat{C} = \hat{S}$, $
    A = S^*$, $B = (S_z+I\mu)^{-1}$, $ C = S$.
Defining the perturbations as $\Delta S \coloneqq \hat{S} - S$ and $\Delta B \coloneqq (\hat{S}_z + I\mu)^{-1} - (S_z+I\mu)^{-1}$, the lemma yields
\begin{align*}
\hat{S}^*(\hat{S}_z + I\mu)^{-1} \hat{S} 
&= S^*(S_z+I\mu)^{-1} S 
+ S^*(S_z+I\mu)^{-1} \Delta C 
+ S^* \Delta B S 
+ S^* \Delta B \Delta C \\
&\quad + \Delta A (S_z+I\mu)^{-1} S 
+ \Delta A (S_z+I\mu)^{-1} \Delta C 
+ \Delta A \Delta B S 
+ \Delta A \Delta B \Delta C.
\end{align*}
Thus, the difference is
\begin{align*}
\hat{S}^*(\hat{S}_z + I\mu)^{-1} \hat{S} - S^*(S_z+I\mu)^{-1} S 
&= S^*(S_z+I\mu)^{-1} \Delta S 
+ S^* \Delta B S 
+ S^* \Delta B \Delta S \\
&\quad + \Delta S^* (S_z+I\mu)^{-1} S 
+ \Delta S^* (S_z+I\mu)^{-1} \Delta S \\
&\quad + \Delta S^* \Delta B S 
+ \Delta S^* \Delta B \Delta S.
\end{align*}
Let 
$$
 \Delta B =   \underbrace{( S_z+I\mu)^{-1}  (S_z - \hat S_z)( S_z+I\mu)^{-1}}_{\Delta_1B} + \underbrace{\sum_{r = 2}^{\infty} ( S_z+I\mu)^{-1}  \{(S_z - \hat S_z)( S_z+I\mu)^{-1}\}^r}_{\Delta_2B}.
 $$
 Note the application of Lemma \ref{lemma:hat_sz_inv}. Thus, $\|\Delta B\|_\op \leq \frac{\delta_z(1+2\delta_z)}{\mu}.$
 Since $S^*(S_z+I\mu)^{-1} \Delta S$,  $\Delta S^* (S_z+I\mu)^{-1} S $, and $S^* \Delta_1 B S $ are part of the Bahadur representation, we control the following terms with the assumed events and Lemma \ref{lemma:Polar} taking $A^* = S^*(S_z+I\mu)^{-1/2}$.
 \begin{enumerate}
\item $T_{\mu, \lambda}^{-1}S^* \Delta_2 B S $:
$$ \|T_{\mu, \lambda}^{-1}S^* \Delta_2 B S\|_\op \leq \kappa\frac{\delta_z^2}{2\sqrt{\lambda\mu}(1-\delta_z) } \leq  \kappa\frac{\delta_z^2}{\sqrt{\lambda\mu}}.$$
\item $T_{\mu, \lambda}^{-1}S^* \Delta B \Delta S $:
$$\|T_{\mu, \lambda}^{-1} S^* \Delta B \Delta S\|_\op\leq  \frac{\delta_z(1+2\delta_z)}{2\sqrt{\lambda\mu}} \tilde \gamma_1.$$
\item $T_{\mu, \lambda}^{-1}\Delta S^* (S_z+I\mu)^{-1} \Delta S $:
$$\| T_{\mu, \lambda}^{-1} \Delta S^* (S_z+I\mu)^{-1} \Delta S \|_\op \leq \frac{\tilde \gamma_1^2}{\lambda\mu}.
$$
\item $ T_{\mu, \lambda}^{-1}\Delta S^* \Delta B S $:
$$
\|  T_{\mu, \lambda}^{-1}\Delta S^* \Delta B S\|_\op\leq
\frac{\tilde \gamma_1\delta_z(1+2\delta_z)\kappa}{\lambda\mu}.$$
\item $T_{\mu, \lambda}^{-1}\Delta S^* \Delta B \Delta S$:
$$\|T_{\mu, \lambda}^{-1}\Delta S^* \Delta B \Delta S\|_\op\leq \frac{\delta_z(1+2\delta_z)}{\lambda\mu}\tilde\gamma_1^2.
$$
\end{enumerate}
Together, the difference $
\hat{S}^*(\hat{S}_z + I\mu)^{-1} \hat{S} - S^*(S_z+I\mu)^{-1} S
$, without the terms accounted for in the Bahadur, can be bounded by the sum of the five bounds 
\begin{align*}
&\kappa\frac{\delta_z^2}{\sqrt{\lambda\mu}}
+ \frac{\delta_z(1+2\delta_z)}{2\sqrt{\lambda\mu}} \tilde \gamma_1
+ \frac{\tilde{\gamma}_1^2}{\lambda\mu}
+ \frac{\tilde{\gamma}_1\delta_z(1+2\delta_z)\kappa}{\lambda\mu}
+ \frac{\delta_z(1+2\delta_z)}{\lambda\mu}\tilde{\gamma}_1^2 \\
& =  \frac{\delta_z}{\sqrt{\lambda\mu}}\{ \delta_z(\kappa +  \tilde \gamma_1) + \frac{\tilde\gamma_1}{2}\} + \frac{\tilde \gamma_1}{\lambda\mu}\{\tilde \gamma_1 + \delta_z(1+2\delta_z)(\kappa+ \tilde \gamma_1)\} .
\end{align*}
Together with the first part of the proof,
$$\|u\|_\Hx\leq 2\delta\gamma + \gamma_1 + \xi_1+ \left[2\delta^2 +
\frac{\delta_z}{\sqrt{\lambda\mu}}\{ \delta_z(\kappa +  \tilde \gamma_1) + \frac{\tilde\gamma_1}{2}\} + \frac{\tilde \gamma_1}{\lambda\mu}\{\tilde \gamma_1 + \delta_z(1+2\delta_z)(\kappa+ \tilde \gamma_1)\} \right]  \| h_0-h_{\mu,\lambda}\|_\Hx.
$$
\end{proof}
\begin{theorem}[Bahadur representation]
\label{theorem:bahadur}
Let $0 < \eta < 1/5$ and let $n\geq \max\{N_{\delta_z}, N_\delta\}$.
Then, with probability at least $1-5\eta$ there exists $u \in \Hx$ such that
$
  \hat{h} - h_{\mu,\lambda} = \mathbb{E}_n[U] + u
$ and $
  \|u\|_{\Hx}
  \;\lesssim\; V(n,\mu,\lambda,\eta) + B(n,\mu,\lambda,\eta,h_0)
  \;=:\; \Delta_U$ where
\begin{align*}
    l(\eta) &= \log(2/\eta), \\  \mathfrak{n}_z(\mu) &= \tr((S_z+I\mu)^{-2}S_z), \\
    V(n, \mu, \lambda, \eta) & =
    \left(\frac{\nmu^{3/2}l(\eta)^4} {n^2\lambda^2\mu}\vee\frac{\nmu^{1/2}l(\eta)^3} {n^{3/2}\lambda^{3/2}\mu^{3/2}} \vee \frac{\nmu l(\eta)^3}{n^{3/2}\lambda^{3/2}\mu}  \vee \frac{l(\eta)^2}{n\lambda \mu^{3/2}}  \vee \frac{\mathfrak{n}_z(\mu)^{1/2}l(\eta)^{2}}{n\sqrt{\lambda\mu}}
\right), \\
B(n, \mu, \lambda, \eta, h_0) &=   \left( \frac{l(\eta)^2}{{n\lambda}\mu^2} \vee \frac{l(\eta)^2\mathfrak{n}_z(\mu)}{n\sqrt{\lambda\mu}} \vee \frac{l(\eta)^2\mathfrak{n}_z(\mu)^{1/2}}{n\lambda\mu} \right )\| h_0-h_{\mu,\lambda}\|_\Hx.
\end{align*}
\end{theorem}
\begin{proof}
    Recall the following bounds from the high probability events in Lemma \ref{lemma:abstractbahadur}: 
    \begin{align*}
        \gamma_1 &= \frac{\tilde{\gamma_1}\tilde{\gamma_2}}{\lambda \mu} ,\\
     \gamma& = \frac{\kappa2\delta_z\tilde \gamma_3}{\lambda} + \frac{\tilde\gamma_1\tilde \gamma_3}{\lambda} + \frac{\tilde \gamma_2}{2\sqrt{\lambda\mu}}, \\
       \xi_1 &= \frac{\delta_z \tilde \gamma_2 }{\sqrt{\lambda\mu}}.
    \end{align*}
Note that for these bounds to hold, each of the primitive events needs to hold. Those are, $\tilde \gamma_1$, $\tilde \gamma_2$, $\tilde \gamma_3$ each with probability $1-\eta$. The $\delta_z$ bound holds by Lemma \ref{lemma:sz_hat} with probability $1 - \eta$ and $\delta$ holds on the events $\tilde \gamma_1$,  $\delta_z$ and $\tilde \gamma_z$. Thus, using a union bound with probability $1-5\eta$, all the bounds hold and 
$$\|u\|_\Hx\leq 2\delta\gamma + \gamma_1 + \xi_1+ \left[2\delta^2 +
\frac{\delta_z}{\sqrt{\lambda\mu}}\biggl\{\delta_z(\kappa +  \tilde \gamma_1) + \frac{\tilde\gamma_1}{2}\biggr\} + \frac{\tilde \gamma_1}{\lambda\mu}\biggl\{\tilde \gamma_1 + \delta_z(1+2\delta_z)(\kappa+ \tilde \gamma_1)\biggl\}\right]  \| h_0-h_{\mu,\lambda}\|_\Hx.
$$
Now, let $l(\eta) = \log(2/\eta)$ collecting terms while suppressing constants and considering a regime where $\mu\le \lambda$
\begin{align*}
    &\delta\gamma  + \gamma_1 +\xi_1 \\
&= \mathcal{O}\left(\left (\frac{\nmu^{1/2}l(\eta)^2} {n\lambda\mu}+ \frac{l(\eta)}{\sqrt{n\lambda}\mu} \right) \left( \frac{\nmu l(\eta)^2}{n\lambda} + \frac{l(\eta)}{\sqrt{n\lambda\mu}}\right)
\ \vee\
\frac{l(\eta)^{2}}{n\lambda\mu}
\ \vee\
\frac{\mathfrak{n}_z(\mu)^{1/2}l(\eta)^{2}}{n\sqrt{\lambda\mu}}
\right) \\
&= \mathcal{O}\left(\frac{\nmu^{3/2}l(\eta)^4} {n^2\lambda^2\mu}\vee\frac{\nmu^{1/2}l(\eta)^3} {n^{3/2}\lambda^{3/2}\mu^{3/2}} \vee \frac{\nmu l(\eta)^3}{n^{3/2}\lambda^{3/2}\mu}  \vee \frac{l(\eta)^2}{n\lambda \mu^{3/2}} \vee 
\frac{l(\eta)^{2}}{n\lambda\mu}
\ \vee\
\frac{\mathfrak{n}_z(\mu)^{1/2}l(\eta)^{2}}{n\sqrt{\lambda\mu}}
\right)
\\
    &\implies  \delta\gamma  + \gamma_1 +\xi_1\lesssim  \frac{\nmu^{3/2}l(\eta)^4} {n^2\lambda^2\mu}\vee\frac{\nmu^{1/2}l(\eta)^3} {n^{3/2}\lambda^{3/2}\mu^{3/2}} \vee \frac{\nmu l(\eta)^3}{n^{3/2}\lambda^{3/2}\mu}  \vee \frac{l(\eta)^2}{n\lambda \mu^{3/2}}  \vee \frac{\mathfrak{n}_z(\mu)^{1/2}l(\eta)^{2}}{n\sqrt{\lambda\mu}} =: V(n, \mu, \lambda, \eta).
\end{align*}
\begin{align*}&
\left[2\delta^2 +\frac{\delta_z}{\sqrt{\lambda\mu}}\{ \delta_z(\kappa +  \tilde \gamma_1) + \frac{\tilde\gamma_1}{2}\} +  \frac{\tilde \gamma_1}{\lambda\mu}\{\tilde \gamma_1 + \delta_z(1+2\delta_z)(\kappa+ \tilde \gamma_1) \} \right] \| h_0-h_{\mu,\lambda}\|_\Hx  \\&= \mathcal{O}\left( \left\{\left (\frac{\nmu^{1/2}l(\eta)^2} {n\lambda\mu}+ \frac{l(\eta)}{\sqrt{n\lambda}\mu} \right)^2\vee\frac{l(\eta)^2\mathfrak{n}_z(\mu)}{n\sqrt{\lambda\mu}} \vee \frac{l(\eta)^2\mathfrak{n}_z(\mu)^{1/2}}{n\lambda\mu} \right \}\| h_0-h_{\mu,\lambda}\|_\Hx
\right )  \\&= \mathcal{O}\left( \left\{\frac{\nmu l(\eta)^4} {n^2\lambda^2\mu^2}\vee \frac{l(\eta)^2}{{n\lambda}\mu^2} \vee\frac{l(\eta)^2\mathfrak{n}_z(\mu)}{n\sqrt{\lambda\mu}} \vee \frac{l(\eta)^2\mathfrak{n}_z(\mu)^{1/2}}{n\lambda\mu} \right \}\| h_0-h_{\mu,\lambda}\|_\Hx
\right ) 
  \\&= \mathcal{O}\left( \left\{ \frac{l(\eta)^2}{{n\lambda}\mu^2} \vee \frac{l(\eta)^2\mathfrak{n}_z(\mu)}{n\sqrt{\lambda\mu}} \vee \frac{l(\eta)^2\mathfrak{n}_z(\mu)^{1/2}}{n\lambda\mu} \right \}\| h_0-h_{\mu,\lambda}\|_\Hx
\right ) 
 \\
&\implies  \left[2\delta^2 +\frac{\delta_z}{\sqrt{\lambda\mu}}\{ \delta_z(\kappa +  \tilde \gamma_1) + \frac{\tilde\gamma_1}{2}\} +  \frac{\tilde \gamma_1}{\lambda\mu}\{\tilde \gamma_1 + \delta_z(1+2\delta_z)(\kappa+ \tilde \gamma_1) \} \right] \| h_0-h_{\mu,\lambda}\|_\Hx \\ &\quad\lesssim  \left\{ \frac{l(\eta)^2}{{n\lambda}\mu^2} \vee\frac{l(\eta)^2\mathfrak{n}_z(\mu)}{n\sqrt{\lambda\mu}}\vee \frac{l(\eta)^2\mathfrak{n}_z(\mu)^{1/2}}{n\lambda\mu} \right \}\| h_0-h_{\mu,\lambda}\|_\Hx=: B(n, \mu, \lambda, \eta, h_0).
\end{align*}

Combining this $ \|u\|_\Hx\lesssim  V(n, \mu, \lambda, \eta) + B(n, \mu, \lambda, \eta, h_0)=: \Delta_U$.

\end{proof}

\section{Feasible bootstrap}\label{sec:bahadur2}
We derive a feasible bootstrap in order to derive $R$. 
Specifically, we aim  to bound $\|Z_{\mathfrak{B}} -\mathfrak{B} \|_\Hx$.
Recall that 
\begin{align*}
    Z_{\mathfrak{B}} &=  \frac{1}{n}\sum_{i = 1}^n\sum_{j = 1}^n\left(\frac{V_i-V_j}{\sqrt{2}}\right)h_{ij}, \quad
    \mathfrak{B} =  \frac{1}{n}\sum_{i = 1}^n\sum_{j = 1}^n\left(\frac{\hat V_i- \hat V_j}{\sqrt{2}}\right)h_{ij} 
\end{align*}
with
\begin{align*}
V_i &
=T_{\mu,\lambda}^{-1}\{S^*(S_z+\mu)^{-1}\phi(Z_i)(Y_i - h_{\mu,\lambda}(X_i)) + S_i^*(S_z+\mu)^{-1}S(h_0-h_{\mu,\lambda})
\\&-S^*(S_z+\mu)^{-1}S_{z,i}(S_z+\mu)^{-1}S  (h_0-h_{\mu,\lambda}) \}.
\end{align*}
The feasible version contains
\begin{align*}
    \hat V_i &
=\hat T_{\mu,\lambda}^{-1}\{\hat S^*( \hat S_z+\mu)^{-1}\phi(Z_i)\hat \epsilon_i + S_i^*(\hat S_z+\mu)^{-1}\E_n[\hat \varepsilon\phi(Z)]\\
&-\hat S^*(\hat S_z+\mu)^{-1}S_{z,i}(\hat S_z+\mu)^{-1}\E_n[\hat \varepsilon\phi(Z)]\}.
\end{align*}
This has already been derived in Appendix \ref{sec:closed_form_Bootstrap}. To lighten notation, let  
\begin{align*}
    w_i &= S^*(S_z+\mu)^{-1}\phi(Z_i)(Y_i - h_{\mu,\lambda}(X_i)) + S_i^*(S_z+\mu)^{-1}S(h_0-h_{\mu,\lambda})- S^*(S_z+\mu)^{-1}S_{z,i}(S_z+\mu)^{-1}S  (h_0-h_{\mu,\lambda}),\\
    \hat w_i &=\hat S^*( \hat S_z+\mu)^{-1}\phi(Z_i)\hat \epsilon_i + S_i^*(\hat S_z+\mu)^{-1}\E_n[\hat \varepsilon\phi(Z)]
-\hat S^*(\hat S_z+\mu)^{-1}S_{z,i}(\hat S_z+\mu)^{-1}\E_n[\hat \varepsilon\phi(Z)] ,\\
w_{ij} &= \frac{w_i - w_j}{\sqrt{2}}, \\
\hat w_{ij} &= \frac{\hat w_i - \hat w_j}{\sqrt{2}} .
\end{align*}

\subsection{Helpful orderings}

\begin{lemma}[Regularized empirical block decomposition]\label{lemma:regularized_cov}
Let $\hat S_x, \hat S_z \succeq 0$.  
Then, for all $\mu > 0$, the block matrix admits the decomposition.
\begin{equation*}
\begin{pmatrix}
\hat S_x & \hat S^{*}\\[2pt]
\hat S & \hat S_z + \mu I 
\end{pmatrix}
=
\begin{pmatrix}
I & \hat S^{*}\\[2pt]
0 & \hat S_z + \mu I
\end{pmatrix}
\begin{pmatrix}
\hat S_x - \hat S^{*}(\hat S_z + \mu I)^{-1}\hat S & 0\\[4pt]
(\hat S_z + \mu I )^{-1}\hat S & I
\end{pmatrix}.
\end{equation*}
Moreover, it holds that
$ \hat S_x \succeq \hat S^{*}(\hat S_z + \mu I)^{-1}\hat S \succeq 0.$
\end{lemma}
\begin{proof}
    The left-hand side is 
    \begin{equation*}
            \begin{pmatrix}
\hat S_x & \hat S^{*}\\[2pt]
\hat S & \hat S_z + \mu I 
\end{pmatrix} = 
            \begin{pmatrix}
\hat S_x & \hat S^{*}\\[2pt]
\hat S & \hat S_z 
\end{pmatrix} +   
\begin{pmatrix}
 0 & 0\\[2pt]
 0 & \mu I
\end{pmatrix}.
    \end{equation*}
Thus, the left-hand side is positive semidefinite. The operator $\hat S_x-\hat S^{*}(\hat S_z+\mu I)^{-1}\hat S$ is self-adjoint. 
Assume, towards a contradiction, that it has an eigenvector $u_\nu\neq0$ with eigenvalue $\nu<0$.
Using the factorization above, a direct computation gives
\[
\begin{pmatrix} u_\nu \\ -( \hat S_z+\mu I)^{-1}\hat Su_\nu \end{pmatrix}^{*}
\begin{pmatrix} \hat S_x & \hat S^{*} \\ \hat S & \hat S_z+\mu I \end{pmatrix}
\begin{pmatrix} u_\nu \\ -( \hat S_z+\mu I)^{-1}\hat Su_\nu \end{pmatrix}
\]
\[=
\begin{pmatrix} u_\nu \\ 0 \end{pmatrix}^{*}
\begin{pmatrix} \{\hat S_x-\hat S^{*}(\hat S_z+\mu I)^{-1}\hat S\}u_\nu \\ 0 \end{pmatrix}
= \nu u_\nu^*u_\nu < 0.
\]
But $\begin{psmallmatrix}\hat S_x&\hat S^{*}\\ \hat S&\hat S_z+\mu I\end{psmallmatrix}$ is positive semidefinite, so its quadratic form cannot be negative. Thus,
\[
\hat S_x-\hat S^{*}(\hat S_z+\mu I)^{-1}\hat S \succeq 0. \qedhere
\]

\end{proof}

\begin{lemma}[Bound on the regularized empirical cross-covariance]
\label{lemma:emp_suhas}
The preconditioned empirical cross-covariance is bounded in operator norm; specifically,
\[
\big\|(\hat S_z+\mu I)^{-1/2}\hat S\big\|_\op \le \kappa_x.
\]
\end{lemma}

\begin{proof}
By Lemma~\ref{lemma:regularized_cov}, we have the operator inequality
\[
\hat S^{*}(\hat S_z+\mu I)^{-1}\hat S \preceq \hat S_x .
\]
Taking operator norms of both sides, this yields
\[
\big\|(\hat S_z+\mu I)^{-1/2}\hat S\big\|_\op^2
= \big\|\hat S^{*}(\hat S_z+\mu I)^{-1}\hat S\big\|_\op
\le \|\hat S_x\|_\op \le \kappa_x^2.
\]
Finally, taking the square root gives the claimed bound.
\end{proof}

\begin{lemma}[High-probability bound for the empirical operator deviation]
Assume $\left\|T_{\mu,\lambda}^{-1}(\hat{T}_\mu-T_\mu)\right\|_{\mathrm{op}}\leq
\delta\leq1/2$. Then, 
    $$\|\hat T_{\mu, \lambda}^{-1}(\hat T_\mu - T_\mu)\|_{\mathrm{op}} \leq \delta_B := 3\delta.$$
\end{lemma}
\begin{proof}
We have that 
\begin{align*}
    \|\hat T_{\mu, \lambda}^{-1}(\hat T_\mu - T_\mu)\|_{\mathrm{op}} &=  \|\hat T_{\mu, \lambda}^{-1}(\hat T_\mu - T_\mu) \pm T_{\mu, \lambda}^{-1}(\hat T_\mu - T_\mu)\|_{\mathrm{op}} \\
    &\leq \|(\hat T_{\mu, \lambda}^{-1} -T_{\mu, \lambda}^{-1}) (\hat T_\mu - T_\mu)\|_{\mathrm{op}} + \|T_{\mu, \lambda}^{-1}(\hat T_\mu - T_\mu)\|_{\mathrm{op}} \\
    &\leq 2 \delta \|T_{\mu, \lambda }^{-1}(\hat T_\mu - T_\mu)\|_{\mathrm{op}}  + \delta \leq 2\delta^2 + \delta \leq 3\delta. & \qquad \qedhere
\end{align*}
\end{proof}
\begin{lemma}[Empirical linearization] \label{lemma:emp_linearization} We show an empirical version of Lemma \ref{lemma:linearization}. Suppose  $\|\hat T_{\mu, \lambda}^{-1}(\hat T_\mu - T_\mu)\|_{\mathrm{op}} \leq \delta_B < 1 $. Then, for all integers $k\geq 1$ and $u \in \Hx$ 
$$
({T}_{\mu,\lambda}^{-1}-\hat T_{\mu,\lambda}^{-1})u=A_1u+A_2 \hat T_{\mu,\lambda}^{-1}u+A_3 \hat T^{-1}_{\mu,\lambda} u,\quad 
\|A_1\|_{\op}\leq \frac{\delta_B^k}{\lambda},\quad \|A_2\|_{\op}\leq \delta_B,\quad \|A_3\|_{\op}\leq  \frac{\delta_B^2}{1-\delta_B}.
$$
Furthermore, when $\delta_B \in  [0, 1/2]$ then $
\|(\hat{T}^{-1}_{\mu,\lambda}-T^{-1}_{\mu,\lambda})u\|_\Hx\leq 2\delta _B\|\hat T^{-1}_{\mu,\lambda}u\|_\Hx.
$
\end{lemma}
\begin{proof}
By the iterated resolvent identity Lemma \ref{lemma:highorder}, 
 \begin{align*}
     T_{\mu,\lambda}^{-1} - \hat{T}_{\mu,\lambda}^{-1} &= T_{\mu,\lambda}^{-1}\{( \hat T_\mu - T_\mu) \hat T_{\mu,\lambda}^{-1}\}^k + \sum_{r = 1}^{k-1}  \hat T_{\mu,\lambda}^{-1}\{( \hat T_\mu -  T_\mu)  \hat T_{\mu,\lambda}^{-1}\}^r\\
     &= T_{\mu,\lambda}^{-1}\{( \hat T_\mu - T_\mu) \hat T_{\mu,\lambda}^{-1}\}^k + \sum_{r = 1}^{k-1} \{ \hat T_{\mu,\lambda}^{-1}( \hat T_\mu -  T_\mu) \}^r \hat T_{\mu,\lambda}^{-1}\\
     &=T_{\mu,\lambda}^{-1}\{( \hat T_\mu - T_\mu) \hat T_{\mu,\lambda}^{-1}\}^k + \hat T_{\mu,\lambda}^{-1}( \hat T_\mu -  T_\mu) \hat T_{\mu,\lambda}^{-1} + \sum_{r = 2}^{k-1} \{ \hat T_{\mu,\lambda}^{-1}( \hat T_\mu -  T_\mu) \}^r \hat T_{\mu,\lambda}^{-1}\\
     &= A_1 + A_2 \hat T_{\mu,\lambda}^{-1} + A_3 \hat T_{\mu,\lambda}^{-1}.
 \end{align*}
Now, we bound the operator norms. By assumption, $\|\hat T_{\mu,\lambda}^{-1}(\hat T_\mu - T_\mu)\|_{\op} \le \delta_B < 1$. Because the operators are self-adjoint, it also holds that $\|( \hat T_\mu - T_\mu) \hat T_{\mu,\lambda}^{-1}\|_{\op} \le \delta_B$. Thus:
\begin{align*}
    \|A_1\|_{\op} &=  \|T_{\mu,\lambda}^{-1}\{( \hat T_\mu - T_\mu) \hat T_{\mu,\lambda}^{-1}\}^k \|_{\op} \leq \frac{\delta^k_B}{\lambda}\\
    \|A_2\|_{\op} &= \|\hat T_{\mu,\lambda}^{-1}( \hat T_\mu -  T_\mu)\|_{\op} \leq \delta_B \\
    \|A_3\|_{\op} &= \Big\| \sum_{r = 2}^{k-1}\{ \hat T_{\mu,\lambda}^{-1}( \hat T_\mu -  T_\mu) \}^r \Big\|_{\op} \leq    \sum_{r = 2}^{k-1}  \delta^r_B \leq  \sum_{r = 2}^{\infty} \delta^r_B \leq \frac{\delta^2_B}{1-\delta_B}.
\end{align*}
If in addition we assume $\delta_B \in  [0, 1/2]$, then  
\begin{align*}
 \|(\hat{T}_{\mu,\lambda}^{-1} -  T_{\mu,\lambda}^{-1}) u \|_\Hx & \leq  \|A_1 u \|_\Hx +  \|A_2 \hat T_{\mu,\lambda}^{-1}u\|_\Hx
 + \|A_3 \hat T_{\mu,\lambda}^{-1}u\|_\Hx  \\
 & \leq \|A_1\|_{\op} \|u\|_\Hx +  \|A_2 \|_{\op} \|\hat T_{\mu,\lambda}^{-1}u\|_\Hx
 + \|A_3 \|_{\op} \|\hat T_{\mu,\lambda}^{-1}u\|_\Hx  \\
 & \leq \frac{\delta_B^k}{\lambda} \|u\|_\Hx + \Big(\delta_B + \frac{\delta_B^2}{1-\delta_B}\Big)\|\hat T_{\mu,\lambda}^{-1}u\|_\Hx  \\
 &= \frac{\delta_B^k}{\lambda} \|u\|_\Hx + \frac{\delta_B}{1-\delta_B} \|\hat T_{\mu,\lambda}^{-1}u\|_\Hx \underset{k \to \infty, \delta_B \in  [0, 1/2]}{\leq} 2\delta_B\|\hat T_{\mu,\lambda}^{-1}u\|_\Hx. \qquad \qedhere
\end{align*}
\end{proof}

\subsection{High probability events}
\begin{lemma}[High-probability bound on the empirical feature–noise]
\label{lemma:gamma_prime}
The empirical second moment of the feature–noise product is bounded. With probability $1-\eta$
\[
\E_n \|\phi(Z_i)\eps_i\|_\Hz^2 \leq \gamma^\prime,
\]
where
\[
\gamma^\prime \coloneqq\bar \sigma^2\kappa_z^2
+ 2l(\eta)\left(\frac{2 \bar \sigma^2 \kappa_z^2}{n}
+ \sqrt{\frac{\bar \sigma^4\kappa_z^4}{n}} \right).
\]
\end{lemma}
\begin{proof}
Let \(\xi_i = \|\phi(Z_i)\epsilon_i\|_{\mathcal{H}_z}^2\). By boundedness, \(0 \le \xi_i \le \bar{\sigma}^2\kappa_z^2\), and
\[
\mathbb{E} \xi_i \le \bar{\sigma}^2 \mathbb{E} \|\phi(Z_i)\|_{\mathcal{H}_z}^2 \le \bar{\sigma}^2\kappa_z^2.
\]
Therefore, \(\E \xi_i^2 \le \bar \sigma^4\kappa_z^4\). Applying Lemma~\ref{lemma:Bernstein} to \(\{\xi_i\}_{i=1}^n\) yields with probability at least \(1-\eta\)
\[
\E_n \xi_i \le  \bar \sigma^2\kappa_z^2
+ 2l(\eta)\left(\frac{2 \bar \sigma^2 \kappa_z^2}{n}
+ \sqrt{\frac{\bar \sigma^4\kappa_z^4}{n}} \right).
\]

\end{proof}
\begin{lemma}[High-probability bound on the preconditioned feature–noise]
\label{lemma:gamma_double_prime}
The empirical second moment of the preconditioned feature–noise product is bounded. With probability $1-\eta$
\[
\E_n \|(S_z+\mu)^{-1}\phi(Z_i)\eps_i\|_\Hz^2 \leq \gamma^{\prime\prime},
\]
where
\[
\gamma^{\prime\prime} \coloneqq\bar \sigma^2\mathfrak{n}_z(\mu)
+ 2l(\eta)\left(\frac{2 \bar \sigma^2 \kappa_z^2/\mu^2}{n}
+ \sqrt{\frac{\bar \sigma^4\kappa_z^4/\mu^4}{n}} \right).
\]
\end{lemma}
\begin{proof}
Let \(\xi_i = \|(S_z+\mu)^{-1}\phi(Z_i)\epsilon_i\|_{\mathcal{H}_z}^2\). By boundedness, \(0 \le \xi_i \le \bar{\sigma}^2\kappa_z^2/\mu^2\), since $\|(S_z+\mu)^{-1}\|_{\mathrm{op}} \le 1/\mu$. For the expectation, let $(\nu_j^z, e_j^z)$ be the eigenpairs of $S_z$. Then
\[
\mathbb{E} \xi_i = \bar{\sigma}^2 \sum_{j=1}^\infty \frac{\nu_j^z}{(\nu_j^z + \mu)^2} = \bar{\sigma}^2 \mathfrak{n}_z(\mu).
\]
Therefore, \(\E \xi_i^2 \le (\bar\sigma^2\kappa_z^2/\mu^2)\,\E\xi_i \le \bar \sigma^4\kappa_z^4/\mu^4\). Applying Lemma~\ref{lemma:Bernstein} to \(\{\xi_i\}_{i=1}^n\) yields the result.
\end{proof}
\begin{lemma}[High-probability bounds for preconditioned covariances]
\label{lemma:delta_sz_prime}
    With probability at least \(1-\eta\), the deviation of the empirical covariance concentrates in the Hilbert–Schmidt norm:
    \[
        \big\|(\hat S_z - S_z)(S_z+\mu)^{-1/2}\big\|_{\mathrm{HS}} \le \delta_z^\prime,
    \]
    where
    \[
        \delta_z^\prime \coloneqq 2l(\eta)\left(\frac{4\kappa_z^2}{\sqrt{\mu}n} + \sqrt{\frac{\kappa_z^2\tr\big(S_z(S_z+\mu)^{-1}\big)}{n}} \right).
    \]
    Similarly, with probability at least $1-\eta$, the deviation of the cross-covariance satisfies:
    \[
        \|(\hat S -S)^*(S_z +\mu)^{-1/2}\|_{\mathrm{HS}} \leq \delta_z^{\prime\prime},
    \]
    where
    \[
        \delta_z^{\prime\prime} \coloneqq 2l(\eta)\left(\frac{4\kappa}{\sqrt{\mu}n} + \sqrt{\frac{\kappa_x^2 \tr\big(S_z(S_z+\mu)^{-1}\big) }{n}} \right).
    \]
\end{lemma}
\begin{proof}
For the first statement, let $\xi_i = (S_{z,i} -S_z)(S_z +\mu)^{-1/2}$. First, $\E \xi_i= 0$. Next, $\|\xi_i\|_{\mathrm{HS}}\leq\frac{2 \kappa_z^2}{\sqrt{\mu}}$ and 
\[
\E \| (S_{z,i} -S_z)(S_z +\mu)^{-1/2}\|^2_{\mathrm{HS}} \leq \kappa_z^2\tr S_z(S_z+\mu)^{-1}.
\]
Using Lemma \ref{lemma:Bernstein}, with probability $1-\eta$:
\[
\|(\hat S_z -S_z)(S_z +\mu)^{-1/2}\|_{\mathrm{HS}}\leq 2l(\eta)\left(\frac{4\kappa_z^2}{\sqrt{\mu}n} + \sqrt{\frac{\kappa_z^2\tr S_z(S_z+\mu)^{-1} }{n}} \right) =: \delta_z^\prime.
\]
For the second statement, let $\xi_i = (S_i - S)^*(S_z + \mu)^{-1/2}$. Note that $\|\xi_i\|_{\mathrm{HS}} \leq \frac{2\kappa}{\sqrt{\mu}}$ and
\[
\E\|\xi_i\|_{\mathrm{HS}}^2 \leq \E\tr S_i^*(S_z + \mu)^{-1}S_i \leq \E\left[\|\psi(X_i)\|^2\tr\big((S_z+\mu)^{-1} S_{z,i}\big)\right]\leq \kappa_x^2\tr (S_z+\mu)^{-1} S_z.
\]
Thus, applying Lemma \ref{lemma:Bernstein}, with probability $1-\eta$:
\[
\|(\hat S -S)^*(S_z +\mu)^{-1/2}\|_{\mathrm{HS}} \leq 2l(\eta)\left(\frac{4\kappa}{\sqrt{\mu}n} + \sqrt{\frac{\kappa_x^2\tr S_z(S_z+\mu)^{-1} }{n}} \right) =: \delta_z^{\prime\prime}.\]
\end{proof}

\begin{lemma}[Empirical resolvent-weighted operator bounds]
\label{lemma:delta_prime}
Suppose $\|(S_z+\mu)^{-1}(\hat S_z - S_z)\|_\op =: \delta_z \in  [0, 1/2]$. Then, with probability at least $1-\eta$,
\[
\mathbb{E}_n\big\|S_{z,i}(\hat S_z+\mu)^{-1}\big\|_\HS^2
\le
2\left\{\kappa_z^2 \frac{\delta_z(1+2\delta_z)}{\mu}\right\}^2
+4l(\eta)\left\{\frac{2\kappa_z^4}{\mu^2 n}
+ \sqrt{\frac{\kappa_z^6 \mathfrak{n}_{z}(\mu)}{\mu^2 n}} \right\}
+2\kappa_z^2 \mathfrak{n}_{z}(\mu)
=:\delta_\mu^{\prime},
\]\label{lemma:delta_prime1}
where $\mathfrak{n}_{z}(\mu) := \mathrm{tr}\big( (S_z+\mu)^{-2} S_z\big)$. In addition,
\[
\mathbb{E}_n\big\|S_i^*(\hat S_z+\mu)^{-1}\big\|_\HS^2
\le
2\left\{\kappa \frac{\delta_z(1+2\delta_z)}{\mu}\right\}^2
+4l(\eta)\left\{\frac{2\kappa^2}{\mu^2 n}
+ \sqrt{\frac{\kappa^2 \kappa_x^2 \mathfrak{n}_{z}(\mu)}{\mu^2 n}} \right\}
+2 \kappa_x^2 \mathfrak{n}_{z}(\mu)
=:\delta_\mu^{\prime\prime}.
\]\label{lemma:delta_prime2}
Consequently, with probability at least $1-\eta$,
\[
\mathbb{E}_n \big\| S_i^*(S_z + \mu)^{-1} \big\|_\HS^2
\le
2l(\eta)\left\{\frac{2\kappa^2}{\mu^2 n}
+ \sqrt{\frac{\kappa^2 \kappa_x^2 \mathfrak{n}_{z}(\mu)}{\mu^2 n}} \right\}
+ \kappa_x^2\mathfrak{n}_{z}(\mu)
=:\delta_\mu^{\prime\prime\prime}.
\]
Moreover, with probability at least \(1-\eta\),
\[
\E_n\big\|(S_z+\mu)^{-1/2}S_{z,i}\big\|_{\mathrm{HS}}^2 \le \delta_\mu^{\prime\prime\prime\prime},
\]
where 
\[
\delta_\mu^{\prime\prime\prime\prime} := \kappa_z^2 \tr (S_z + \mu)^{-1}S_{z} + 2l(\eta)\left \{ \frac{2\kappa_z^4}{\mu n}+ \sqrt{\frac{\kappa_z^6\tr (S_z + \mu)^{-1}S_{z}}{\mu n}}\right\}.
\]
\end{lemma}
\begin{proof}
For the first two statements, we use the basic decomposition
\[
\mathbb{E}_n \big\|T(\hat S_z+\mu)^{-1}\big\|_\HS^2
\le
2\mathbb{E}_n \big\|T\big[(\hat S_z+\mu)^{-1}-(S_z+\mu)^{-1}\big]\big\|_\HS^2
+2\mathbb{E}_n \big\|T(S_z+\mu)^{-1}\big\|_\HS^2,
\]
with the appropriate choice of $T$ and the fact that for bounded rank-one operators $\|\cdot\|_\HS =\|\cdot\|_\op $.

\smallskip\noindent
{(i)} Take $T=S_{z,i}$. By Lemma~\ref{lemma:hat_sz_inv} and $\delta_z\le \tfrac12$,
\[
\big\|S_{z,i}\big[(\hat S_z+\mu)^{-1}-(S_z+\mu)^{-1}\big]\big\|_\op
\le \kappa_z^2 \frac{\delta_z(1+2\delta_z)}{\mu}.
\]
For the second term, let $\xi_i:=\|S_{z,i}(S_z+\mu)^{-1}\|_\HS^2$. Then $\xi_i\le \kappa_z^4/\mu^2$ and
\[
\mathbb{E}\xi_i
= \mathbb{E}\mathrm{tr}\big(S_{z,i}(S_z+\mu)^{-2}S_{z,i}\big)
\le \kappa_z^2\mathrm{tr}\big((S_z+\mu)^{-2}S_z\big)
= \kappa_z^2 \mathfrak{n}_z(\mu),
\]
which implies $\mathbb{E}\xi_i^2 \le (\kappa_z^6/\mu^2)\mathfrak{n}_z(\mu)$. Applying Lemma~\ref{lemma:Bernstein} yields the claimed bound.

\smallskip\noindent
(ii) Take $T=S_i^*$. By Lemma~\ref{lemma:hat_sz_inv} and $\delta_z\le \tfrac12$,
\[
\big\|S_i^*\big[(\hat S_z+\mu)^{-1}-(S_z+\mu)^{-1}\big]\big\|_\op
\le \kappa \frac{\delta_z(1+2\delta_z)}{\mu}.
\]
For the second term, with $\xi_i:=\|S_i^*(S_z+\mu)^{-1}\|_\HS^2\le \kappa^2/\mu^2$,
\[
\mathbb{E}\xi_i
= \mathbb{E}\|S_i^*(S_z+\mu)^{-1}\|_\HS ^2
= \mathbb{E}\mathrm{tr}\big(S_i^*(S_z+\mu)^{-2}S_i\big)
= \mathbb{E}\big[\|\psi(X_i)\|_\Hx^2\mathrm{tr}\big((S_z+\mu)^{-2}S_{z,i}\big)\big]
\le \kappa_x^2 \mathfrak{n}_z(\mu),
\]
hence $\mathbb{E}\xi_i^2 \le (\kappa^2 \kappa_x^2/\mu^2)\mathfrak{n}_z(\mu)$. Lemma~\ref{lemma:Bernstein} gives the desired result for $\mathbb{E}_n\|S_i^*(\hat S_z+\mu)^{-1}\|_\HS^2$, and the subsequent bound for $\mathbb{E}_n\|S_i^*(S_z+\mu)^{-1}\|_\HS^2$ follows by the same Bernstein inequality applied directly to $\xi_i$.
For the last statement, let $\xi_i =\|(S_z + \mu)^{-1/2}S_{z,i}\|_{\mathrm{HS}}^2$. Clearly, $\xi_i \leq \frac{\kappa_z^4}{\mu}$. Furthermore, 
\[
\E\xi_i \leq\E\|(S_z + \mu)^{-1/2}S_{z,i}\|_{\mathrm{HS}}^2 = \E\tr (S_z + \mu)^{-1}S_{z,i}^2 \leq \kappa_z^2 \E\tr (S_z + \mu)^{-1}S_{z,i} = \kappa_z^2\tr (S_z + \mu)^{-1}S_{z}.
\]
Also, $\E \xi_i^2 \leq \frac{\kappa_z^6\tr (S_z + \mu)^{-1}S_{z}}{\mu}$. By Lemma \ref{lemma:Bernstein}, with probability $1-\eta$,
\[
\E_n\|(S_z + \mu)^{-1/2}S_{z,i}\|_{\mathrm{HS}}^2 \leq \kappa_z^2 \tr (S_z + \mu)^{-1}S_{z} + 2l(\eta) \left \{ \frac{2\kappa_z^4}{\mu n}+ \sqrt{\frac{\kappa_z^6\tr (S_z + \mu)^{-1}S_{z}}{\mu n}}\right\}=: \delta_\mu^{\prime\prime\prime\prime} . \qedhere
\]
\end{proof}

\subsection{Combining events}
\begin{lemma}[Consistency]\label{lemma:consistency}
Let $\lambda>0$ and $\mu>0$.
Assume that the following bounds are satisfied:
\begin{align*}
  \bigl\|(\hat S_z - S_z)(S_z + \mu I)^{-1/2}\bigr\|_{\mathrm{op}}
  &\le \delta_z',\\
  \bigl\|(\hat S - S)^*(S_z + \mu I)^{-1/2}\bigr\|_{\mathrm{op}}
  &\le \delta_z'',\\
  \|S - \hat S\|_{\mathrm{op}} &\le \tilde{\gamma}_1,\\
  \bigl\|\mathbb{E}_n[\phi(Z_i)\varepsilon_i]\bigr\|_{\Hz}
  &\le \tilde{\gamma}_2.
\end{align*}
Define the bias and projected bias by
$  b_{\mu,\lambda}
  := \|h_0 - h_{\mu,\lambda}\|_{\Hx}$ and $
  v_{\mu,\lambda}
  := \bigl\|(S_z + \mu I)^{-1/2} S\,(h_0 - h_{\mu,\lambda})\bigr\|_{\Hz}.$
Then, on the event from the abstract Bahadur representation
(Lemma~\ref{lemma:abstractbahadur}),
\[
   \Delta_{h_{\mu,\lambda}}
  :=\|h_{\mu,\lambda} - \hat{h}\|_{\Hx}
  \le
  \frac{\tilde{\gamma}_1}{2\sqrt{\lambda\mu}}\,b_{\mu,\lambda}
     + \frac{\delta_z''}{\lambda}\,v_{\mu,\lambda}
     + \frac{\tilde{\gamma}_2}{2\sqrt{\lambda\mu}}
     + \frac{\kappa_x\delta_z'}{2\sqrt{\lambda\mu}}\,b_{\mu,\lambda}
     + \|\Delta_U\|_{\Hx},
\]
and $\Delta_U \in \Hx$ is the Bahadur remainder term from
Lemma~\ref{lemma:abstractbahadur}. In particular,
\[
  \Delta_h := \|h_0 - \hat{h}\|_{\Hx} 
  \le \|h_0 - h_{\mu,\lambda}\|_{\Hx}
       + \|h_{\mu,\lambda} - \hat{h}\|_{\Hx}
  \le b_{\mu,\lambda} + \Delta_{h_{\mu,\lambda}}
\]
\end{lemma}
\begin{proof}
Lemma \ref{lemma:abstractbahadur} demonstrated that 
$$\hat{h}-h_{\mu,\lambda}=\E_n(U_i)+\Delta_U,$$where 
\begin{align*}
 \E_n(U) &=  
T_{\mu,\lambda}^{-1}\{S^*(S_z+\mu)^{-1}(\hat S - S) + (\hat S - S)^*(S_z + \mu)^{-1}S\}(h_0 - h_{\mu,\lambda})
\\
&+
T_{\mu,\lambda}^{-1}S^*(S_z + \mu)^{-1}\mathbb{E}_n(\phi_Z\varepsilon)
 \\
&+    T_{\mu,\lambda}^{-1} \{S^*(S_z + \mu)^{-1}(S_z - \hat S_z)(S_z + \mu)^{-1}S\} (h_0 - h_{\mu, \lambda}).
\end{align*}
Using that $\|\hat h- \hml \|_\Hx \leq \| \E_n(U_i) \|_\Hx  +  \|\Delta_U\|_\Hx $, focussing on the individual terms using the bias ($\bml$), projected bias ($\vml$), and Lemma \ref{lemma:delta_sz_prime}: 
\begin{align*}
    \|\E_n(U_i)\|_\Hx &\leq \frac{1}{2\sqrt{\lambda\mu}}\tilde \gamma_1\bml + \frac{1}{\lambda}\delta_z^{\prime\prime} \vml + \frac{1}{2\sqrt{\lambda\mu}}\tilde \gamma_2 + \frac{\kappa_x}{2\sqrt{\lambda\mu}}\delta_z^\prime \bml.
\end{align*}
Combining this with Lemma \ref{lemma:abstractbahadur} gives the result. For $\Delta_h := \|h_0 - \hat{h}\|$, we can write $\Delta_h \leq \|h_0 - \hml\|_\Hx + \|\hml - \hat{h}\|_\Hx \leq \bml  + \Delta_{h_{\mu, \lambda}}.$
\end{proof}

\begin{lemma}[$\Delta_1$-bound] \label{lemma:DELTA1}
Let $\lambda>0$ and $\mu>0$. Let $u = \frac{1}{n}\sum_{i = 1}^n\sum_{j = 1}^n\hat{w}_{ij}  h_{ij}.$
Assume that the following bounds are satisfied:\begin{align*}
         \|(S_z+\mu)^{-1}(\hat S_z - S_z)\|_\op \leq \delta_z &\in  [0, 1/2],\\
        \|\E_n[\phi(Z_i)\eps_i]\|_\Hz
&\leq  \tilde \gamma_2,\\
        \E_n \|\phi(Z_i)\eps_i\|_\Hz^2 &\leq \gamma^\prime ,\\
        \E_n\|S_{z,i}^*( \hat S_z+\mu)^{-1} \|_\op^2 &\leq  \delta_\mu^\prime ,\\
        \E_n\|S_{i}^*( \hat S_z+\mu)^{-1} \|_\op^2 &\leq  \delta_\mu^{\prime \prime}.
    \end{align*} Then, conditional on data  $(X_i,Z_i,Y_i)_{i=1}^n$, with probability $1-\eta$, 
    \begin{align*}
        \| \hTml u \|_\Hx &\leq \left\{1 + \sqrt{2\log(1/\eta)}\right\} \sqrt{\E\|\hTml u\|_\Hx^2} \\ &\leq 
        \left\{1 + \sqrt{2\log(1/\eta)}\right\} \sqrt{
    14  \Big\{ \frac{\kappa^2\Delta_h^2}{\lambda\mu} + \frac{4\kappa^2\kappa_x^2\Delta_h^2}{\lambda^2\mu}  + \frac{\kappa_z^4\kappa_x^2 \Delta_h^2}{4\lambda\mu^2} +4\Delta_h^2+ \frac{\gamma^\prime}{4\mu \lambda} + \frac{\tilde \gamma_2^2 \delta_\mu^{\prime \prime}}{\lambda^2} + \frac{\tilde \gamma_2^2 \delta_\mu^{\prime}}{4\lambda\mu} \Big \}}.
    \end{align*}
   If, in addition, $\|\hat T_{\mu, \lambda}^{-1}(\hat T_\mu - T_\mu)\|_{\mathrm{op}} \leq \delta_B \in  [0, 1/2]$, then with probability $1-\eta$ with respect to the bootstrap weights, $$\|\Delta_1\|_\Hx \leq2\delta_B   \left\{1 + \sqrt{2\log(1/\eta)}\right\} \sqrt{
    14\Big\{ \frac{\kappa^2\Delta_h^2}{\lambda\mu} + \frac{4\kappa^2\kappa_x^2\Delta_h^2}{\lambda^2\mu}  + \frac{\kappa_z^4\kappa_x^2 \Delta_h^2}{4\lambda\mu^2} +4\Delta_h^2+ \frac{\gamma^\prime}{4\mu \lambda} + \frac{\tilde \gamma_2^2 \delta_\mu^{\prime \prime}}{\lambda^2} + \frac{\tilde \gamma_2^2 \delta_\mu^{\prime}}{4\lambda\mu} \Big \}},$$
   where $\Delta_h:= \|h_0-\hat h\|_\Hx$ and $\Delta_1 = (\hat T_{\mu, \lambda}^{-1} - T_{\mu, \lambda}^{-1} )(\frac{1}{n}\sum_{i = 1}^n\sum_{j = 1}^n\hat{w}_{ij}  h_{ij})$ .
\end{lemma}
\begin{proof}
Let $u = \frac{1}{n}\sum_{i = 1}^n\sum_{j = 1}^n\hat{w}_{ij}  h_{ij}.$ Conditional on the data, $\hTml u$ is Gaussian. Using Lemma \ref{lemma:Gaussianborel}, with probability $1-\eta$,
    $$\|\hTml u\|_\Hx\leq \left\{1 + \sqrt{2\log(1/\eta)}\right\} \sqrt{\E\|\hTml u\|_\Hx^2}.$$
Since $\E_h\|\hTml u\|_\Hx^2 = \frac{1}{n^2}\sum_{i = 1}^n\sum_{j = 1}^n\|\hTml\hat{w}_{ij}\|_\Hx^2 $, within each term, 
\begin{align*}
\|\hTml \hat{w}_{ij}\|_\Hx^2 = \frac{1}{2}\|\hTml\{ \hat w_i - \hat w_j\}\|_\Hx^2 \leq \|\hTml\hat w_i\|_\Hx^2+ \|\hTml\hat w_j\|_\Hx^2 .
\end{align*}
Consider the useful rewriting of
\begin{align*}
\phi(Z_i)\hat \epsilon_i &= \phi(Z_i)(Y_i - \hat h(X_i)) = \phi(Z_i)(h_0(X_i)- \hat h(X_i) + \epsilon_i) = S_i(h_0-\hat h) + \phi(Z_i)\epsilon_i \\
\E_n[\hat{\epsilon}\phi(Z)]&= \E_n[(Y - \hat h(X)) \phi(Z)] = \hat{S}(h_0 - \hat h) + \E_n[\epsilon \phi(Z)].
\end{align*}
Plugging these quantities into $\hat w_i$ yielding
\begin{align*}
    \hat w_i &= \hat S^*( \hat S_z+\mu)^{-1}(S_i(h_0-\hat h) + \phi(Z_i)\epsilon_i) \\&+ S_i^*(\hat S_z+\mu)^{-1}(\hat{S}(h_0 - \hat h) + \E_n[\epsilon \phi(Z)])
\\&-\hat S^*(\hat S_z+\mu)^{-1}S_{z,i}(\hat S_z+\mu)^{-1}(\hat{S}(h_0 - \hat h) + \E_n[\epsilon \phi(Z)]).
\end{align*}
Combining the terms gives 
\begin{align*}
    \hat w_i &= \underbrace{\left\{\hat S^*( \hat S_z+\mu)^{-1}S_i + S_i^*(\hat S_z+\mu)^{-1}\hat{S} -\hat S^*(\hat S_z+\mu)^{-1}S_{z,i}(\hat S_z+\mu)^{-1}\hat{S}\right \}(h_0-\hat h )}_A\\
    &+ \underbrace{\hat S^*( \hat S_z+\mu)^{-1} \phi(Z_i)\epsilon_i}_B \\
    &+ \underbrace{\left \{S_i^*(\hat S_z+\mu)^{-1}    -\hat S^*(\hat S_z+\mu)^{-1}S_{z,i}(\hat S_z+\mu)^{-1} \right \}\E_n[\epsilon \phi(Z)]}_C.
\end{align*}
By the parallelogram law
\begin{align*}
\frac{1}{n^2}\sum_{i = 1}^n\sum_{j = 1}^n\|\hTml\hat{w}_{ij}\|_\Hx^2  \leq \frac{2}{n}\sum_{i = 1}^n \|\hTml\hat w_i\|_\Hx^2 .
\end{align*}
Now, focus on the individual components of this expression, and as demonstrated in Lemma \ref{lemma:Polar}
$ \|\hat T_{\mu, \lambda}^{-1}\hat S^*( \hat S_z+\mu)^{-1}\|_{\mathrm{op}} \leq \frac{1}{2\sqrt{\lambda\mu}}.$
\begin{enumerate}
    \item[(A)] We add and subtract the estimated $\hTml\hat T_\mu$, resulting in\begin{align*}\label{ref:fixedA}
        &\|\hTml\left\{\hat S^*( \hat S_z+\mu)^{-1}S_i + S_i^*(\hat S_z+\mu)^{-1}\hat{S} -\hat S^*(\hat S_z+\mu)^{-1}S_{z,i}(\hat S_z+\mu)^{-1}\hat{S}\right \}(h_0-\hat h ) \\&\pm 2    \hat T_{\mu, \lambda}^{-1}\left\{\hat S^*( \hat S_z+\mu)^{-1}\hat{S}\right\} (h_0 - \hat h) \|_\Hx\\
        &\leq \|\hTml \hat S^*( \hat S_z+\mu)^{-1}(S_i -\hat S)\|_\op\Delta_h+ \|\hTml (S_i-\hat S)^*(\hat S_z+\mu)^{-1}\hat{S} \|_\op\Delta_h \\ &+\|\hTml\hat S^*(\hat S_z+\mu)^{-1}S_{z,i}(\hat S_z+\mu)^{-1}\hat{S}\|_\op\Delta_h 
        +\|2\hat T_{\mu, \lambda}^{-1}\left\{\hat S^*( \hat S_z+\mu)^{-1}\hat{S}\right\}\|_\op \Delta_h\\
        & \leq \frac{1}{2\sqrt{\lambda\mu}}\|S_i -\hat S\|_\op \Delta_h + \frac{ \kappa_x}{\lambda\sqrt{\mu}}\|S_i -\hat S\|_\op \Delta_h + \frac{\kappa_z^2\kappa_x}{2\sqrt{\lambda}\mu}\Delta_h + 2 \Delta_h.
        \end{align*}
    \item[(B)] Next,
    $$\|\hat T_{\mu, \lambda}^{-1}\hat S^*( \hat S_z+\mu)^{-1} \phi(Z_i)\epsilon_i\|_\Hx \leq \frac{1}{2\sqrt{\mu \lambda}}\|\phi(Z_i)\epsilon_i\|_\Hz.$$
    \item[(C)] Finally,
   \begin{align*}
   &\|\hTml \left \{S_i^*(\hat S_z+\mu)^{-1}    -\hat S^*(\hat S_z+\mu)^{-1}S_{z,i}(\hat S_z+\mu)^{-1} \right \}\E_n[\epsilon \phi(Z)]\|_\Hx \\
   &\leq \frac{1}{\lambda }\|S_i^*(\hat S_z+\mu)^{-1} \|_\op \| \E_n[\epsilon \phi(Z)]\| _\Hz+ \frac{1}{2\sqrt{\lambda\mu}} \|S_{z,i}(\hat S_z+\mu)^{-1} \|_\op\|\E_n[\epsilon \phi(Z)]\|_\Hz
   \end{align*}
\end{enumerate}
Recall we have split $\hTml \hat w_i$ into seven terms. Combining them, 
\begin{align*}
     2 \frac{1}{n} \sum_{i = 1}^n \|\hTml\hat w_i\|_\Hx^2 &\leq 2
     \E_n\|\hTml\hat w_i\|_\Hx^2 \\&\leq 14 \Big\{\frac{\kappa^2\Delta_h^2}{\lambda\mu} + \frac{4\kappa^2\kappa_x^2\Delta_h^2}{\lambda^2\mu}  + \frac{\kappa_z^4\kappa_x^2\Delta_h^2}{4\lambda\mu^2} +4\Delta_h^2\\
     &+ \frac{1}{4\mu \lambda}E_n\|\phi(Z_i)\epsilon_i\|_\Hz^2 \\&+  \frac{\| \E_n[\epsilon \phi(Z)]\|_\Hz^2}{\lambda ^2}E_n\|S_i^*(\hat S_z+\mu)^{-1} \|_\op^2  + \frac{\|\E_n[\epsilon \phi(Z)]\|_\Hz^2}{4\lambda\mu} \E_n\|S_{z,i}(\hat S_z+\mu)^{-1} \|_\op^2\Big  \}.
\end{align*}
Now with the events from Lemma \ref{lemma:primitive_bounds}, 
\ref{lemma:gamma_prime}, and
\ref{lemma:delta_prime1}
\begin{align*}
      2  \frac{1}{n} \sum_{i = 1}^n \|\hTml\hat w_i\|_\Hx^2 &\leq 14  \Big\{ \frac{\kappa^2\Delta_h^2}{\lambda\mu} + \frac{4\kappa^2\kappa_x^2\Delta_h^2}{\lambda^2\mu}  + \frac{\kappa_z^4\kappa_x^2 \Delta_h^2}{4\lambda\mu^2} +4\Delta_h^2+ \frac{\gamma^\prime}{4\mu \lambda} + \frac{\tilde \gamma_2^2 \delta_\mu^{\prime \prime}}{\lambda^2} + \frac{\tilde \gamma_2^2 \delta_\mu^{\prime}}{4\lambda\mu} \Big \}.
\end{align*}
Furthermore, by Lemma \ref{lemma:emp_linearization}, if  $\delta_B \in [0,1/2]$, then  with probability $1-\eta$,
\begin{align*}
\|\Delta_1\|_\Hx &= \|(\hat T_{\mu, \lambda}^{-1} - T_{\mu, \lambda}^{-1} )u\|_\Hx = \|( T_{\mu, \lambda}^{-1} -\hat  T_{\mu, \lambda}^{-1} )u\|_\Hx \leq 2\delta_B  \|\hTml u\|_\Hx\\ 
&\leq2\delta_B   \left\{1 + \sqrt{2\log(1/\eta)}\right\} \sqrt{
    14\Big\{ \frac{\kappa^2\Delta_h^2}{\lambda\mu} + \frac{4\kappa^2\kappa_x^2\Delta_h^2}{\lambda^2\mu}  + \frac{\kappa_z^4\kappa_x^2 \Delta_h^2}{4\lambda\mu^2} +4\Delta_h^2+ \frac{\gamma^\prime}{4\mu \lambda} + \frac{\tilde \gamma_2^2 \delta_\mu^{\prime \prime}}{\lambda^2} + \frac{\tilde \gamma_2^2 \delta_\mu^{\prime}}{4\lambda\mu} \Big \}}.
\end{align*}\end{proof}

\begin{lemma}[Inverse estimation error bound]
\label{lemma:SSz_cross_emp}
    Suppose $\|(S_z+\mu)^{-1}(\hat S_z - S_z)\|_{\mathrm{op}} = \delta_z \in  [0, 1/2]$ and $\|\hat{S}-S\|_{\mathrm{op}} \leq \tilde \gamma_1$. Then, for all $u \in \Hz$,
    \[
        \big\|\hat{T}_{\mu,\lambda}^{-1} \{S^*(S_z+\mu)^{-1} - \hat S^*(\hat S_z+\mu)^{-1}\}u\big\|_\Hx
        \leq \frac{\tilde \gamma_1}{\lambda}\|(S_z + \mu)^{-1}u\|_\Hz + \frac{\delta_z}{\sqrt{\lambda\mu}}\|u\|_\Hz.
    \]
\end{lemma}
\begin{proof}
Using the resolvent identity, we can write
    \begin{align*}
    &\|\hTml \{S^*(S_z+\mu)^{-1} - \hat S^*(\hat S_z+\mu)^{-1}\} u\|_\Hx \\
    &=  \|\hTml \{S^*(S_z+\mu)^{-1} - \hat S^*(\hat S_z+\mu)^{-1} \pm \hat S^*( S_z+\mu)^{-1} \}u  \|_\Hx \\
    & = \|\hTml \{(S-\hat S)^*(S_z+\mu)^{-1} - \hat S^*\{(\hat S_z+\mu)^{-1} -( S_z+\mu)^{-1} \}u \}\| _\Hx\\
    &= \|\hTml \{(S-\hat S)^*(S_z+\mu)^{-1} + \hat S^*(\hat S_z+\mu)^{-1}(\hat S_z - S_z)(S_z+\mu)^{-1} u \}\| _\Hx\\
    &\leq \frac{\tilde \gamma_1}{\lambda}\|(S_z + \mu)^{-1}u\|_\Hz + \| \hTml \hat S^* (\hat S_z + \mu)^{-1} (\hat S_z - S_z)(S_z+\mu)^{-1} u\|_\Hx \\
    &\leq  \frac{\tilde \gamma_1}{\lambda}\|(S_z + \mu)^{-1}u\|_\Hz + \frac{1}{2\sqrt{\lambda\mu}}2\delta_z\|u\|_\Hz.
    \end{align*}
\end{proof}
\begin{lemma}[$\Delta_2$-bound]
\label{lemma:DELTA2}
Let $\mu,\lambda>0$ and $
\Delta_2 = \Tml^{-1}\left\{\frac{1}{n}\sum_{i = 1}^n\sum_{j = 1}^n(\hat{w}_{ij}-w_{ij})  h_{ij}\right\}$. Suppose the following bounds hold: \begin{align*}
        \|(S_z+\mu)^{-1}(\hat S_z - S_z)\|_\op \leq \delta_z &\in  [0, 1/2]\\
        \|(\hat S_z -S_z)(S_z  +\mu)^{-1/2}\|_\op  &\leq \delta_z^\prime\\ 
        \|\hat T_{\mu, \lambda}^{-1}(\hat T_\mu - T_\mu)\|_\op  \leq \delta_B &\in  [0, 1/2]\\
        \|S -\hat S\|_\op  &\leq  \tilde \gamma_1\\
        \|\E_n[\phi(Z_i)\eps_i]\|_\Hz 
        &\leq  \tilde \gamma_2\\
        \E_n \|\phi(Z_i)\eps_i\|_\Hz^2 &\leq \gamma^\prime \\
        \E_n \|(S_z+\mu)^{-1}\phi(Z_i)\eps_i\|_\Hz^2 &\leq \gamma^{\prime\prime} \\
        \E_n\|S_{z,i}^*( \hat S_z+\mu)^{-1} \|_\op ^2 &\leq  \delta_\mu^\prime \\
        \E_n\|S_{i}^*( \hat S_z+\mu)^{-1} \|_\op^2 &\leq  \delta_\mu^{\prime \prime}\\
        \mathbb{E}_n \big\| S_i^*(S_z + \mu)^{-1} \big\|_\op ^2 &\leq \delta_\mu^{\prime \prime\prime}\\
        \E_n\|(S_z + \mu)^{-1/2}S_{z,i}\|_\op ^2 &\leq \delta_\mu^{\prime\prime\prime\prime}.
    \end{align*} Then, conditional on data  $(X_i,Z_i,Y_i)_{i=1}^n$, with probability $1-\eta$, 
$$\| \Delta_2 \|_\Hx \leq \left\{1 + \sqrt{2\log(1/\eta)}\right\} \sqrt{\bar \Delta_2}, $$
where
\begin{align*}
\bar  \Delta_2 &= 60(1 + 4\delta_B^2)\Bigg\{\frac{\tilde \gamma_1^2}{\lambda^2}(\delta_\mu^{\prime\prime\prime}  \bml^2 + \gamma^{\prime\prime}) + \frac{\delta_z^2}{\lambda\mu}(\kappa^2\bml^2 + \gamma^\prime)  + \frac{\kappa^2}{4\lambda\mu}\Delta_{\hml}^2 + \frac{1}{\lambda^2}\delta_\mu^{\prime\prime}(\delta_z^\prime)^2\vml^2+ \frac{\tilde \gamma_1^2}{\lambda^2}\delta_\mu^{\prime \prime} \bml^2 + \frac{\kappa^2\kappa_x^2}{\lambda^2 \mu}\Delta_{\hml}^2 \\&+ \frac{\tilde \gamma_2^2}{\lambda^2}\delta_\mu^{\prime \prime} +\frac{\tilde \gamma_1^2}{\lambda^2 \mu^2} \delta_\mu^{\prime\prime\prime\prime} v_{\mu,\lambda}^2 + \frac{\delta_z^2}{\lambda\mu}\delta_\mu^{\prime\prime\prime\prime}\vml^2 + \frac{1}{4\lambda\mu}\delta_\mu^\prime(\delta_z^\prime)^2\vml^2+ \frac{\tilde \gamma_1^2\bml^2}{4\lambda \mu}\delta_\mu^\prime  + \frac{\kappa_z^4\kappa_x^2}{4\lambda\mu^2 }\Delta_{\hml}^2 + \frac{\tilde \gamma_2^2}{4\lambda \mu} \delta_\mu^\prime\Bigg\}  \end{align*} and $\Delta_{h_{\mu, \lambda}} :=  \|h_{\mu, \lambda} - \hat{h}\|_\Hx$.
\end{lemma}
\begin{proof}
Let $u = \frac{1}{n}\sum_{i = 1}^n\sum_{j = 1}^n(\hat{w}_{ij}-w_{ij})  h_{ij}.$ Conditional on the data, with probability $1-\eta$,
$$
\| \Tml^{-1} u \|_\Hx \leq \left\{1 + \sqrt{2\log(1/\eta)}\right\} \sqrt{\E\|\Tml^{-1} u\|_\Hx^2}. $$
Furthermore, $\E_h \|\Tml^{-1}u \|_\Hx^2 = \frac{1}{n^2}\sum_{i = 1}^n\sum_{j = 1}^n\|\Tml^{-1}(\hat{w}_{ij}-w_{ij})\|_\Hx^2.$ Note that 
    $$\hat w_{ij} - w_{ij}  = \frac{\hat w_i - w_i}{\sqrt{2}} - \frac{\hat w_j - w_j}{\sqrt{2}}.$$
Considering that each $\hat w_i - w_i$ can be written as
\begin{align*}
& S^*(S_z+\mu)^{-1}\phi(Z_i)(Y_i - h_{\mu,\lambda}(X_i)) + S_i^*(S_z+\mu)^{-1}S(h_0-h_{\mu,\lambda})- S^*(S_z+\mu)^{-1}S_{z,i}(S_z+\mu)^{-1}S  (h_0-h_{\mu,\lambda})\\
    &-\{\hat S^*( \hat S_z+\mu)^{-1}\phi(Z_i)\hat \epsilon_i + S_i^*(\hat S_z+\mu)^{-1}\E_n[\hat \varepsilon\phi(Z)]
-\hat S^*(\hat S_z+\mu)^{-1}S_{z,i}(\hat S_z+\mu)^{-1}\E_n[\hat \varepsilon\phi(Z)]\} \\
&= S^*(S_z+\mu)^{-1}\phi(Z_i)(Y_i - h_{\mu,\lambda}(X_i)) + S_i^*(S_z+\mu)^{-1}S(h_0-h_{\mu,\lambda})- S^*(S_z+\mu)^{-1}S_{z,i}(S_z+\mu)^{-1}S  (h_0-h_{\mu,\lambda})\\
    &-\{\hat S^*( \hat S_z+\mu)^{-1}\phi(Z_i)(Y_i - \hat h(X_i)) + S_i^*(\hat S_z+\mu)^{-1}\E_n[(Y - \hat h(X))\phi(Z)]\\&\quad
-\hat S^*(\hat S_z+\mu)^{-1}S_{z,i}(\hat S_z+\mu)^{-1}\E_n[ (Y_i - \hat h(X))\phi(Z)]\}.
\end{align*}
Within each term of $\|\Tml^{-1}(\hat w_{ij}-w_{ij})\|_\Hx^2$,
\begin{align*}
\|\Tml^{-1}(\hat w_{ij}-w_{ij})\|_\Hx^2 = \frac{1}{2}\|\Tml^{-1}\{(\hat w_i - w_i) - (\hat w_j - w_j)\}\|_\Hx^2 \leq \|\Tml^{-1} \Delta w_i\|_\Hx^2+ \|\Tml^{-1}  \Delta w_j\|_\Hx^2 .
\end{align*}
Using the parallelogram law again,
\begin{align*}
    \E_h \|T_{\mu,\lambda}^{-1}u \|_\Hx^2
    &= \frac{1}{n^2}\sum_{i = 1}^n\sum_{j = 1}^n\|T_{\mu,\lambda}^{-1}(\hat{w}_{ij}-w_{ij})\|_\Hx^2
    \leq \frac{2}{n}\sum_{i= 1}^{n}\|T_{\mu,\lambda}^{-1}\Delta w_i\|_\Hx^2 \\
    &\leq \frac{2}{n}\sum_{i= 1}^{n}\left\{ 2\|(T_{\mu,\lambda}^{-1} - \hat{T}_{\mu,\lambda}^{-1})\Delta w_i\|_\Hx^2 + 2\|\hat{T}_{\mu,\lambda}^{-1} \Delta w_i\|_\Hx^2 \right\} \\
    &\leq \frac{2}{n}\sum_{i= 1}^{n} (8\delta_B^2 + 2)\|\hat{T}_{\mu,\lambda}^{-1} \Delta w_i\|_\Hx^2.
\end{align*}
We further decompose 
\begin{align*}
\hTml\Delta w_i
= \hTml&\Bigg\{ \underbrace{
\Big[
S^*(S_z+\mu)^{-1}
-
\hat S^*(\hat S_z+\mu)^{-1}
\Big]
\phi(Z_i)
\big(Y_i - h_{\mu,\lambda}(X_i)\big)}_A\\
&+
\underbrace{
\hat S^*( \hat S_z+\mu)^{-1}
\phi(Z_i)
\big(\hat h(X_i)-h_{\mu,\lambda}(X_i)\big)}_B\\
&+
\underbrace{
S_i^*\Big[(S_z+\mu)^{-1}-(\hat S_z+\mu)^{-1}\Big]
S\big(h_0-h_{\mu,\lambda}\big)}_{C}\\
&+\underbrace{
S_i^*(\hat S_z+\mu)^{-1}
\Big(
S\big(h_0-h_{\mu,\lambda}\big)
-E_n[\big(Y-\hat h(X)\big)\phi(Z)]
\Big)}_D\\
&-\underbrace{\Big[S^*(S_z+\mu)^{-1}
-
\hat S^*(\hat S_z+\mu)^{-1}
\Big]
S_{z,i}
(S_z+\mu)^{-1}
S\big(h_0-h_{\mu,\lambda}\big)}_E\\
&-\underbrace{\hat S^*(\hat S_z+\mu)^{-1}
S_{z,i}
\Big[
(S_z+\mu)^{-1}-(\hat S_z+\mu)^{-1}
\Big]
S\big(h_0-h_{\mu,\lambda}\big)}_F\\
& -\underbrace{\hat S^* (\hat S_z+\mu)^{-1}
S_{z,i}
(\hat S_z+\mu)^{-1}
\Big(
S\big(h_0-h_{\mu,\lambda}\big)
-\E_n[
\big(Y-\hat h(X)\big)\phi(Z)]
\Big)\Bigg\}}_G.
\end{align*}

\begin{enumerate}
    \item[(A)] 
    By Lemma \ref{lemma:SSz_cross_emp},
    \begin{align*}
       \|\hTml \Big[
S^*(S_z+\mu)^{-1}
-
\hat S^*(\hat S_z+\mu)^{-1}
\Big]
\phi(Z_i)
\big(Y_i - h_{\mu,\lambda}(X_i)\big)\|_\Hx &\leq \frac{\tilde \gamma_1}{\lambda}\|\Szmu\phi(Z_i)
\big(Y_i - h_{\mu,\lambda}(X_i)\big)\|_\Hz\\
&+ \frac{\delta_z}{\sqrt{\lambda\mu}} \|\phi(Z_i)
\big(Y_i - h_{\mu,\lambda}(X_i)\big)\|_\Hz.
    \end{align*}
Using that 
$\phi(Z_i)
\big(Y_i - h_{\mu,\lambda}(X_i) = S_i(h_0 - h_{\mu,\lambda}) +\phi(Z_i)\eps_i$. The first term becomes
$$\frac{\tilde \gamma_1}{\lambda}\|\Szmu\phi(Z_i)
\big(Y_i - h_{\mu,\lambda}(X_i)\big)\|_\Hz\leq \frac{\tilde \gamma_1}{\lambda}\left \{\|\Szmu S_i\|_\op b_{\mu, \lambda} +  \|\Szmu\phi(Z_i)\eps_i\|_\Hz \right \}.$$
For the second term,
$$
\frac{\delta_z}{\sqrt{\lambda\mu}} \|\phi(Z_i)
\big(Y_i - h_{\mu,\lambda}(X_i)\big)\|_\Hz
\leq \frac{\delta_z}{\sqrt{\lambda\mu}} \{\kappa b_{\mu, \lambda} +  \|\phi(Z_i)\eps_i\|_\Hz\}.$$
Together, this yields
 \begin{align*}
    & \|\hTml \Big[
S^*(S_z+\mu)^{-1}
-
\hat S^*(\hat S_z+\mu)^{-1}
\Big]
\phi(Z_i)
\big(Y_i - h_{\mu,\lambda}(X_i)\big)\|_\Hx  \\&\leq \frac{\tilde \gamma_1}{\lambda}\left \{\|\Szmu S_i\|_\op b_{\mu, \lambda} +  \|\Szmu\phi(Z_i)\eps_i\|_\Hz \right \} + \frac{\delta_z}{\sqrt{\lambda\mu}} \{\kappa b_{\mu, \lambda} +  \|\phi(Z_i)\eps_i\|_\Hz\} .
\end{align*}
\item[(B)] Next, 
\begin{align*}
    \|\hTml\hat S^*( \hat S_z+\mu)^{-1}
\phi(Z_i)
\big(\hat h(X_i)-h_{\mu,\lambda}(X_i)\big)\|_\Hx = \|\hTml\hat S^*( \hat S_z+\mu)^{-1}
S_i(\hat h - h_{\mu, \lambda})\big)\|_\Hx \leq \frac{\kappa}{2\sqrt{\lambda\mu}}\Delta_{h_{\mu,\lambda}}.
\end{align*}
    \item[(C)] Next, 
    \begin{align*}
        \|\hTml S_i^*\Big[(S_z+\mu)^{-1}-(\hat S_z+\mu)^{-1}\Big]
S\big(h_0-h_{\mu,\lambda}\big)\|_\Hx &=   \|\hTml S_i^*(\hat S_z+\mu)^{-1}(S_z-\hat S_z) ( S_z+\mu)^{-1}
S\big(h_0-h_{\mu,\lambda}\big)\|_\Hx \\
&\leq \frac{1}{\lambda}\|S_i^* (\hat{S_z} +\mu)^{-1}\|_\op\delta_z^\prime v_{\mu,\lambda}.
    \end{align*}
    \item[(D)] Next, 
    \begin{align*}
        &\|\hTml S_i^*(\hat S_z+\mu)^{-1}
\Big(
S\big(h_0-h_{\mu,\lambda}\big)
-E_n[
\big(Y-\hat h(X)\big)\phi(Z)]
\Big)\|_\Hx\\
&= \|\hTml S_i^*(\hat S_z+\mu)^{-1}
\Big\{(S- \hat S)\big(h_0-h_{\mu,\lambda}\big) - \hat{S} (h_{\mu, \lambda} - \hat h)  - \E_n(\epsilon\phi(Z))\Big\}\|_\Hx \\
&\leq \frac{\tilde \gamma_1}{\lambda} \| S_i^*(\hat S_z+\mu)^{-1}\|_\op b_{\mu, \lambda}  +\frac{\kappa\kappa_x}{\lambda\sqrt{\mu}}\Delta_{h_{\mu,\lambda}} + \frac{\tilde \gamma_2}{\lambda}\| S_i^*(\hat S_z+\mu)^{-1}\|_\op .
    \end{align*}
    \item[(E)]
    By Lemma \ref{lemma:SSz_cross_emp}, 
\begin{align*}
   &\| \hTml\Big[S^*(S_z+\mu)^{-1}
-
\hat S^*(\hat S_z+\mu)^{-1}
\Big]
S_{z,i}
(S_z+\mu)^{-1}
S\big(h_0-h_{\mu,\lambda}\big)\|_\Hx \\
&\leq \frac{\tilde \gamma_1}{\lambda} \|\Szmu S_{z,i}
(S_z+\mu)^{-1}
S\big(h_0-h_{\mu,\lambda})\|_\Hz+ \frac{\delta_z}{\sqrt{\lambda\mu}}\|S_{z,i}
(S_z+\mu)^{-1}
S\big(h_0-h_{\mu,\lambda})\|_\Hz \\
&\leq \frac{\tilde \gamma_1}{\lambda \mu} \|(S_z + \mu)^{-1/2} S_{z,i}\|_\op v_{\mu,\lambda} + \frac{\delta_z}{\sqrt{\lambda\mu}}\|(S_z + \mu)^{-1/2}S_{z,i}\|_\op v_{\mu, \lambda}.\end{align*} 
    \item[(F)] Next, 
    \begin{align*}
        \| \hTml \hat S^*(\hat S_z+\mu)^{-1}
S_{z,i}
\Big[
(S_z+\mu)^{-1}-(\hat S_z+\mu)^{-1}
\Big]
S\big(h_0-h_{\mu,\lambda}\big)\|_\Hx \leq \frac{1}{2\sqrt{\lambda\mu}} \|S_{z,i}(\hat S_z +\mu)^{-1}\|_\op\delta_z^\prime v_{\mu,\lambda}.
    \end{align*}
    \item[(G)]  Finally,
    \begin{align*}
   & \|\hTml\hat S^* (\hat S_z+\mu)^{-1}
S_{z,i}
(\hat S_z+\mu)^{-1}
\Big(
S\big(h_0-h_{\mu,\lambda}\big)
-\E_n[
\big(Y-\hat h(X)\big)\phi(Z)]
\Big)\|_\Hx \\
&\leq\frac{1}{2\sqrt{\lambda\mu}}\left\{ \|S_{z,i}
(\hat S_z+\mu)^{-1}(S- \hat S)\big(h_0-h_{\mu,\lambda}\big) \|_\Hz+ \|S_{z,i}
(\hat S_z+\mu)^{-1} \hat{S} (h_{\mu, \lambda} - \hat h)\|_\Hz +\| S_{z,i}
(\hat S_z+\mu)^{-1}\E_n(\epsilon\phi(Z))\|_\Hz \right\} \\
&\leq \frac{\tilde\gamma_1 b_{\mu,\lambda}}{2\sqrt{\lambda\mu}} \|S_{z,i}
(\hat S_z+\mu)^{-1}\|_\op +  \frac{\kappa_z^2\kappa_x}{2\sqrt{\lambda}\mu}\Delta_{h_{\mu,\lambda}} + \frac{\tilde \gamma_2}{2\sqrt{\mu \lambda}} \|S_{z,i}
(\hat S_z+\mu)^{-1}\|_\op.\end{align*}
\end{enumerate}
We now combine these bounds to control the term $\frac{2}{n}\sum_{i= 1}^{n} (2 + 8\delta_B^2)\|\hTml \Delta w_i\|_\Hx^2$. Although $\hTml \Delta w_i$ was initially decomposed into seven groups (A--G), further expansion yields a total of 15 additive terms. Applying the inequality $\|\sum_{k=1}^m v_k\|^2 \leq m \sum_{k=1}^m \|v_k\|^2$ with $m=15$, we bound the squared norm of the sum by 15 times the sum of the squared norms:
\begin{align*}
&\frac{2}{n}\sum_{i= 1}^{n} (2 + 8\delta_B^2)\|\hTml \Delta w_i\|_\Hx^2 \\
&\leq 60(1 + 4\delta_B^2)\Bigg\{\frac{\tilde \gamma_1^2}{\lambda^2}(\delta_\mu^{\prime\prime\prime} \bml^2 + \gamma^{\prime\prime}) + \frac{\delta_z^2}{\lambda\mu}(\kappa^2\bml^2 + \gamma^\prime)  + \frac{\kappa^2}{4\lambda\mu}\Delta_{\hml}^2 + \frac{1}{\lambda^2}\delta_\mu^{\prime\prime}(\delta_z^\prime)^2\vml^2+ \frac{\tilde \gamma_1^2}{\lambda^2}\delta_\mu^{\prime \prime} \bml^2 + \frac{\kappa^2\kappa_x^2}{\lambda^2 \mu}\Delta_{\hml}^2 + \frac{\tilde \gamma_2^2}{\lambda^2}\delta_\mu^{\prime \prime} \\
&+\frac{\tilde \gamma_1^2}{\lambda^2 \mu^2} \delta_\mu^{\prime\prime\prime\prime} v_{\mu,\lambda}^2 + \frac{\delta_z^2}{\lambda\mu}\delta_\mu^{\prime\prime\prime\prime}\vml^2 + \frac{1}{4\lambda\mu}\delta_\mu^\prime(\delta_z^\prime)^2\vml^2+ \frac{\tilde \gamma_1^2\bml^2}{4\lambda \mu}\delta_\mu^\prime  + \frac{\kappa_z^4\kappa_x^2}{4\lambda\mu^2 }\Delta_{\hml}^2 + \frac{\tilde \gamma_2^2}{4\lambda \mu} \delta_\mu^\prime\Bigg\} =:
\bar \Delta_2.
\end{align*}
Therefore, conditional on the data and the assumed high probability events, with probability $1-\eta$,
$$
    \| \Delta_2 \|_\Hx \leq \left\{1 + \sqrt{2\log(1/\eta)}\right\} \sqrt{\bar \Delta_2}. $$
\end{proof}
\subsection{Main result}
\begin{lemma}[Decomposition of the bootstrap error]
    $\mathfrak{B} - Z_{\mathfrak{B}} = \Delta_1 + \Delta_2,$
where $\Delta_1 = (\hat T_{\mu, \lambda}^{-1} - T_{\mu, \lambda}^{-1} )(\frac{1}{n}\sum_{i = 1}^n\sum_{j = 1}^n\hat{w}_{ij}  h_{ij})$ and $\Delta_2 = \Tml^{-1}\left\{\frac{1}{n}\sum_{i = 1}^n\sum_{j = 1}^n(\hat{w}_{ij}-w_{ij})  h_{ij}\right\}.$
\end{lemma}
\begin{proof}
    The proof follows by matching symbols with Lemma G.3 in \cite{singh2023kernel}.
\end{proof}

\begin{theorem}[Feasible bootstrap]
\label{Theorem:feasible_bootstrap}
  Suppose $n\geq \max\{N_{\delta_z}, 3N_\delta\}$ such that $\delta_B = 3\delta \le 1/2 $ and $\delta_z\in  [0, 1/2]$. Then with probability  $1 -15\eta$,
 $$\|\mathfrak{B} - Z_{\mathfrak{B}}\|_\Hx \leq C_\eta \Bigg[
\frac{1}{\lambda} A
+\frac{1}{\lambda\sqrt{\mu}} B
+\frac{1}{\sqrt{\lambda\mu}} C
+\frac{1}{\sqrt{\lambda}\mu} D
\Bigg],$$
where
\begin{align*}
A &= 2\delta_B\sqrt{14}\Big(\tilde\gamma_2\sqrt{\delta_\mu^{\prime\prime}}\Big)
   + \sqrt{60(1 + 4\delta_B^2)}\Big(
      \tilde\gamma_1(\bml\sqrt{\delta_\mu^{\prime\prime\prime} }+\sqrt{\gamma^{\prime\prime}})
    + \delta_z^\prime\vml\sqrt{\delta_\mu^{\prime\prime}}
    + \tilde\gamma_1\bml\sqrt{\delta_\mu^{\prime\prime}}
    + \tilde\gamma_2\sqrt{\delta_\mu^{\prime\prime}}
   \Big), \\[1.0em]
B &= 2\delta_B\sqrt{14}\Big(2\kappa{\kappa_x}\Delta_h\Big)
   + \sqrt{60(1 + 4\delta_B^2)}\Big(\kappa_x
      \kappa\Delta_{\hml}
    + \mu^{-1/2}\tilde\gamma_1v_{\mu,\lambda}\sqrt{\delta_\mu^{\prime\prime\prime\prime}}
   \Big), \\[1.0em]
C &= 2\delta_B\sqrt{14}\Big(
      \kappa\Delta_h
    + \tfrac12\sqrt{\gamma^\prime}
    + \tfrac12\tilde\gamma_2\sqrt{\delta_\mu^\prime}
    + 2\Delta_h\sqrt{\lambda\mu}
   \Big) \\
&\quad + \sqrt{60(1 + 4\delta_B^2)}\Big(
      \delta_z(\kappa\bml+\sqrt{\gamma^\prime})
    + \tfrac12\kappa\Delta_{\hml}
    + \delta_z\vml\sqrt{\delta_\mu^{\prime\prime\prime\prime}}
    + \tfrac12\delta_z^\prime\vml\sqrt{\delta_\mu^\prime}
    + \tfrac12\tilde\gamma_1\bml\sqrt{\delta_\mu^\prime}
    + \tfrac12\tilde\gamma_2\sqrt{\delta_\mu^\prime}
   \Big), \\[1.0em]
D &= 2\delta_B\sqrt{14}\Big(\tfrac12\kappa_z^2{\kappa_x}\Delta_h\Big)
   + \sqrt{60(1 + 4\delta_B^2)}\Big(\tfrac12\kappa_x\kappa_z^2\Delta_{\hml}\Big).
\end{align*}
\end{theorem}
\begin{proof}
Let $C_\eta = \left\{1 + \sqrt{2\log(1/\eta)}\right\}$
We first collect the events and then determine the probability. Using Lemma \ref{lemma:DELTA1} and Lemma \ref{lemma:DELTA2}, with probability $1-2\eta$, 
\begin{align*}
\| \mathfrak{B} - Z_{\mathfrak{B}}\|_\Hx
&\le \| \Delta_1\|_\Hx + \| \Delta_2\|_\Hx \\
&\le 2\delta_B C_\eta \sqrt{14\Big(
  \frac{\kappa^2\Delta_h^2}{\lambda\mu} +\frac{4\kappa^2\kappa_x^2\Delta_h^2}{\lambda^2\mu}
 +\frac{\kappa_z^4\kappa_x^2 \Delta_h^2}{4\lambda\mu^2}
 + 4\Delta_h^2
 +\frac{\gamma^\prime}{4\lambda \mu}
 +\frac{\tilde \gamma_2^2 \delta_\mu^{\prime \prime}}{\lambda^2}
 +\frac{\tilde \gamma_2^2 \delta_\mu^{\prime}}{4\lambda\mu}
\Big)} \\
&\qquad + C_\eta \Big[60(1 + 4\delta_B^2)\Big\{
  \frac{\tilde \gamma_1^2}{\lambda^2}(\delta_\mu^{\prime\prime\prime}  \bml^2 + \gamma^{\prime\prime})
 +\frac{\delta_z^2}{\lambda\mu}(\kappa^2\bml^2 + \gamma^\prime)
 +\frac{\kappa^2}{4\lambda\mu}\Delta_{\hml}^2
 +\frac{1}{\lambda^2}\delta_\mu^{\prime\prime}(\delta_z^\prime)^2\vml^2 \\
&\qquad\qquad\qquad\qquad
 +\frac{\tilde \gamma_1^2} {\lambda^2}\delta_\mu^{\prime \prime}\bml^2
 +\frac{\kappa^2\kappa_x^2}{\lambda^2 \mu}\Delta_{\hml}^2
 +\frac{\tilde \gamma_2^2}{\lambda^2}\delta_\mu^{\prime \prime}
 +\frac{\tilde \gamma_1^2}{\lambda^2 \mu^2}\delta_\mu^{\prime\prime\prime\prime} v_{\mu,\lambda}^2
 +\frac{\delta_z^2}{\lambda\mu}\delta_\mu^{\prime\prime\prime\prime}\vml^2 \\
&\qquad\qquad\qquad\qquad
 +\frac{1}{4\lambda\mu}\delta_\mu^\prime(\delta_z^\prime)^2\vml^2
 +\frac{\tilde \gamma_1^2\bml^2}{4\lambda \mu}\delta_\mu^\prime
 +\frac{\kappa_z^4\kappa_x^2}{4\lambda\mu^2}\Delta_{\hml}^2
 +\frac{\tilde \gamma_2^2}{4\lambda \mu}\delta_\mu^\prime
\Big\} \Big]^{1/2} \\[0.4em]
&\le C_\eta \Bigg[
  2\delta_B \sqrt{14} \Bigg(
     \frac{\kappa}{\sqrt{\lambda\mu}}\Delta_h
   + \frac{2\kappa\kappa_x}{\lambda\sqrt{\mu}}\Delta_h
   + \frac{\kappa_z^2\kappa_x}{2\sqrt{\lambda}\mu}\Delta_h
   + \frac{2\Delta_h\sqrt{\lambda\mu}}{\sqrt{\lambda\mu}}
   + \frac{\sqrt{\gamma^\prime}}{2\sqrt{\lambda\mu}}
   + \frac{\tilde\gamma_2}{\lambda}\sqrt{\delta_\mu^{\prime\prime}}
   + \frac{\tilde\gamma_2}{2\sqrt{\lambda\mu}}\sqrt{\delta_\mu^\prime}
  \Bigg) \\
&\qquad\qquad
 + \sqrt{60(1 + 4\delta_B^2)} \Bigg(
     \frac{\tilde\gamma_1}{\lambda}\big(\bml\sqrt{\delta_\mu^{\prime\prime\prime} }+\sqrt{\gamma^{\prime\prime}}\big)
   + \frac{\delta_z}{\sqrt{\lambda\mu}}\big(\kappa\bml+\sqrt{\gamma^\prime}\big)
   + \frac{\kappa}{2\sqrt{\lambda\mu}}\Delta_{\hml}
\\
&\qquad\qquad
   + \frac{1}{\lambda}\delta_z^\prime\vml\sqrt{\delta_\mu^{\prime\prime}}
   + \frac{\tilde\gamma_1}{\lambda}\bml\sqrt{\delta_\mu^{\prime\prime}}
+ \frac{\kappa\kappa_x}{\lambda\sqrt{\mu}}\Delta_{\hml} 
   + \frac{\tilde\gamma_2}{\lambda}\sqrt{\delta_\mu^{\prime\prime}}
   + \frac{\tilde\gamma_1}{\lambda{\mu}}v_{\mu,\lambda}\sqrt{\delta_\mu^{\prime\prime\prime\prime}}
   + \frac{\delta_z}{\sqrt{\lambda\mu}}\vml\sqrt{\delta_\mu^{\prime\prime\prime\prime}}\\
&\qquad\qquad
   + \frac{\delta_z^\prime\vml}{2\sqrt{\lambda\mu}}\sqrt{\delta_\mu^\prime}
   + \frac{\tilde\gamma_1\bml}{2\sqrt{\lambda\mu}}\sqrt{\delta_\mu^\prime} 
   + \frac{\kappa_x\kappa_z^2}{2\sqrt{\lambda}\mu}\Delta_{\hml}
   + \frac{\tilde\gamma_2}{2\sqrt{\lambda\mu}}\sqrt{\delta_\mu^\prime}
  \Bigg)
\Bigg] \\[0.4em]
&= C_\eta \Bigg[
   \frac{1}{\lambda} A
 + \frac{1}{\lambda\sqrt{\mu}} B
 + \frac{1}{\sqrt{\lambda\mu}} C
 + \frac{1}{\sqrt{\lambda}\mu} D
\Bigg].
\end{align*}
Here,
\begin{align*}
A &= 2\delta_B\sqrt{14}\Big(\tilde\gamma_2\sqrt{\delta_\mu^{\prime\prime}}\Big)
   + \sqrt{60(1 + 4\delta_B^2)}\Big(
      \tilde\gamma_1(\bml\sqrt{\delta_\mu^{\prime\prime\prime} }+\sqrt{\gamma^{\prime\prime}})
    + \delta_z^\prime\vml\sqrt{\delta_\mu^{\prime\prime}}
    + \tilde\gamma_1\bml\sqrt{\delta_\mu^{\prime\prime}}
    + \tilde\gamma_2\sqrt{\delta_\mu^{\prime\prime}}
   \Big), \\[1.0em]
B &= 2\delta_B\sqrt{14}\Big(2\kappa{\kappa_x}\Delta_h\Big)
   + \sqrt{60(1 + 4\delta_B^2)}\Big(\kappa_x
      \kappa\Delta_{\hml}
    + \mu^{-1/2}\tilde\gamma_1v_{\mu,\lambda}\sqrt{\delta_\mu^{\prime\prime\prime\prime}}
   \Big), \\[1.0em]
C &= 2\delta_B\sqrt{14}\Big(
      \kappa\Delta_h
    + \tfrac12\sqrt{\gamma^\prime}
    + \tfrac12\tilde\gamma_2\sqrt{\delta_\mu^\prime}
    + 2\Delta_h\sqrt{\lambda\mu}
   \Big) \\
&\quad + \sqrt{60(1 + 4\delta_B^2)}\Big(
      \delta_z(\kappa\bml+\sqrt{\gamma^\prime})
    + \tfrac12\kappa\Delta_{\hml}
    + \delta_z\vml\sqrt{\delta_\mu^{\prime\prime\prime\prime}}
    + \tfrac12\delta_z^\prime\vml\sqrt{\delta_\mu^\prime}
    + \tfrac12\tilde\gamma_1\bml\sqrt{\delta_\mu^\prime}
    + \tfrac12\tilde\gamma_2\sqrt{\delta_\mu^\prime}
   \Big), \\[1.0em]
D &= 2\delta_B\sqrt{14}\Big(\tfrac12\kappa_z^2{\kappa_x}\Delta_h\Big)
   + \sqrt{60(1 + 4\delta_B^2)}\Big(\tfrac12\kappa_x\kappa_z^2\Delta_{\hml}\Big).
\end{align*}
Now, the events with bounds $\delta_\mu^{\prime}, \delta_\mu^{\prime\prime}, \delta_\mu^{\prime\prime\prime}, \delta_\mu^{\prime\prime\prime\prime}, \tilde \gamma_1, \tilde \gamma_2, \tilde \gamma$ each hold with probability $1-\eta$. Moreover, $\gamma^\prime, \gamma^{\prime\prime}, \delta_z,\delta_z^\prime$ also hold with probability $1-\eta$. Notice that $\delta_B$  depends only on $\delta$ (Lemma \ref{lemma:delta}) which holds with $1-\eta$ since we have already accounted for  $\tilde \gamma_1$ and $\delta_z$.
When using the consistency bounds from Lemma \ref{lemma:consistency}, the bound holds with probability $1-13\eta $ when taking into account the $2\eta $ from before. For $\Delta_{h_{\mu, \lambda}} $, we need to consider the unconsidered events from our Bahadur representation.
Note $\|\hat h- \hml \|_\Hx\leq \| \E_n(U_i) \|_\Hx + \Delta_U $
with
\begin{align*}
    \|\E_n(U_i)\|_\Hx&\leq \frac{1}{2\sqrt{\lambda\mu}}\tilde \gamma_1\bml + \frac{1}{\lambda}\delta_z^{\prime\prime} \vml + \frac{1}{2\sqrt{\lambda\mu}}\tilde \gamma_2 + \frac{\kappa_x}{2\sqrt{\lambda\mu}}\delta_z^\prime \bml.
\end{align*}
Here, only $\delta_z^{\prime\prime}$ is new. 
In $\Delta_U$ (see Theorem \ref{theorem:bahadur}) only $\tilde \gamma_3$ (through  $\gamma$) is unaccounted for. 
Thus, we can use the derived bound for $\Delta_U$ with an additional probability of $1-\eta$ and
$$\Delta_U\leq 2\delta\gamma + \gamma_1 + \xi_1+ \left[2\delta^2 +
\frac{\delta_z}{\sqrt{\lambda\mu}}\{ \delta_z(\kappa +  \tilde \gamma_1) + \frac{\tilde\gamma_1}{2}\} + \frac{\tilde \gamma_1}{\lambda\mu}\{\tilde \gamma_1 + \delta_z(1+2\delta_z)(\kappa+ \tilde \gamma_1)\} \right]  \| h_0-h_{\mu,\lambda}\|_\Hx.
$$
Given that the probabilistic part of the bound of $\Delta_h$ is only $\Delta_{h_{\mu, \lambda}}$ we can bound the overall $\|\mathfrak{B} - Z\|_\Hx$ with the above events and probability $1-15\eta$.
\end{proof}

\section{Gaussian coupling}
\subsection{Polynomial decay}
\begin{assumption}[Spectrum of $S_x$]
\label{ass:lwx}Let $S_x=\E \{\psi(X)\otimes \psi(X)^*\}$  be the  positive, self-adjoint covariance operator acting on $\Hx$ with eigenvalues $\nu_s(S_x)$ arranged in nonincreasing order. The spectrum of $S_x$ decays polynomially, i.e.
$$
\nu_s(S_x)\asymp \omega_x s^{-1/(\rho_x-1)},
$$
with $\rho_x \in [1,2].$
\end{assumption}

\begin{assumption}[Spectrum of \(S_z\)] \label{ass:lwz} Let $S_z=\E \{\phi(Z)\otimes \phi(Z)^*\}$be a positive, self-adjoint covariance operator acting on $\Hz$ with eigenvalues $\nu_s(S_z)$ arranged in nonincreasing order. The spectrum of $S_z$ decays polynomially, i.e.
$$
\nu_s(S_z) \asymp \omega_z s^{-1/(\rho_z -1)},
$$
with $\rho_z \in [1,2].$
\end{assumption}

\begin{lemma}[Effective dimension bound]\label{lemma:effective_dim_bound}
Let $\lambda>0$ and $\mu>0$.
Suppose that Assumptions~\ref{ass:lwx} and~\ref{ass:lwz} hold with
exponents $\rho_x,\rho_z \in [1,2]$ and constants $\omega_x,\omega_z$.
Then there exist positive constants
$C_z = C_z(\rho_z,\omega_z)$ and $C_x = C_x(\rho_x,\omega_x)$ such that
\begin{align*}
  \mathfrak{n}_z(\mu)
  &:= \tr\bigl((S_z+\mu I)^{-2} S_z\bigr)
    \le C_z\mu^{-\rho_z},\\
  \mathfrak{m}(\lambda,\mu)
  &:= \tr\bigl((T_\mu+\lambda I)^{-2} T_\mu\bigr)
    \le C_x\lambda^{-\rho_x},\\
  \tilde{\mathfrak{m}}(\lambda,\mu)
  &:= \tr\bigl(
        T_{\mu,\lambda}^{-1} S^*(S_z+\mu I)^{-1} S_z
        (S_z+\mu I)^{-1} S T_{\mu,\lambda}^{-1}
      \bigr)
    \le C_x\lambda^{-\rho_x}.
\end{align*}
In particular, this means
\[
  \mathfrak{n}_z(\mu) \lesssim_{\rho_z,\omega_z} \mu^{-\rho_z},
  \qquad
  \mathfrak{m}(\lambda,\mu) \lesssim_{\rho_x,\omega_x} \lambda^{-\rho_x},
  \qquad
  \tilde{\mathfrak{m}}(\lambda,\mu) \lesssim_{\rho_x,\omega_x} \lambda^{-\rho_x}.
\]
\end{lemma}

\begin{proof}
The bound for $\mathfrak{n}_z(\mu)$ follows from Proposition~K.2 of
\cite{singh2023kernel} applied to $S_z$; in the notation of that paper,
set $\rho_z = 1 + 1/\beta$. This yields
\[
  \mathfrak{n}_z(\mu)
  = \tr\bigl((S_z+\mu I)^{-2} S_z\bigr)
  \le C_z\mu^{-\rho_z}
\]
for a constant $C_z = C_z(\rho_z,\omega_z)$. For $\mathfrak{m}(\lambda,\mu)$, note that
\[
  T_\mu = S^*(S_z+\mu I)^{-1}S,
  \qquad
  T = S^* S_z^{-1} S,
\]
and $(S_z+\mu I)^{-1} \preceq S_z^{-1}$ implies $T_\mu \preceq T$.
Moreover, by Lemma~\ref{lemma:X>T} one has $T \preceq S_x$.
Since $T_\mu$ is positive and $T_\mu \preceq S_x$, it follows that
\[
  \mathfrak{m}(\lambda,\mu)
  = \tr\bigl((T_\mu+\lambda I)^{-2} T_\mu\bigr)
  \le \frac{1}{\lambda}
      \tr\bigl((T_\mu+\lambda I)^{-1}T_\mu\bigr)
  \le \frac{1}{\lambda}
      \tr\bigl((S_x+\lambda I)^{-1}S_x\bigr).
\]
Define the generalized effective dimension of $S_x$ by
\[
  \psi(m,c)
  := \sum_{s=m+1}^\infty \frac{\nu_s}{(\nu_s+\lambda)^c},
\]
where $(\nu_s)$ denotes the eigenvalues of $S_x$.
In this notation, $\tr\bigl((S_x+\lambda I)^{-1}S_x\bigr) = \psi(0,1),$
and Proposition~K.2 of \cite{singh2023kernel} gives $
  \psi(0,1) \le C_x\lambda^{-(\rho_x-1)}$.
Hence $
  \mathfrak{m}(\lambda,\mu)
  \le \frac{1}{\lambda}\psi(0,1)
  \le C_x\lambda^{-\rho_x}.
$
Finally,
\[
  \tilde{\mathfrak{m}}(\lambda,\mu)
  = \tr\bigl(
      T_{\mu,\lambda}^{-1} S^*(S_z+\mu I)^{-1} S_z
      (S_z+\mu I)^{-1} S T_{\mu,\lambda}^{-1}
    \bigr)
  \le \tr\bigl(
      T_{\mu,\lambda}^{-1} S^*(S_z+\mu I)^{-1} S T_{\mu,\lambda}^{-1}
    \bigr)
  = \mathfrak{m}(\lambda,\mu),
\]
where the inequality uses $0 \preceq S_z \preceq S_z+\mu I$.
The previous bound on $\mathfrak{m}(\lambda,\mu)$ then yields
\(
  \tilde{\mathfrak{m}}(\lambda,\mu) \le C_x\lambda^{-\rho_x}.
\)
\end{proof}

\subsection{Main result}
\begin{remark}[Rate of $\Delta_{h_{\mu,\lambda}}$]
Let $b_{\mu,\lambda}$ and $v_{\mu,\lambda}$ denote the bias and projected bias,
respectively, as in Lemma~\ref{lemma:consistency}. Under Assumption~\ref{ass:lwz}, Proposition~K.2 of \cite{singh2023kernel} applied with $c=1$ gives
$\tr\bigl((S_z+\mu I)^{-1} S_z\bigr) \lesssim \mu^{-(\rho_z-1)}.$
Together with Lemma~\ref{lemma:effective_dim_bound} ($\mathfrak{n}_z(\mu)\lesssim \mu^{-\rho_z}$), the expressions for $\tilde{\gamma}_1$, $\tilde{\gamma}_2$,
$\delta_z'$ and $\delta_z''$ and a regime with $\mu \le \lambda$, the
bound of $\Delta_{h_{\mu,\lambda}}$ in
Lemma~\ref{lemma:consistency} yields
\begin{align*}
  \Delta_{h_{\mu,\lambda}}
  &= \mathcal{O}\!\left(
        \frac{\tilde{\gamma}_1}{\sqrt{\lambda\mu}}b_{\mu,\lambda}
        + \frac{\delta_z''}{\lambda}v_{\mu,\lambda}
        + \frac{\tilde{\gamma}_2}{\sqrt{\lambda\mu}}
        + \frac{\delta_z'}{\sqrt{\lambda\mu}}b_{\mu,\lambda}
        + \Delta_U
     \right) \\
  &= \mathcal{O}\!\left(
        \frac{1}{\sqrt{n\lambda\mu}}b_{\mu,\lambda}
        + \frac{1}{\lambda}\sqrt{\frac{\mu\mathfrak{n}_z(\mu)}{n}}
          v_{\mu,\lambda}
        + \frac{1}{\sqrt{n\lambda\mu}}
        + \frac{1}{\sqrt{\lambda\mu}}\sqrt{\frac{\mu\mathfrak{n}_z(\mu)}{n}}
          b_{\mu,\lambda}
        + \Delta_U
     \right) \\
  &= \mathcal{O}\!\left(
        \frac{1}{\lambda}\sqrt{\frac{\mu\mathfrak{n}_z(\mu)}{n}}
          v_{\mu,\lambda}
        + \frac{1}{\sqrt{n\lambda\mu}}
        + \frac{1}{\sqrt{\lambda\mu}}\sqrt{\frac{\mu\mathfrak{n}_z(\mu)}{n}}
          b_{\mu,\lambda}
        + \Delta_U
     \right).
\end{align*}
In particular, since $v_{\mu,\lambda}\leq \bml $ the term involving $v_{\mu,\lambda}$ is dominated, and one obtains
\[
  \Delta_{h_{\mu,\lambda}}
  = \mathcal{O}\!\left(
      \frac{1}{\sqrt{n\lambda\mu}}
      + \frac{1}{\sqrt{\lambda\mu}}\sqrt{\frac{\mu\mathfrak{n}_z(\mu)}{n}}
        b_{\mu,\lambda}
      + \Delta_U
    \right).
\]
\end{remark}

\begin{theorem}
[Gaussian approximation]\label{theorem:gaussian-approx2}
Suppose  $n$ satisfies the rate Assumption \ref{ass:rate condition} i.e. $n\geq \max\{N_{\delta_z}, N_\delta\}$ and that Assumptions \ref{ass:lwx} and \ref{ass:lwz} hold. Choose a scheme for the regularization $(\lambda,\mu)$ so that $\|U_i\|_{\Hx}\lesssim\frac{1}{\sqrt{\lambda\mu}}$. In addition, suppose the assumptions for the bias upper bound hold (Proposition \ref{prop:general-bound}). Then, there exists a sequence $(Z_i)_{1 \leq i \leq n}$ of Gaussians in $\Hx$, with covariance $\Sigma$ such that with probability $1-\eta$, 
$$\|\sqrt{n}(\hat h - \hml) - \frac{1}{\sqrt{n}} \sum_{i = 1}^nZ_i\|_\Hx \lesssim  Q_\bullet(T, n, \lambda, \mu) \widetilde Ml(\eta/18) +Q_{\mathrm{res}},   $$
where
$$Q_\bullet = \inf_{m\geq 1}\left\{
 \frac{\sigma(S_x, m)}{\lambda}
+ \frac{m^2 \log(m^2)}{\sqrt{n\mu\lambda}}\right\} $$ 
and
$$Q_{\mathrm{res}} = \left(\left\{\frac{l(\eta/6)^2}{\sqrt{n}\lambda\mu^2} \vee\frac{l(\eta/6)^2 \nmu}{\sqrt{n}\lambda^{1/2}\mu^{1/2}} \right\} \|\hml - h_0\|_\Hx  \vee \frac{ \nmu l(\eta/6)^3}{n\lambda^{3/2}\mu}  \vee \frac{l(\eta/6)^2}{\sqrt{n}\lambda \mu^{3/2}} \right).$$
\end{theorem}
\begin{proof}
Recall that by Theorem \ref{theorem:bahadur}, with probability $1-5\eta$,
$$\|\sqrt{n}(\hat h - \hml) - \frac{1}{\sqrt{n}}\sum_{i = 1}^nU_i\|_\Hx \lesssim \sqrt{n}\left\{B(n, \mu, \lambda, \eta, h_0) +   V(n, \mu, \lambda, \eta)\right\}.$$
Now, we study the behavior $V$ and $B$ by plugging in the tractable behavior of the effective dimensions, while noting that $\mu\leq \lambda$ by the bias argument. Recall that
\begin{align*}
    V(n, \mu, \lambda, \eta) & =
\mathcal{O}\left(\frac{\nmu^{3/2}l(\eta)^4} {n^2\lambda^2\mu}\vee\frac{\nmu^{1/2}l(\eta)^3} {n^{3/2}\lambda^{3/2}\mu^{3/2}} \vee \frac{\nmu l(\eta)^3}{n^{3/2}\lambda^{3/2}\mu}  \vee \frac{l(\eta)^2}{n\lambda \mu^{3/2}} \vee \frac{\mathfrak{n}_z(\mu)^{1/2}l(\eta)^{2}}{n\sqrt{\lambda\mu}}  
\right)  \\
&= \mathcal{O}\left(\frac{l(\eta)^4}{n^2\lambda^2\mu^{1+(3/2)\rho_z}}\vee\frac{l(\eta)^3} {n^{3/2}\lambda^{3/2}\mu^{3/2 + (1/2)\rho_z}} \vee \frac{ l(\eta)^3}{n^{3/2}\lambda^{3/2}\mu^{1+\rho_z}}  \vee \frac{l(\eta)^2}{n\lambda \mu^{3/2}} \vee \frac{l(\eta)^2}{n\sqrt{\lambda \mu^{1+\rho_z}} }
\right),\\
&= \mathcal{O}\left(\frac{l(\eta)^4}{n^2\lambda^2\mu^{1+(3/2)\rho_z}}\vee\frac{l(\eta)^3} {n^{3/2}\lambda^{3/2}\mu^{3/2 + (1/2)\rho_z}} \vee \frac{ l(\eta)^3}{n^{3/2}\lambda^{3/2}\mu^{1+\rho_z}}  \vee \frac{l(\eta)^2}{n\lambda \mu^{3/2}}\right)
\end{align*}
and 
\begin{align*}
B(n, \mu, \lambda, \eta, h_0) &= 
\Ocal\left (\left\{ \frac{l(\eta)^2}{{n\lambda}\mu^2} \vee \frac{l(\eta)^2\mathfrak{n}_z(\mu)}{n\sqrt{\lambda\mu}} \vee \frac{l(\eta)^2\mathfrak{n}_z(\mu)^{1/2}}{n\lambda\mu} \right \}\bml \right) \\&= \Ocal\left (\left\{ \frac{l(\eta)^2}{{n\lambda}\mu^2} \vee \frac{l(\eta)^2}{n\lambda^{1/2}\mu^{1/2 + \rho_z}} \vee \frac{l(\eta)^2}{n\lambda\mu^{1+ (1/2)\rho_z}} \right \}\bml \right) \\
&= \Ocal\left (\left\{\frac{l(\eta)^2}{{n\lambda}\mu^2} \vee\frac{l(\eta)^2}{n\lambda^{1/2}\mu^{1/2 + \rho_z}} \right\}\bml \right).
\end{align*}
Now, we proceed to simplify using $\mu \leq \lambda$ and $\bml\lesssim\lambda^\alpha$:
\begin{align*}
B + V &= \Ocal\left( \left\{\frac{l(\eta)^2}{{n\lambda}\mu^2} \vee\frac{l(\eta)^2}{n\lambda^{1/2}\mu^{1/2 + \rho_z}} \right\} \bml  \vee \frac{l(\eta)^4}{n^2\lambda^2\mu^{1+(3/2)\rho_z}}\vee\frac{l(\eta)^3} {n^{3/2}\lambda^{3/2}\mu^{3/2 + (1/2)\rho_z}} \vee \frac{ l(\eta)^3}{n^{3/2}\lambda^{3/2}\mu^{1+\rho_z}}  \vee \frac{l(\eta)^2}{n\lambda \mu^{3/2}}\right) \\
 &= \Ocal\left( \left\{\frac{l(\eta)^2}{{n\lambda}\mu^2} \vee\frac{l(\eta)^2}{n\lambda^{1/2}\mu^{1/2 + \rho_z}} \right\}\bml  \vee \frac{ l(\eta)^3}{n^{3/2}\lambda^{3/2}\mu^{1+\rho_z}}  \vee \frac{l(\eta)^2}{n\lambda \mu^{3/2}}\right)
\end{align*}
Multiplying by $\sqrt{n}$, we arrive at the final form for polynomial decay:
$$
\|\sqrt{n}(\hat h - \hml) - \frac{1}{\sqrt{n}}\sum_{i = 1}^n U_i\|_\Hx \lesssim
\left(\left\{\frac{l(\eta)^2}{\sqrt{n}\lambda\mu^2} \vee\frac{l(\eta)^2}{\sqrt{n}\lambda^{1/2}\mu^{1/2 + \rho_z}} \right\} \bml  \vee \frac{ l(\eta)^3}{n\lambda^{3/2}\mu^{1+\rho_z}}  \vee \frac{l(\eta)^2}{\sqrt{n}\lambda \mu^{3/2}} \right).$$
In general form,
$$
\|\sqrt{n}(\hat h - \hml) - \frac{1}{\sqrt{n}}\sum_{i = 1}^nU_i\|_\Hx \lesssim
\left(\left\{\frac{l(\eta)^2}{\sqrt{n}\lambda\mu^2} \vee\frac{l(\eta)^2 \nmu}{\sqrt{n}\lambda^{1/2}\mu^{1/2}} \right\} \bml  \vee \frac{ \nmu l(\eta)^3}{n\lambda^{3/2}\mu}  \vee \frac{l(\eta)^2}{\sqrt{n}\lambda \mu^{3/2}} \right).$$
Recall the upper bound for the $\|U_i\|_\Hx \lesssim \frac{\widetilde{M}}{\sqrt{\lambda\mu}}$
where $\widetilde{M}$ encompasses all constants in the upper bound; see Corollary \ref{col:const bound a} for details. Now, using Theorem A.1 from \cite{singh2023kernel} together with Lemma \ref{lemma:local_width_bound},  with probability $1-\eta,$
\begin{align*}
\left\| \frac{1}{\sqrt{n}} \sum_{i=1}^n (U_i - Z_i) \right\|_\Hx
&\lesssim  
\inf_{m\geq 1}\left\{
\sqrt{\log(6/\eta)} \sigma(\Sigma,m)
+ \frac{a m^2 \log(m^2 / \eta)}{\sqrt{n}}\right\} \\
&\lesssim  
\inf_{m\geq 1}\left\{
\sqrt{\log(6/\eta)} \frac{C_{\mathrm{lw}} \bar \sigma}{\lambda} \sigma(S_x, m)
+ \frac{a m^2 \log(m^2 / \eta)}{\sqrt{n}}\right\}  \\
&\lesssim
\inf_{m\geq 1}\left\{
\sqrt{\log(6/\eta)} \frac{\widetilde M}{\lambda} \sigma(S_x, m)
+ \frac{\widetilde M m^2 \log(m^2 / \eta)}{\sqrt{n\lambda\mu}}\right\}  .
\end{align*}
Using a union bound, the Gaussian approximation holds with $1-6\eta$. Rescaling $\eta/6$
 gives the result.\end{proof}

\section{Bootstrap coupling}
\begin{assumption}[Combined rate condition]\label{ass:combined_rate_condition}
Let $\lambda, \mu>0$.
Let $N_\delta$ and $N_{\delta_z}$ denote the sample size thresholds
as defined in Lemma \ref{lemma:sz_hat} and Assumption \ref{ass:rate condition}.
Assume that the sample size $n$ satisfies
\[
  n \ge \max\Bigl\{
    3N_\delta,
    N_{\delta_z},
    \mu^{-2}\lambda^{-1}\tilde{\mathfrak{m}}(\lambda,\mu)^{-1}
  \Bigr\}.
\]
\end{assumption}
\begin{remark}[Dominant terms in $\delta$]
Let $\delta$ be as in Lemma~\ref{lemma:delta}.
Under the bounds for $\delta_z$, $\tilde{\gamma}_1$ and $\tilde{\gamma}_z$
in that lemma, it holds that
\[
  \delta
  = \mathcal{O}\!\left(
      \frac{\delta_z \tilde{\gamma}_1 + \tilde{\gamma}_1^2}{\lambda\mu}
      + \frac{\tilde{\gamma}_z}{\sqrt{\lambda}\mu}
      + \frac{\tilde{\gamma}_1}{\sqrt{\lambda\mu}}
      + \frac{\tilde{\gamma}_1}{\lambda\sqrt{\mu}}
    \right).
\]
If Assumption~\ref{ass:combined_rate_condition} holds and $\mu \le \lambda$,
then
\[
  \delta
  = \mathcal{O}\!\left(
      \frac{\sqrt{\mathfrak{n}_z(\mu)}}{n\mu\lambda}
      \vee
      \frac{1}{n^{3/2}\lambda\mu^{2}}
      \vee
      \frac{1}{\sqrt{n\lambda}\mu}
    \right)
  = \mathcal{O}\!\left(
      \frac{1}{\sqrt{n\lambda}\mu}
    \right).
\]
The same rate holds for $\delta_B$, since $\delta_B = 3\delta$.
\end{remark}

\begin{lemma}[Feasible bootstrap rate]\label{lemma:bootstrap_rate}
 Suppose assumptions \ref{ass:combined_rate_condition}, \ref{ass:lwx}, and \ref{ass:lwz} hold. In addition, suppose that the assumptions for the bias upper bound hold (Proposition \ref{prop:general-bound}). Choose a scheme for the regularization $(\lambda,\mu)$ so that $\|U_i\|_{\Hx}\lesssim\frac{1}{\sqrt{\lambda\mu}}$. Then,  with probability $1-14\eta$,
\begin{align*}
\|\mathfrak{B}-Z_{\mathfrak{B}}\|_\Hx
&\lesssim C_\eta\Bigg[ \left\{
\frac{l(\eta)}{\mu^2\sqrt{n} \lambda}  + \frac{\nmu l(\eta)^2 }{n\lambda\mu^{3/2 }}\right\} \bml
+\frac{l(\eta)}{\lambda\mu^{3/2}\sqrt{n}}
 + \frac{ \nmu l(\eta)^3}{n^{3/2}\lambda^{2}\mu^{2}}
\Bigg],
\end{align*}
\end{lemma}
\begin{proof}
Recall the elements of the feasible bootstrap
\begin{itemize}
\item[Term A]
\begin{align*}
    A &= \Ocal(2\delta_B\sqrt{14}\Big(\tilde\gamma_2\sqrt{\delta_\mu^{\prime\prime}}\Big)
   + \sqrt{60(1 + 4\delta_B^2)}\Big(
      \tilde\gamma_1(\bml\sqrt{\delta_\mu^{\prime\prime\prime} }+\sqrt{\gamma^\prime})
    + \delta_z^\prime\vml\sqrt{\delta_\mu^{\prime\prime}}
    + \tilde\gamma_1\bml\sqrt{\delta_\mu^{\prime\prime}}
    + \tilde\gamma_2\sqrt{\delta_\mu^{\prime\prime}}
   \Big)\\
   &= \Ocal\left( \delta_B\Big(\tilde \gamma_2\sqrt{\delta_\mu^{\prime\prime}}\Big) + 
\tilde\gamma_1
+\delta_z^\prime\vml\sqrt{\delta_\mu^{\prime\prime}}
+\tilde\gamma_2\sqrt{\delta_\mu^{\prime\prime}}
\right) \\
&= \Ocal\left( \delta_B\sqrt{\delta_\mu^{\prime \prime}}  (n^{-1/2} )\vee  
\sqrt{\frac{\mu \nmu}{n}}\vml \sqrt{\delta_\mu^{\prime\prime} } \vee \frac{1}{\sqrt{n}}\sqrt{\delta_\mu^{\prime\prime} }  \right) \\
&= \Ocal\left( \sqrt{\delta_\mu^{\prime \prime}} \left\{ \delta_B n^{-1/2} \vee  
\sqrt{\frac{\mu \nmu}{n}} \vml\vee \frac{1}{\sqrt{n}}\right\}\right) \\
&= \Ocal(\sqrt{\delta_\mu^{\prime \prime}} n^{-1/2}),
\end{align*}
\item[Term B]\begin{align*}
        B &= \Ocal\left(2\delta_B\sqrt{14}\Big(2\kappa{\kappa_x}\Delta_h\Big)
   + \sqrt{60(1 + 4\delta_B^2)}\Big(
      \kappa_x\kappa\Delta_{\hml}
    + \mu^{-1/2}\tilde\gamma_1v_{\mu,\lambda}\sqrt{\delta_\mu^{\prime\prime\prime\prime}}
   \Big)\right)  \\
   &= \Ocal(\delta_B ( \bml + \Delta_{\hml}) +\Delta_{\hml} )\\
   &= \Ocal(\delta_B\bml + \Delta_{\hml}),
\end{align*}
\item[Term C] \begin{align*}
C &=\Ocal\Big( 2\delta_B\sqrt{14}\Big(
      \kappa\Delta_h
    + \tfrac12\sqrt{\gamma^\prime}
    + \tfrac12\tilde\gamma_2\sqrt{\delta_\mu^\prime}
    + 2\Delta_h\sqrt{\lambda\mu}
   \Big) \\
&\quad + \sqrt{60(1 + 4\delta_B^2)}\Big(
      \delta_z(\kappa\bml+\sqrt{\gamma^\prime})
    + \tfrac12\kappa\Delta_{\hml}
    + \delta_z\vml\sqrt{\delta_\mu^{\prime\prime\prime\prime}}
    + \tfrac12\delta_z^\prime\vml\sqrt{\delta_\mu^\prime}
    + \tfrac12\tilde\gamma_1\bml\sqrt{\delta_\mu^\prime}
    + \tfrac12\tilde\gamma_2\sqrt{\delta_\mu^\prime}
   \Big)\Big)\\
   &= \Ocal \Bigg((\delta_B\Big(
\Delta_h
+\tilde\gamma_2\sqrt{\delta_\mu^\prime}
\Big) +\Big(
\delta_z\bml
+\Delta_{\hml}
+\delta_z\vml\sqrt{\delta_\mu^{\prime\prime\prime\prime}}
+\delta_z^\prime\vml\sqrt{\delta_\mu^\prime}
+\tilde\gamma_2\sqrt{\delta_\mu^\prime}
\Big)\Bigg)\\
&=\Ocal \Bigg(\delta_B\Big(
\bml
+\tilde\gamma_2\sqrt{\delta_\mu^\prime}
\Big) +\Big(
\delta_z\bml
+\Delta_{\hml}
+\delta_z\vml\sqrt{\delta_\mu^{\prime}}
+\tilde\gamma_2\sqrt{\delta_\mu^\prime}
\Big)\Bigg) \\
&= \Ocal(\sqrt{\delta_\mu^\prime}(\delta_B \tilde \gamma_2 + \delta_z \vml + \tilde\gamma_2) + \delta_z\bml + \Delta_{\hml}) \\
&= \Ocal(\sqrt{\delta_\mu^\prime}(\delta_z \vml + \tilde\gamma_2) + \delta_z\bml + \Delta_{\hml}) ,
\end{align*}
\item[Term D]\begin{align*} 
    D &= 2\delta_B\sqrt{14}\Big(\tfrac12\kappa_z^2{\kappa_x}\Delta_h\Big)
   + \sqrt{60(1 + 4\delta_B^2)}\Big(\tfrac12\kappa_x\kappa_z^2\Delta_{\hml}\Big) \\
    &= \Ocal(\delta_B\bml + \Delta_{\hml}).
   \end{align*}
\end{itemize}
Let $C_\eta = \left\{1 + \sqrt{2\log(1/\eta)}\right\}$, we can now study the error bound of the bootstrap in detail: 
\begin{align*}
\|\mathfrak{B} - Z_{\mathfrak{B}}\|_\Hx \leq C_\eta \Bigg[
\frac{1}{\lambda} A
+\frac{1}{\lambda\sqrt{\mu}} B
+\frac{1}{\sqrt{\lambda\mu}} C
+\frac{1}{\sqrt{\lambda}\mu} D
\Bigg].
\end{align*}
Since $\mu\leq \lambda$ and $c_{det} = \frac{1}{\lambda\sqrt{\mu}}  +\frac{1}{\sqrt{\lambda}\mu}\lesssim \frac{1}{\sqrt{\lambda}\mu}.$
Thus,
\begin{align*}
\|\mathfrak{B} - Z_{\mathfrak{B}}\|_\Hx
&\lesssim C_\eta\Big[
  \frac{1}{\lambda}\sqrt{\delta_\mu^{\prime\prime}}n^{-1/2}
+ c_{det}\delta_B\bml
+ \big(c_{det} +\frac{1}{\sqrt{\lambda\mu}} \big)\Delta_{\hml}
+ \frac{1}{\sqrt{\lambda\mu}}\big(\sqrt{\delta_\mu^\prime}(\delta_z\vml+\tilde\gamma_2)+\delta_z\bml\big)
\Big]\\
&\lesssim C_\eta\Big[
  \frac{1}{\lambda}\sqrt{\delta_\mu^{\prime\prime}}n^{-1/2}
+\frac{1}{\sqrt{\lambda}\mu}\delta_B\bml
+ \frac{1}{\sqrt{\lambda}\mu}\Delta_{\hml}
+ \frac{1}{\sqrt{\lambda\mu}}\big(\sqrt{\delta_\mu^\prime}(\delta_z\vml+n^{-1/2})+\delta_z\bml\big)
\Big]\\
&\lesssim C_\eta\Big[
\frac{1}{\sqrt{\lambda}\mu}\delta_B\bml
+ \frac{1}{\sqrt{\lambda}\mu}(  \frac{1}{\sqrt{n\lambda\mu}}+ \frac{1}{\sqrt{\lambda\mu}}\sqrt{\frac{\mu\nmu}{n}}\bml + \Delta_U))
+ \frac{1}{\sqrt{\lambda\mu}}\big(\sqrt{\delta_\mu^\prime}(\delta_z\vml+n^{-1/2})+\delta_z\bml\big)
\Big]\\
&\lesssim C_\eta\Big[
\frac{1}{\sqrt{\lambda}\mu}\delta_B\bml
+ \frac{1}{\sqrt{\lambda}\mu}(  \frac{1}{\sqrt{n\lambda\mu}}+ \frac{1}{\sqrt{\lambda\mu}}\sqrt{\frac{\mu\nmu}{n}}\bml + \Delta_U))
\Big]\\
&\lesssim C_\eta\Big[
\frac{1}{\sqrt{\lambda}\mu}\delta_B\bml
+ \frac{1}{\sqrt{\lambda}\mu}(  \frac{1}{\sqrt{n\lambda\mu}}+ \frac{1}{\sqrt{\lambda\mu}}\sqrt{\frac{\mu\nmu}{n}}\bml + \Delta_U))
\Big]
\\
&\lesssim C_\eta\Big[
\frac{1}{\sqrt{\lambda}\mu}\frac{1}{\sqrt{n\lambda}\mu}
\lambda^\alpha
+ \frac{1}{\sqrt{\lambda}\mu}(  \frac{1}{\sqrt{n\lambda\mu}}+ \frac{1}{\sqrt{n \lambda \mu^{\rho_z}}}\lambda^\alpha+ \Delta_U))
\Big]
\\
&\lesssim C_\eta\Big[
\frac{1}{\sqrt{\lambda}\mu}\frac{1}{\sqrt{n\lambda}\mu}
\lambda^\alpha
+ \frac{1}{\sqrt{\lambda}\mu}(  \frac{1}{\sqrt{n\lambda\mu}}+ \Delta_U))
\Big].
\end{align*}
Recall from Theorem~\ref{theorem:bahadur} that the unscaled Bahadur remainder satisfies
$$\Delta_U \lesssim \left(\left\{\frac{l(\eta)^2}{{n\lambda}\mu^2} \vee\frac{l(\eta)^2}{n\lambda^{1/2}\mu^{1/2 + \rho_z}} \right\} \lambda^\alpha  \vee \frac{ l(\eta)^3}{n^{3/2}\lambda^{3/2}\mu^{1+\rho_z}}  \vee \frac{l(\eta)^2}{n\lambda \mu^{3/2}}\right).$$
Plugging this in gives, for polynomial decay and using Assumption \ref{ass:combined_rate_condition}
\begin{align*}
\|\mathfrak{B}-Z_{\mathfrak{B}}\|_\Hx
&\lesssim
\frac{\lambda^{\alpha-1}}{\mu^2\sqrt{n}}
+\frac{1}{\lambda\mu^{3/2}\sqrt{n}}
+ \frac{1}{{n\lambda^{3/2}}\mu^3} \lambda^\alpha  + \frac{1}{n\lambda\mu^{3/2 + \rho_z}}\lambda^\alpha+ \frac{ 1}{n^{3/2}\lambda^{2}\mu^{2+\rho_z}} + \frac{1}{n\lambda^{3/2} \mu^{5/2}} \\
&\lesssim \frac{\lambda^{\alpha-1}}{\mu^2\sqrt{n}}
+\frac{1}{\lambda\mu^{3/2}\sqrt{n}}
 + \frac{1}{n\lambda\mu^{3/2 + \rho_z}}\lambda^\alpha+ \frac{ 1}{n^{3/2}\lambda^{2}\mu^{2+\rho_z}}
\end{align*}
In general form,
\begin{align*}
\|\mathfrak{B}-Z_{\mathfrak{B}}\|_\Hx
&\lesssim C_\eta\Bigg[ \left\{
\frac{l(\eta)}{\mu^2\sqrt{n} \lambda}  + \frac{\nmu l(\eta)^2 }{n\lambda\mu^{3/2 }}\right\} \bml
+\frac{l(\eta)}{\lambda\mu^{3/2}\sqrt{n}}
 + \frac{ \nmu l(\eta)^3}{n^{3/2}\lambda^{2}\mu^{2}}
\Bigg],
\end{align*}
which holds with probability $1-14\eta.$
\end{proof}
\begin{theorem}[Bootstrap approximation]\label{theorem:bootstrap-approx}  Suppose assumptions \ref{ass:combined_rate_condition},  \ref{ass:lwx}, and \ref{ass:lwz} hold. In addition, suppose that the assumptions for the bias upper bound hold (Proposition \ref{prop:general-bound}).  Choose a scheme for the regularization $(\lambda,\mu)$ according to Corollary \ref{col:const bound a} so that $\|U_i\|_{\Hx}\lesssim\frac{1}{\sqrt{\lambda\mu}}$. Then, there exists
a random element $Z \in \Hx$ whose conditional distribution given U is Gaussian with covariance $\Sigma$, such that with probability $1-\eta$, 
\begin{align*}
    \mathbb{P}\left[ \|\mathfrak{B} - Z \|_{\Hx} \lesssim \widetilde M l(\eta/3)^{3/2} R_\bullet(n,\lambda, \mu) + R_{\mathrm{res}} |U\right]\geq 1-\eta,
\end{align*}
where  $\widetilde M$ is an absolute constant, and the key quantities are
\begin{align*}
    R_\bullet&= \inf_{m\ge 1}\left[ m^{1/4}\left\{ \frac{ \tilde{\mathfrak{m}}(\lambda,\mu)}{\mu\lambda n} + \frac{1}{n^{2}\mu^2 \lambda^2} \right\}^{1/4} + \frac{ \sigma(S_x, m)}{\lambda} \right],\\ R_{\mathrm{res}}&=\sqrt{2l\left(\eta/{7}\right)}\Bigg[ \left\{
\frac{l(\eta/{14})}{\mu^2\sqrt{n} \lambda}  + \frac{\nmu l(\eta/{14})^2 }{n\lambda\mu^{3/2 }}\right\} \|\hml - h_0\|_{\Hx} \
+\frac{l(\eta/{14})}{\lambda\mu^{3/2}\sqrt{n}}
 + \frac{ \nmu l(\eta/{14})^3}{n^{3/2}\lambda^{2}\mu^{2}}
\Bigg].
\end{align*}
\end{theorem}

\begin{proof}
    Let $l(\eta) := \log(2/\eta)$ and $\bml := \|\hml - h_0\|_{\Hx}$. Recall from Lemma \ref{lemma:bootstrap_rate}, with probability $1-\eta $, 
\begin{align*}
& \|\mathfrak{B}-Z_{\mathfrak{B}}\|_{\Hx}\lesssim \sqrt{2l\left(\eta/{7}\right)}\Bigg[ \left\{
\frac{l(\eta/{14})}{\mu^2\sqrt{n} \lambda}  + \frac{\nmu l(\eta/{14})^2 }{n\lambda\mu^{3/2 }}\right\} \bml
+\frac{l(\eta/{14})}{\lambda\mu^{3/2}\sqrt{n}}
 + \frac{ \nmu l(\eta/{14})^3}{n^{3/2}\lambda^{2}\mu^{2}} 
\Bigg].
\end{align*}
This defines $R_{\mathrm{res}}.$
Now, apply Corollary B.1 from \cite{singh2023kernel} taking $W = Z_{\mathfrak{B}}$ and $W^\prime = {\mathfrak{B}}$. Then, there must exist $Z$ with the desired conditional distribution, such that with probability $1-\eta$, conditional upon $\sigma(U)$, 
\begin{align*}
    \|Z- \mathfrak{B}\|_{\Hx} &\lesssim C^\prime \log(6/\eta)^{3/2} \inf_{m\ge 1}\left[ m^{1/4}\left\{ \frac{a^{2}\sigma^{2}(\Sigma,0)}{n} + \frac{a^{4}}{n^{2}} \right\}^{1/4} + \sigma(\Sigma,m) \right]
+R_{\mathrm{res}} .
\end{align*}
Similar for the Gaussian coupling we use that $a\lesssim \frac{\widetilde{M}} 
{ \sqrt{\lambda\mu}}$, due to Corollary \ref{col:const bound a} and  Lemma \ref{lemma:local_width_bound}: 
\begin{align*}
    \|Z- \mathfrak{B}\|_{\Hx} &\lesssim C^\prime l(\eta/3)^{3/2} \inf_{m\ge 1}\left[ m^{1/4}\left\{ \frac{a^{2}\sigma^{2}(\Sigma,0)}{n} + \frac{a^{4}}{n^{2}} \right\}^{1/4} + \sigma(\Sigma,m) \right]
+R_{\mathrm{res}} \\
&\lesssim \widetilde M l(\eta/3)^{3/2} \inf_{m\ge 1}\left[ m^{1/4}\left\{ \frac{ \tilde{\mathfrak{m}}(\lambda,\mu)}{\mu\lambda n} + \frac{1}{n^{2}\mu^2 \lambda^2} \right\}^{1/4} + \frac{ \sigma(S_x, m)}{\lambda} \right]
+R_{\mathrm{res}}. \qquad \qedhere
\end{align*}
\end{proof}

\section{Uniform confidence band}\label{sec:uniform_inference}
\subsection{High level summary}
\begin{enumerate}
    \item $Q(n, \lambda, \mu, \eta) =  Q_\bullet( n, \lambda, \mu, \eta) \widetilde Ml(\eta/18) +Q_{\mathrm{res}} $;
    \item  $R(n, \lambda, \mu, \eta) = \widetilde M l(\eta/3)^{\frac{3}{2}} R_\bullet(n,\lambda, \mu, \eta) + R_{\mathrm{res}} $;
    \item $L(\lambda,\mu,\eta) =  \sqrt{\frac{1}{2 }\underline{\sigma}^2\mathfrak{\tilde m}(\lambda, \mu)} - \left \{ 2 + \sqrt{2 \log(1/\eta)}\right \}\sqrt{\frac{ 2\bar \sigma^2}{\lambda} }$;
  \item $B(\lambda,\mu) = \sqrt{n}\frac{C_\alpha \lambda^{\alpha} \|T^{-\alpha} h_0\|_{\Hx}}{1 - C_\beta r \mu^{\beta}/\lambda^{1/2}}$.  
\end{enumerate}
We now explicitly provide upper bounds for $(Q,R, L, B)$ under polynomial decay.

\noindent\textbf{B and L:}
By the bias upper bound (Proposition~\ref{prop:Bias upper bound})
$$B(\lambda,\mu) = \sqrt{n}\frac{C_\alpha \lambda^{\alpha} \|T^{-\alpha} h_0\|_{\Hx}}{1 - C_\beta r \mu^{\beta}/\lambda^{1/2}}.$$ 
Lemma \ref{lemma:valbo} implies that
$$\norm{Z}_\Hx \gtrsim L(\lambda,\mu,\eta) =  \sqrt{\frac{1}{2 }\underline{\sigma}^2\mathfrak{\tilde m}(\lambda, \mu)} - \left \{ 2 + \sqrt{2 \log(1/\eta)}\right \}\sqrt{\frac{ 2\bar \sigma^2}{\lambda} }.$$
\textbf{Q:} Note that $Q = Q_\bullet  + Q_{\mathrm{res}}$. The latter is already derived in Theorem \ref{theorem:gaussian-approx2}. For $Q_\bullet$, by Assumption~\ref{ass:lwx},
$\sigma(S_x, m) \lesssim_{\rho_x,\omega_x} m^{1/2 - 1/(2\rho_x - 2)}.$
Recall that
$$Q_\bullet = \inf_{m\geq 1}\left\{
 \frac{\sigma(S_x, m)}{\lambda}
+ \frac{m^2 \log(m^2)}{\sqrt{n\mu\lambda}}\right\}. $$ 
Balancing the two terms (ignoring logarithmic factors) leads to the equation:
\begin{align*}
    \frac{m^{\frac12-\frac{1}{2\rho_x-2}}}{\lambda} &= \frac{m^2}{\sqrt{n\lambda\mu}} \\
    \implies \quad m^{-\frac32-\frac{1}{2\rho_x-2}} &= \left(\frac{\lambda}{n\mu}\right)^{1/2} \\
    \implies \quad m &= \left(\frac{n\mu}{\lambda}\right)^{\frac{\rho_x-1}{3\rho_x-2}}.
\end{align*}
Substituting this optimal $m$ back into the first term yields the upper bound:
\[
    Q_\bullet \lesssim \frac{\sigma_m(S_x)}{\lambda} \lesssim_{\rho_x,\omega_x} \frac{1}{\lambda} \left(\frac{n\mu}{\lambda}\right)^{\frac{\rho_x-2}{2(3\rho_x-2)}}.
\]
\noindent\textbf{Bound for R:}
Similarly, $R = R_\bullet + R_{\mathrm{res}}$, with $R_{\mathrm{res}}$ derived in Theorem \ref{theorem:bootstrap-approx}. For $R_\bullet$, recall:
\[
    R_\bullet(n, \lambda, \mu) = \inf_{m\ge 1}\left[ m^{1/4}\left\{ \frac{ \tilde{\mathfrak{m}}(\lambda,\mu)}{\mu\lambda n} + \frac{1}{n^{2}\mu^2 \lambda^2} \right\}^{1/4} + \frac{ \sigma_m(S_x)}{\lambda} \right].
\]
Using Assumption \ref{ass:combined_rate_condition} and $\mu \le \lambda$, the first term in the brace dominates the higher-order $n^{-2}$ term. Using the bound $\tilde{\mathfrak{m}}(\lambda,\mu) \lesssim \lambda^{-\rho_x}$, we balance the coupling error with the tail sum:
\begin{align*}
    m^{1/4}\left(\frac{\lambda^{-\rho_x}}{n\lambda\mu}\right)^{1/4} &= \frac{m^{\frac12-\frac{1}{2\rho_x-2}}}{\lambda} \\
    \implies \quad m^{-\frac14+\frac{1}{2\rho_x-2}} &= (n\mu)^{1/4}\lambda^{(\rho_x-3)/4} \\
    \implies \quad m &= (n\mu)^{-\frac{\rho_x-1}{\rho_x-3}} \lambda^{-(\rho_x-1)}.
\end{align*}
Substituting this into the tail sum bound gives:
\[
    R_\bullet \lesssim \frac{\sigma_m(S_x)}{\lambda} \lesssim_{\rho_x,\omega_x} \lambda^{-\rho_x/2} (n\mu)^{-\frac{\rho_x-2}{2(\rho_x-3)}}.
\]

\subsection{Restrictions for valid inference}
We derive restrictions on possible data-generating processes that allow us to use the guarantees of \cite{singh2023kernel}.
\subsubsection{$B \ll L$} We require the bias $B$ to be negligible compared to the variance proxy $L \asymp \lambda^{-\rho_x/2}$.
By Proposition~\ref{prop:Bias upper bound}, the denominator $1 - C_\beta r \mu^{\beta}/\lambda^{1/2}$ is bounded away from zero under the maintained regularization condition $\mu^\beta r < \lambda^{1/2}$ (Assumption~\ref{assn:restricted-link} with $\beta \geq 1/2$). In the regime $\lambda = \mu^\iota$ with $\iota   \leq 1$, this requires $r\mu^{\beta - \iota/2} < 1$, which holds for $\mu$ sufficiently small since $\beta \geq 1/2 \geq \iota/2$. Similarly, for $\mu = \lambda/C$, it holds provided $C > (C_\beta r)^{2}$. In either case, $1 - C_\beta r \mu^{\beta}/\lambda^{1/2} \geq c > 0$ for a constant $c$ independent of $(\lambda,\mu)$, so the bias satisfies  
\[
B = n^{1/2}\frac{C_\alpha \lambda^{\alpha} \|T^{-\alpha} h_0\|_{\Hx}}{1 - C_\beta r \mu^{\beta}/\lambda^{1/2}} \lesssim n^{1/2}\lambda^{\alpha}. 
\]
The condition $B \ll L$ then requires  

\[
n^{1/2}\lambda^{\alpha} \ll \lambda^{-\rho_x/2}  \quad\Longleftrightarrow\quad   n \ll \lambda^{-(\rho_x+2\alpha)}.  
\]
\subsubsection{$Q  + R \ll L$}
We determine the sample size $n$ required to ensure that the strong approximation errors $Q_\bullet$ and $R_\bullet$ are dominated by $L \asymp \lambda^{-\rho_x/2}$.

\paragraph{Condition for $Q_\bullet$:}
Using the optimized bound for $Q_\bullet$ derived previously, the condition $Q_\bullet \ll L$ requires:
\begin{align*}
\frac{1}{\lambda} \left(\frac{n\mu}{\lambda}\right)^{\frac{\rho_x-2}{2(3\rho_x-2)}} &\ll \lambda^{-\rho_x/2} \\
\left(\frac{n\mu}{\lambda}\right)^{\frac{\rho_x-2}{2(3\rho_x-2)}} &\ll \lambda^{1-\rho_x/2} \\
\frac{n\mu}{\lambda} &\gg \left( \lambda^{\frac{2-\rho_x}{2}} \right)^{\frac{2(3\rho_x-2)}{\rho_x-2}} \\
n &\gg \mu^{-1}\lambda^{-(3\rho_x-3)}.
\end{align*}

\paragraph{Condition for $R_\bullet$:}
Using the optimized bound for $R_\bullet$ derived previously, the condition $R_\bullet \ll L$ requires that
\[
\lambda^{-\rho_x/2} (n\mu)^{-\frac{\rho_x-2}{2(\rho_x-3)}} \ll \lambda^{-\rho_x/2}.
\]
Canceling the $\lambda^{-\rho_x/2}$ term, this simplifies to
\[
(n\mu)^{-\frac{\rho_x-2}{2(\rho_x-3)}} \ll 1 \iff n \gg \mu^{-1}.
\]
\paragraph{Conditions for $Q_{\mathrm{res}}$:} Using the definition of $Q_{\mathrm{res}}$ from Theorem \ref{theorem:gaussian-approx2} and plugging in the bias and effective dimension upper bounds, we derive the following sufficient conditions for $Q_{\mathrm{res}} \ll L$:
\begin{align*}
\text{(1)}\quad &\frac{ \lambda^{\alpha-1}}{\sqrt{n}\mu^2} \le \lambda^{-\rho_x/2} \quad\Longleftrightarrow\quad n \ge  \mu^{-4}\lambda^{2\alpha-2+\rho_x} \\[0.5em] \text{(2)}\quad & \frac{\lambda^{\alpha-\frac{1}{2}}}{\sqrt{n}\mu^{\frac{1}{2} + \rho_z}} \le \lambda^{-\rho_x/2}  \quad\Longleftrightarrow\quad n \ge \mu^{-(1+2\rho_z)}\lambda^{2\alpha - 1 + \rho_x}  \\[0.5em] \text{(3)}\quad &\frac{1}{n\lambda^{\frac{3}{2}}\mu^{1+\rho_z}} \le \lambda^{-\rho_x/2} \quad\Longleftrightarrow\quad n \ge \mu^{-(1+\rho_z)}\lambda^{-\frac{3}{2}+\frac{\rho_x}{2}} \\[0.5em] \text{(4)}\quad &\frac{1}{\sqrt{n}\lambda \mu^{\frac{3}{2}}} \le \lambda^{-\rho_x/2} \quad\Longleftrightarrow\quad n \ge  \mu^{-3}\lambda^{-2+\rho_x} \end{align*}

\paragraph{Conditions for $R_{\mathrm{res}}$:} Similarly, using the definition of $R_{\mathrm{res}}$ from Theorem \ref{theorem:bootstrap-approx}: 
\begin{align*} \text{(1)}\quad &\frac{\lambda^{\alpha-1}}{\mu^2\sqrt{n}} \le \lambda^{-\rho_x/2} \quad\Longleftrightarrow\quad n \ge \mu^{-4}\lambda^{2\alpha-2+\rho_x} \\[0.35em]
\text{(2)}\quad & \frac{\lambda^{\alpha-1}}{n\mu^{\frac{3}{2} + \rho_z}} \le \lambda^{-\rho_x/2} \quad\Longleftrightarrow\quad n \ge \mu^{-(\frac{3}{2} + \rho_z)}\lambda^{\alpha - 1 + \frac{\rho_x}{2}} \\[0.35em]
\text{(3)}\quad &\frac{1}{\lambda\mu^{\frac{3}{2}}\sqrt{n}} \le \lambda^{-\rho_x/2} \quad\Longleftrightarrow\quad n \ge \mu^{-3}\lambda^{ -2+\rho_x} \\[0.35em] \text{(4)}\quad &\frac{1}{n^{\frac{3}{2}}\lambda^{2}\mu^{2+\rho_z}} \le \lambda^{-\rho_x/2} \quad\Longleftrightarrow\quad n \ge \mu^{-\frac{4}{3}-\frac{2}{3}\rho_z}\lambda^{ -\frac{4}{3}+\frac{\rho_x}{3}}  \end{align*}
Collecting these constraints, we summarize the resulting sufficient conditions on the sample size $n$ and regularization parameters $(\lambda, \mu)$ in Table~\ref{table:restrictions}. 
\begin{table}[h!]
\centering
\caption{Final restrictions on $\lambda,\mu$.}\label{table:restrictions}
\begin{tabular}{ll}
\hline
Component & Sufficient restriction on $(\lambda,\mu)$\\\hline
$B \ll L $ & $n \ll \lambda^{-(\rho_x+2\alpha)}$ \\[0.15em]
$Q_\bullet \ll L $ & $n \gg \mu^{-1}\lambda^{-(3\rho_x-3)}$ \\[0.15em]
$R_\bullet \ll L $ & $n \gg \mu^{-1}$ \\[0.15em]
$Q_{\mathrm{res}}^{(1)} \ll L $ & $n \gg  \mu^{-4}\lambda^{2\alpha-2+\rho_x} $ \\[0.15em]
$Q_{\mathrm{res}}^{(2)} \ll L $ &
$n \gg \mu^{-(1+2\rho_z)}\lambda^{2\alpha - 1 + \rho_x}$  \\[0.15em]
$Q_{\mathrm{res}}^{(3)} \ll L $ & $n \gg \mu^{-(1+\rho_z)}\lambda^{-\frac{3}{2}+\frac{\rho_x}{2}}$ \\[0.15em]
$Q_{\mathrm{res}}^{(4)} \ll L $ & $n \gg  \mu^{-3}\lambda^{-2+\rho_x}$ \\[0.35em]
$R_{\mathrm{res}}^{(1)} \ll L $ & $n \gg \mu^{-4}\lambda^{2\alpha-2+\rho_x}$ \\[0.15em]
$R_{\mathrm{res}}^{(2) }\ll L $
& $ n \gg \mu^{-(\frac{3}{2} + \rho_z)}\lambda^{\alpha - 1 + \frac{\rho_x}{2}}$\\[0.15em]
$R_{\mathrm{res}}^{(3) }\ll L $ & $n \gg \mu^{-3}\lambda^{-2+\rho_x}$ \\[0.15em]
$R_{\mathrm{res}}^{(4)}\ll L $ & $n \gg \mu^{-\frac{4}{3}-\frac{2}{3}\rho_z}\lambda^{ -\frac{4}{3}+\frac{\rho_x}{3}} $  \\\hline
\end{tabular}
\end{table}

\subsubsection{Parameter restrictions from combined bounds}
We now derive the conditions on the parameters $\alpha, \rho_x, \rho_z,$ and $\iota$ required to ensure that the set of data-generating processes defined by the approximation lower bound ($Q + R \ll L$) and the bias upper bound ($B \ll L$) is non-empty. Recall that the bias constraint imposes the upper bound $n \ll \lambda^{-(\rho_x + 2\alpha)}$. We substitute $\lambda=\mu^{\iota}$ and analyze the regime $\mu \to 0$.
 For $Q_\bullet$, the condition $\mu^{-1}\lambda^{-(3\rho_x-3)} \ll \lambda^{-(\rho_x+2\alpha)}$ becomes
\[
\mu^{-\bigl(1+\iota(2\rho_x-3-2\alpha)\bigr)} \ll 1 \iff 1+\iota(2\rho_x-3-2\alpha) < 0 \iff 2\rho_x < 3+2\alpha-\frac{1}{\iota}.
\]
For $R_\bullet$, the condition $\mu^{-1} \ll \lambda^{-(\rho_x+2\alpha)}$ implies $\mu^{-1} \ll \mu^{-\iota(\rho_x+2\alpha)}$. This holds for $\mu \to 0$ if and only if
\[
1 < \iota(\rho_x+2\alpha) \iff \rho_x+2\alpha > \frac{1}{\iota}.
\]
Next, we analyze the restrictions imposed by $Q_{\mathrm{res}}$. We require the lower bound exponent to be strictly larger than the upper bound exponent (in terms of $\mu^{-k}$ decay):
\begin{align*} \text{(1)} \quad & \mu^{-4}\lambda^{2\alpha-2+\rho_x} \ll \lambda^{-(\rho_x+2\alpha)} \iff \iota(4\alpha+2\rho_x-2) > 4 \\[0.6em] \text{(2)}
\quad & \mu^{-(1+2\rho_z)}\lambda^{2\alpha - 1 + \rho_x}  \ll \lambda^{-(\rho_x+2\alpha)} \iff  \iota(4\alpha + 2\rho_x - 1) > 1+2\rho_z
\\[0.6em] \text{(3)}
\quad & \mu^{-(1+\rho_z)}\lambda^{-\frac{3}{2}+\frac{\rho_x}{2}} \ll \lambda^{-(\rho_x+2\alpha)} \iff \iota(4\alpha+3\rho_x-3) > 2(1+\rho_z) \\[0.6em] \text{(4)} \quad & \mu^{-3}\lambda^{-2+\rho_x} \ll \lambda^{-(\rho_x+2\alpha)} \iff \iota(2\alpha+2\rho_x-2) > 3 \end{align*}
Similarly, the conditions from $R_{\mathrm{res}}$ yield
\begin{align*} \text{(1)} \quad & \mu^{-4}\lambda^{2\alpha-2+\rho_x} \ll \lambda^{-(\rho_x+2\alpha)} \iff \iota(4\alpha+2\rho_x-2) > 4\\[0.6em] \text{(2)}\quad &    \mu^{-(\frac{3}{2} + \rho_z)}\lambda^{\alpha - 1 + \frac{\rho_x}{2}} \ll \lambda^{-(\rho_x+2\alpha)} \iff \iota(6\alpha + 3\rho_x - 2) > 3 + 2\rho_z 
\\[0.6em] \text{(3)} \quad & \mu^{-3}\lambda^{-2+\rho_x} \ll \lambda^{-(\rho_x+2\alpha)} \iff \iota(2\alpha+2\rho_x-2) > 3 \\[0.6em] \text{(4)} \quad & \mu^{-\frac{4+2\rho_z}{3}}\lambda^{-\frac{4}{3}+\frac{\rho_x}{3}} \ll \lambda^{-(\rho_x+2\alpha)} \iff \iota(6\alpha+4\rho_x-4) > 2\rho_z + 4  \end{align*}
These combined parameter restrictions are summarized in Table \ref{table:restrictions combined}.
\begin{table}[h!]
\centering
\caption{Parameter restrictions ensuring approximation errors ($Q+R$) and bias ($B$) are simultaneously dominated by the Variance proxy ($L$). The final column identifies the binding constraints; satisfying these conditions implies that all other listed restrictions also hold.}\label{table:restrictions combined}. 
\begin{tabular}{llc}
\hline
Component & Restriction & Binding? \\\hline
$Q_{\bullet}$ & $2\rho_x < 3 + 2\alpha - \frac{1}{\iota}$& \checkmark \\[0.2em]
$R_{\bullet}$ & $\rho_x + 2\alpha > \frac{1}{\iota}$ &\\[0.4em]
$Q_{\mathrm{res}}^{(1)}$ & $\iota(4\alpha+2\rho_x-2) > 4$ &\\[0.15em]
$Q_{\mathrm{res}}^{(2 )}$ &
$  \iota(4\alpha + 2\rho_x - 1) > 1+2\rho_z$&\\[0.15em]
$Q_{\mathrm{res}}^{(3)}$ & $\iota(4\alpha+3\rho_x-3) > 2(1+\rho_z)$ &\\[0.15em]
$Q_{\mathrm{res}}^{(4)}$ & $\iota(2\alpha+2\rho_x-2) > 3$ &\\[0.4em]
$R_{\mathrm{res}}^{(1)}$ & $\iota\left(4\alpha + 2\rho_x - 2\right) > 4$ &\\[0.15em]
$R_{\mathrm{res}}^{(2)}$ & $\iota(6\alpha + 3\rho_x - 2) > 3 + 2\rho_z $ & \\[0.15em]
$R_{\mathrm{res}}^{(3)}$ & $\iota\left(2\alpha + 2\rho_x - 2\right) > 3$&\checkmark \\[0.15em]
$R_{\mathrm{res}}^{(4)}$ & $\iota(6\alpha+4\rho_x-4) > 2\rho_z + 4$ &\checkmark \\\hline
\end{tabular}
\end{table}

\newpage
\subsection{Main result}
\begin{proof}[Proof of Theorem \ref{thm:valid-inference}]
We appeal to Proposition~1 of \cite{singh2023kernel}. To establish the claimed validity and sharpness, it suffices to show that there exist choices of $\eta$ and $\delta$ satisfying:
\begin{equation*}
\frac{\Delta(n,\lambda,\mu,\eta)+B(\lambda,\mu)}
     {L(\lambda,\mu,1-\chi-2\eta)-\Delta(n,\lambda,\mu,\eta)} \le \delta \le \frac{1}{2},
\end{equation*}
where $\Delta := Q + R$ represents the total approximation error. We set $\eta = 1/n$ and $\delta = 1/\log(n)$. The parameter restrictions summarized in Tables \ref{table:restrictions} and \ref{table:restrictions combined} ensure that the bias $B$ and approximation error $\Delta$ decay strictly faster than the variance lower bound $L$. Specifically, our bounds imply that there exists $\varepsilon_1, \varepsilon_2 > 0$ so that 
\[
\frac{B}{L} = \Ocal(n^{-\varepsilon_1}) \quad \text{and} \quad \frac{\Delta}{L} = \Ocal(n^{-\varepsilon_2}).
\]
Consequently, for sufficiently large $n$, we have $\Delta \le L/2$, allowing us to bound the left-hand side
\[
\frac{\Delta+B}{L-\Delta} \le \frac{\Delta+B}{L/2} = 2\left(\frac{\Delta}{L} + \frac{B}{L}\right) = \Ocal(n^{-(\varepsilon_1\wedge \varepsilon_2)}).
\]
Since a polynomial decay dominates a logarithmic decay, we have $\Ocal(n^{-\varepsilon}) = o(1/\log(n))$. Therefore, for sufficiently large $n$, 
\[
\frac{\Delta+B}{L-\Delta} \le \frac{1}{\log(n)} = \delta.
\]
Finally, since $n$ is large, $\delta = 1/\log(n) \le 1/2$. Thus, the condition on the incremental factor is satisfied. By Proposition~1 of \cite{singh2023kernel}, this choice of $\delta$ guarantees that the confidence sets are valid with tolerance $\tau=\Ocal(\eta) = \Ocal(1/n)$ and satisfy sharpness with slack $2\delta = 2/\log(n)$.
\end{proof}

\end{document}